\input ./style/arxiv-vmsta.cfg
\documentclass[numbers,compress,v1.0.1]{vmsta}

\volume{4}% Updated by VTEXPTS2LaTeX.exe, 28.12.2017 15:44
\issue{4}% Updated by VTEXPTS2LaTeX.exe, 28.12.2017 15:44
\pubyear{2017}
\firstpage{315}% Updated by VTEXPTS2LaTeX.exe, 28.12.2017 15:44
\lastpage{351}% Updated by VTEXPTS2LaTeX.exe, 28.12.2017 15:44
\doi{10.15559/17-VMSTA89}% Updated by VTEXPTS2LaTeX.exe, 20.11.2017
\usepackage{dcolumn}
\usepackage{mathbh}

\startlocaldefs

\allowdisplaybreaks

\endlocaldefs

\newtheorem{thm}{Theorem}
\newtheorem{remark}{Remark}
\newtheorem{lemma}{Lemma}

\theoremstyle{definition}

\theoremstyle{assumption}

\theoremstyle{proposition}

\begin{document}

\begin{frontmatter}

\title{The risk model with stochastic premiums, dependence and a
threshold dividend strategy}

\author{\inits{O.}\fnm{Olena}\snm{Ragulina}}\email{ragulina.olena@gmail.com}
\address{{Taras Shevchenko National University of Kyiv} \\
Department of Probability Theory, Statistics and Actuarial Mathematics\\
64 Volodymyrska Str., 01601 Kyiv, {Ukraine}}

\markboth{O. Ragulina}{The risk model with stochastic premiums,
dependence and a threshold dividend strategy}

\begin{abstract}
The paper deals with a generalization of the risk model with stochastic
premiums where dependence structures between claim sizes and inter-claim
times as well as premium sizes and inter-premium times are modeled by
Farlie--Gumbel--Morgenstern copulas. In addition, dividends are paid to its
shareholders according to a threshold dividend strategy. We derive integral
and integro-differential equations for the Gerber--Shiu function and the
expected discounted dividend payments until ruin. Next, we concentrate on
the detailed investigation of the model in the case of exponentially
distributed claim and premium sizes. In particular, we find explicit
formulas for the ruin probability in the model without either dividend
payments or dependence as well as for the expected discounted dividend
payments in the model without dependence. Finally, numerical illustrations
are presented.
\end{abstract}

\begin{keywords}
\kwd{Risk model with stochastic premiums}
\kwd{Farlie--Gumbel--Morgenstern copula}
\kwd{threshold strategy}
\kwd{Gerber--Shiu function}
\kwd{expected discounted dividend payments}
\kwd{ruin probability}
\kwd{integro-differential equation}
\end{keywords}
\begin{keywords}[2010]
\kwd{91B30}
\kwd{60G55}
\kwd{62P05}
\end{keywords}

\received{4 September 2017}% Updated by
%VTEXPTS2LaTeX.exe, 20.11.2017 09:35
\revised{21 October 2017}% Updated by
%VTEXPTS2LaTeX.exe, 20.11.2017 09:35
\accepted{8 November 2017}% Updated by
%VTEXPTS2LaTeX.exe, 20.11.2017 09:35
\publishedonline{8 December 2017}
\end{frontmatter}

\section{Introduction}
\label{sec:1}

In the actuarial literature, a lot of attention is paid to the
investigation of the ruin measures such as the ruin probability,
the surplus prior to ruin and the deficit at ruin (see, e.g., \cite
{AsAl2010,MiRa2016,RoScScTe1999} and references therein).
A unified approach to the study of these risk measures together by
combining them into one function was proposed by
Gerber and Shiu \cite{GeSh1998}, who introduced the expected discounted
penalty function for the classical risk model.
The so-called Gerber--Shiu function has been investigated further by
many authors (see, e.g.,
\cite{ChTa2003,GeSh2005,LiGa2004_2,Su2005,XiWu2006}) in more
general risk models.
In those risk models, claim sizes and inter-claim times are assumed to
be mutually independent, which simplifies
the investigation of the ruin measures.
Nevertheless, this assumption has been proved to be very restrictive in
some real applications. For instance, in modelling
damages caused by natural catastrophic events, the intensity of the
catastrophe and the time elapsed since the last catastrophe
are expected to be dependent \cite{Bou2003,NiKa2008}.
That is why more and more authors have concentrated on the
investigation of risk models with dependence between claim sizes
and inter-claim times recently.

Albrecher and Boxma \cite{AlBo2004} consider a generalization of the
classical risk model, where the distribution of the time
between two claims depends on the previous claim size (see also \cite
{AlBo2005} for an extension).
Albrecher and Teugels \cite{AlTe2006} apply the random walk approach
and allow the inter-claim time and its subsequent
claim size to be dependent according to an arbitrary structure.
Boudreault, Cossette, Landriault and Marceau \cite{BoCoLaMa2006}
consider a particular dependence structure between
the inter-claim time and the subsequent claim size and derive the
defective renewal equation satisfied by the expected
discounted penalty function.
In \cite{MeZhGu2008}, the authors study the ruin probability in a model
where the time between two claim occurrences
determines the distribution of the next claim size.
The ruin probability in a model with independent but not necessarily
identically distributed claim sizes and inter-claim times
is investigated in \cite{AnBeKiSi2015} (see also references therein).

Cossette, Marceau and Marri \cite{CoMaMa2008,CoMaMa2010} deal with an
extension of the classical compound Poisson risk model
where a dependence structure between the claim size and the inter-claim
time is introduced through a Farlie--Gumbel--Morgenstern
copula and its generalization.
They derive the integro-differential equation and the Laplace transform
of the Gerber--Shiu discounted penalty function and
concentrate on exponentially distributed claim sizes.
Zhang and Yang \cite{ZhYa2011} extend these results to the compound
Poisson risk model perturbed by a Brownian motion.
In \cite{ChVr2014,Gu2013,YoXi2012}, the authors deal with the Sparre
Andersen risk model where the inter-claim times follow
the Erlang distribution and extend results obtained in \cite
{CoMaMa2008,CoMaMa2010}.

In all these papers, the dependence structure between the claim sizes
and the inter-claim times is described by
the Farlie--Gumbel--Morgenstern copula.
This copula is often used in applications to introduce dependence
structures due to its tractability and simplicity.
It allows positive and negative dependence as well as independence.
Nevertheless, the Farlie--Gumbel--Morgenstern copula has been shown to
be somewhat limited since it does not allow the modeling of
high dependencies. Indeed, its dependence parameter is $\theta\in
[-1,1]$, so its Spearman's rho and Kendall's tau are
$\rho_{\theta}= \theta/3\in[-1/3,1/3]$ and $\tau_{\theta}= 2\theta
/9\in[-2/9,2/9]$, respectively (see, e.g.,
\cite{BeJa2017,BePaZa2012,Ne2006} and references therein). This
limited range of dependence restricts the usefulness of
this copula for modeling. Note that the dependence parameter $\theta$
can be easily estimated from real data due to
the simple relations between it and the measures of association $\rho
_{\theta}$ and $\tau_{\theta}$.
For more information on the Farlie--Gumbel--Morgenstern copula, we
refer to, e.g., \cite{DeDhGoKa2005,Ne2006}
(see also \cite{CoMaMa2010} and references therein for applications of
this copula).
Despite the popularity of the Farlie--Gumbel--Morgenstern copula, other
copulas have been used in risk theory, for instance,
an Archimedean copula \cite{AlCoLo2011}, a Gaussian copula \cite
{CzKaBrMi2012} and a Spearman copula \cite{He2014}.

Risk models where an insurance company pays dividends to its
shareholders are of great interest in risk theory.
Dividend strategies for insurance risk models were first proposed by De
Finetti \cite{De1957}, who dealt with a binomial model.
Barrier strategies for the classical risk model and its different
generalizations have been studied in a number of papers
(see, e.g., \cite{ChLi2011,LiWuSo2009,LiGa2004_1,LiPa2006,LiWiDr2003,Wa2015,YaZh2008}).
For optimal dividend problems in insurance risk models, see the
monograph by Schmidli \cite{Sc2008} and references therein.

Cossette, Marceau and Marri \cite{CoMaMa2011,CoMaMa2014} consider the
classical risk process with a constant dividend barrier
and a dependence structure between claim sizes and inter-claim times
introduced through the Farlie--Gumbel--Morgenstern copula.
They analyze the Gerber--Shiu function and the expected discounted
dividend payments and then concentrate on exponentially
distributed claim sizes investigating the impact of the dependence on
ruin quantities.
The same model is studied in \cite{LiLiPe2014}, where, in particular,
the authors show that the solution to
the integro-differential equation for the Gerber--Shiu function is a
linear combination of the Gerber--Shiu function
with no barrier and the solution to the associated homogeneous
integro-differential equation.
For some earlier results in this direction, see also \cite{La2008}.

Shi, Liu and Zhang \cite{ShLiZh2013} consider the compound Poisson risk
model with a threshold dividend strategy and
a dependence structure modeled by the Farlie--Gumbel--Morgenstern copula.
They derive integro-differential equations for the Gerber--Shiu
function and the expected discounted dividend payments
paid until ruin as well as renewal equations for these functions, which
are used to obtain explicit formulas for them.

The present paper deals with a generalization of the risk model with
stochastic premiums introduced and studied in \cite{Boi2003}
(see also \cite{MiRa2016}). In contrast to the classical compound
Poisson risk model, where premiums arrive with constant
intensity and are not random, in this risk model premiums also form a
compound Poisson process, i.e. they arrive at random
times and their sizes are also random (see also \cite{MiRaSt2014,MiRaSt2015}
for a generalization of the classical risk model
where an insurance company gets additional funds whenever a claim arrives).
In \cite{Boi2003}, claim sizes and inter-claim times are assumed to be
mutually independent, and the same assumption
is made concerning premium arrivals. In this paper, we suppose that the
dependence structures between claim sizes and
inter-claim times as well as premium sizes and inter-premium times are
modeled by the
Farlie--Gumbel--Morgenstern copulas, which allows positive and negative
dependence as well as independence.
In addition, we suppose that the insurance company pays dividends to
its shareholders according to a threshold dividend strategy.
To be more precise, this implies that when the surplus is below some
fixed threshold, no dividends are paid,
and when the surplus exceeds or equals the threshold, dividends are
paid continuously at some constant rate.
Our subjects of investigation are the Gerber--Shiu function, a special
case of which is the ruin probability,
and the expected discounted dividend payments until ruin.

The rest of the paper is organized as follows. In Section~\ref{sec:2},
we describe the risk model we deal with.
In Section~\ref{sec:3}, we derive integral and integro-differential
equations for the Gerber--Shiu function.
In Section~\ref{sec:4}, we obtain corresponding equations for the
expected discounted dividend payments until ruin.
Section~\ref{sec:5} deals with exponentially distributed claim and
premium sizes in some special cases of the model.
Namely, we consider the ruin probability in the model without either
dividend payments or dependence,
and the expected discounted dividend payments in the model without dependence.
In these simpler models, we reduce the integral and
integro-differential equations derived in Sections~\ref{sec:2}
and~\ref{sec:3} to linear differential equations and find explicit
solutions to these equations.
Section~\ref{sec:6} provides some numerical illustrations.

\section{Description of the model}
\label{sec:2}

Let $(\varOmega, \mathfrak{F}, \mathbb{P})$ be a probability space
satisfying the usual conditions, and let all the stochastic objects
we use below be defined on it.

In the risk model with stochastic premiums introduced in \cite{Boi2003}
(see also \cite{MiRa2016}), claim sizes form a sequence
$(Y_i)_{i\ge1}$ of non-negative independent and identically
distributed (i.i.d.) random variables (r.v.'s) with cumulative
distribution function (c.d.f.) $F_{Y}(y)=\mathbb P[Y_i\le y]$.
The number of claims on the time interval $[0,t]$ is a Poisson process
$(N_t)_{t\ge0}$ with constant intensity $\lambda>0$.
Next, premium sizes form a sequence $(\bar Y_i)_{i\ge1}$ of
non-negative i.i.d. r.v.'s with c.d.f.
$\bar F_{\bar Y}(y)=\mathbb P[\bar Y_i\le y]$.
The number of premiums on the time interval $[0,t]$ is a Poisson
process $(\bar N_t)_{t\ge0}$ with constant intensity $\bar\lambda>0$.
Thus, the total claims and premiums on $[0,t]$ equal $\sum_{i=1}^{N_t}
Y_i$ and $\sum_{i=1}^{\bar N_t} \bar Y_i$, respectively.
We set $\sum_{i=1}^{0} Y_i =0$ if $N_t =0$, and $\sum_{i=1}^{0} \bar
Y_i =0$ if $\bar N_t =0$.
In what follows, we also assume that the r.v.'s $(Y_i)_{i\ge1}$ have a
probability density function (p.d.f.) $f_{Y}(y)$
and a finite expectation $\mu>0$, and the r.v.'s $(\bar Y_i)_{i\ge1}$
have a probability density function (p.d.f.) $f_{\bar Y}(y)$
and a finite expectation $\bar\mu>0$.

We denote a non-negative initial surplus of the insurance company by $x$.
Let $X_t(x)$ be its surplus at time $t$ provided that the initial
surplus is $x$.
Then the surplus process $ (X_t(x) )_{t\ge0}$ follows the equation
\begin{equation}
\label{eq:1} X_t(x) =x+\sum_{i=1}^{\bar N_t}
\bar Y_i -\sum_{i=1}^{N_t}
Y_i, \quad t\ge0.
\end{equation}

In \cite{Boi2003}, the r.v.'s $(Y_i)_{i\ge1}$ and $(\bar Y_i)_{i\ge
1}$, and the processes $(N_t)_{t\ge0}$ and $(\bar N_t)_{t\ge0}$
are assumed to be mutually independent. In this paper, we suppose that
the claim sizes $(Y_i)_{i\ge1}$ and the inter-claim times
are not independent but with a dependence structure modeled by a
Farlie--Gumbel--Morgenstern copula, and we make the same assumption
concerning premium arrivals.

To be more precise, let $(T_i)_{i\ge1}$ be a sequence of inter-arrival
times of $(N_t)_{t\ge0}$.
In particular, $T_1$ is the time of the first claim.
Thus, $(T_i)_{i\ge1}$ are i.i.d. r.v.'s with p.d.f. $f_{T}(t)=\lambda
e^{-\lambda t}$.
We assume that $(Y_i,T_i)_{i\ge1}$ are i.i.d. random vectors and for
every fixed $i\ge1$, the dependence structure between $Y_i$ and $T_i$
is modeled by a Farlie--Gumbel--Morgenstern copula with parameter
$\theta\in[-1,1]$, i.e.
\[
C_{\theta}^{\mathrm{FGM}}(u_1,u_2)
=u_1 u_2 +\theta u_1 u_2
(1-u_1) (1-u_2), \quad u_1,u_2
\in[0,1]
\]
(see, e.g., \cite{DeDhGoKa2005,Ne2006} for more information on copulas).
In other words, a claim size depends on the time elapsed from the
previous claim.
Therefore, the bivariate c.d.f. of $(Y_i,T_i)_{i\ge1}$ is defined by\vadjust{\goodbreak}
\begin{equation*}
\begin{split} F_{Y,T}(y,t) &=C_{\theta}^{\mathrm{FGM}}
\bigl(F_{Y}(y),F_{T}(t)\bigr)
\\
&=F_{Y}(y)F_{T}(t) +\theta F_{Y}(y)F_{T}(t)
\bigl(1\,{-}\,F_{Y}(y)\bigr) \bigl(1\,{-}\,F_{T}(t)\bigr), \quad y
\ge0, \  t\ge0. \end{split} %
\end{equation*}
The corresponding bivariate p.d.f. of $(Y_i,T_i)_{i\ge1}$ is given by
\begin{equation}
\label{eq:2} %
\begin{split} f_{Y,T}(y,t)
&=f_{Y}(y)f_{T}(t) +\theta f_{Y}(y)f_{T}(t)
\bigl(1-2F_{Y}(y)\bigr) \bigl(1-2F_{T}(t)\bigr)
\\
&=\lambda e^{-\lambda t} f_{Y}(y) +\theta h_{Y}(y)
\bigl(2\lambda e^{-2\lambda t} -\lambda e^{-\lambda t}\bigr), \quad y\ge0, \  t
\ge0, \end{split} %
\end{equation}
where $h_{Y}(y) =f_{Y}(y) (1-2F_{Y}(y))$, $y\ge0$.
Note that the case $\theta=0$ corresponds to the situation where the
claim sizes and the inter-claim times are independent.

Next, let $(\bar T_i)_{i\ge1}$ be a sequence of inter-arrival times of
$(\bar N_t)_{t\ge0}$.
In particular, $\bar T_1$ is the time of the first premium.
Therefore, $(\bar T_i)_{i\ge1}$ are i.i.d. r.v.'s with p.d.f. $f_{\bar
T}(t)=\bar\lambda e^{-\bar\lambda t}$.
We also suppose that $(\bar Y_i, \bar T_i)_{i\ge1}$ are i.i.d. random
vectors and for every fixed $i\ge1$, the dependence structure
between $\bar Y_i$ and $\bar T_i$ is modeled by a
Farlie--Gumbel--Morgenstern copula with parameter $\bar\theta\in[-1,1]$.
So the bivariate p.d.f. of $(\bar Y_i, \bar T_i)_{i\ge1}$ is given by
\begin{equation}
\label{eq:3} f_{\bar Y, \bar T}(y,t) =\bar\lambda e^{-\bar\lambda t}
f_{\bar Y}(y) +\bar\theta h_{\bar Y}(y) \bigl(2\bar\lambda
e^{-2\bar\lambda t} -\bar \lambda e^{-\bar\lambda t}\bigr), \quad y\ge0, \ t\ge0,
\end{equation}
where $h_{\bar Y}(y) =f_{\bar Y}(y) (1-2F_{\bar Y}(y))$, $y\ge0$.
The case $\bar\theta=0$ corresponds to the situation where the premium
sizes and the inter-premium times are independent.
The random vectors $(Y_i,T_i)_{i\ge1}$ and $(\bar Y_i, \bar T_i)_{i\ge
1}$ are assumed to be mutually independent.

From~\eqref{eq:2} and~\eqref{eq:3} we obtain the conditional p.d.f.'s
of the claim and premium sizes:
\begin{equation}
\label{eq:4} f_{Y|T}(y \,|\, t) =\frac{f_{Y,T}(y,t)}{f_{T}(t)}
=f_{Y}(y) +\theta h_{Y}(y) \bigl(2e^{-\lambda t}-1
\bigr), \quad y\ge0, \ t\ge0,
\end{equation}
and
\begin{equation}
\label{eq:5} f_{\bar Y | \bar T}(y \,|\, t) =\frac{f_{\bar Y, \bar T}(y,t)}{f_{\bar T}(t)}
=f_{\bar Y}(y) +\bar\theta h_{\bar Y}(y) \bigl(2e^{-\bar\lambda t}-1
\bigr), \quad y\ge0, \ t\ge0.
\end{equation}

Moreover, we suppose that the insurance company pays dividends to its
shareholders according to the following threshold
dividend strategy. Let $b>0$ be a threshold. When the surplus is below
$b$, no dividends are paid.
When the surplus exceeds or equals $b$, dividends are paid continuously
at a rate $d>0$.
Let $ (X_t^b(x) )_{t\ge0}$ denote the modified surplus process
under this threshold dividend strategy.
Then
\begin{equation}
\label{eq:6} X_t^b(x) =x+\sum
_{i=1}^{\bar N_t} \bar Y_i -\sum
_{i=1}^{N_t} Y_i -d \int
_0^t \mathbh{1}\bigl(X_s^b(x)
\ge b\bigr)\, \mathrm{d}s , \quad t\ge0,
\end{equation}
where $\mathbh{1}(\cdot)$ is the indicator function.

Let $(D_t)_{t\ge0}$ denote the dividend distributing process. For the
threshold dividend strategy described above, we have
\begin{equation*}
\mathrm{d}D_t= %
\begin{cases}
d\,\mathrm{d}t &\text{if} \ X_t^b(x)\ge b, \\
0 &\text{if} \ X_t^b(x)<b.
\end{cases} %
\end{equation*}

Next, let $\tau_b(x)= \inf\{t\ge0\colon X_t^b(x) <0\}$ be the ruin time
for the risk process $ (X_t^b(x) )_{t\ge0}$
defined by~\eqref{eq:6}. In what follows, we omit the dependence on $x$
and write $\tau_b$ instead of $\tau_b(x)$
when no confusion can arise.

For $\delta_0 \ge0$, the Gerber--Shiu function is defined by
\[
m(x,b) =\mathbb E \bigl[ e^{-\delta_0 \tau_b}\, w\bigl(X_{\tau
_b-}^b(x),\big|X_{\tau_b}^b(x)\big|
\bigr)\, \mathbh{1}(\tau_b<\infty) \,|\, X_0^b(x)=x
\bigr], \quad x\ge0,
\]
where $w(\cdot,\cdot)$ is a bounded non-negative measurable function,
$X_{\tau_b-}^b(x)$ is the surplus immediately before ruin
and $|X_{\tau_b}^b(x)|$ is a deficit at ruin. Note that if $w(\cdot
,\cdot) \equiv1$ and $\delta_0=0$, then $m(x,b)$ becomes
the infinite-horizon ruin probability
\[
\psi(x) =\mathbb E\bigl[\mathbh{1}(\tau_b<\infty) \,|\,
X_0^b(x)=x\bigr].
\]

For $\delta>0$, the expected discounted dividend payments until ruin
are defined by
\[
v(x,b) =\mathbb E \Biggl[ \int_0^{\tau_b}
e^{-\delta t}\, \mathrm{d}D_t \, |\, X_0^b(x)=x
\Biggr], \quad x\ge0.
\]

For simplicity of notation, we also write $m(x)$ and $v(x)$ instead of
$m(x,b)$ and $v(x,b)$, respectively, when no confusion can arise.
Moreover, we set
\begin{equation}
\label{eq:7} m(x,b)= %
\begin{cases}
m_1(x) &\text{if} \ x\in[0,b], \\
m_2(x) &\text{if} \ x\in[b,\infty),
\end{cases} %
\end{equation}
and
\begin{equation}
\label{eq:8} v(x,b)= %
\begin{cases}
v_1(x) &\text{if} \ x\in[0,b], \\
v_2(x) &\text{if} \ x\in[b,\infty).
\end{cases} %
\end{equation}
Thus, the functions $m_1(x)$ and $v_1(x)$ are defined on $[0,b]$, the
functions $m_2(x)$ and $v_2(x)$ are defined on $[b,\infty)$
and we have $m_1(b)=m_2(b)$ and $v_1(b)=v_2(b)$.

\section{Equations for the Gerber--Shiu function}
\label{sec:3}

\begin{thm}
\label{thm:1}
Let the surplus process $ (X_t^b(x) )_{t\ge0}$ follow~\eqref
{eq:6} under the above assumptions with $\theta\neq0$ and $\bar\theta
\neq0$.
Moreover, let the p.d.f.'s $f_{Y}(y)$ and $f_{\bar Y}(y)$ have the
derivatives $f'_{Y}(y)$ and $f'_{\bar Y}(y)$ on $\mathbb R_{+}$,
which are continuous and bounded on $\mathbb R_{+}$, and let
$w(u_1,u_2)$ have the second derivatives $w''_{u_1 u_1}(u_1,u_2)$,
$w''_{u_1 u_2}(u_1,u_2)$ and\break $w''_{u_2 u_2}(u_1,u_2)$ on $\mathbb
R_{+}^2$, which are continuous and bounded on $\mathbb R_{+}^2$
as functions of two variables. Then the Gerber--Shiu function $m(x)$
satisfies the equations
\begin{align}
%
%\begin{split}
&(\lambda+ \bar\lambda+
\delta_0)m_1(x)\notag\\
&\quad =\lambda \Biggl( \int
_0^x m_1(x-y)
f_{Y}(y)\, \mathrm{d}y +\int_x^{\infty}
w(x,y-x) f_{Y}(y)\, \mathrm{d}y \Biggr)
\notag\\
&\qquad+\frac{\lambda\theta(\bar\lambda+\delta_0)}{2\lambda+ \bar
\lambda+\delta_0} \Biggl( \int_0^x
m_1(x-y) h_{Y}(y)\, \mathrm{d}y +\int
_x^{\infty} w(x,y-x) h_{Y}(y)\, \mathrm{d}y
\Biggr)
\notag\\
&\qquad+\bar\lambda \Biggl( \int_0^{b-x}
m_1(x+y) f_{\bar Y}(y)\, \mathrm{d}y +\int
_{b-x}^{\infty} m_2(x+y) f_{\bar Y}(y)
\, \mathrm{d}y \Biggr)
\notag\\
&\qquad+\frac{\bar\lambda\bar\theta(\lambda+\delta_0)}{\lambda+ 2\bar
\lambda+\delta_0} \Biggl( \int_0^{b-x}
m_1(x+y) h_{\bar Y}(y)\, \mathrm{d}y +\int
_{b-x}^{\infty} m_2(x+y) h_{\bar Y}(y)
\, \mathrm{d}y \Biggr),
\notag\\
&\qquad\qquad\qquad\qquad\qquad\qquad\qquad\qquad\qquad\qquad\qquad \qquad\qquad
x\in[0,b],\label{eq:9} %\end{split} %
\end{align}
and
\begin{equation}
\label{eq:10} %
\begin{split} &d^3 m'''_2(x)
+(4\lambda+ 4\bar\lambda+3\delta_0) d^2
m''_2(x)
\\
&\qquad+ \bigl( (\lambda\,{+}\, 2\bar\lambda\,{+}\,\delta_0) (3\lambda\,{+}\, 2\bar
\lambda \,{+}\,2\delta_0) \!+\!(2\lambda\,{+}\, \bar\lambda\,{+}\,
\delta_0) (\lambda\,{+}\, \bar\lambda\,{+}\,\delta_0) \bigr) d
m'_2(x)
\\
&\qquad+(\lambda+ 2\bar\lambda+\delta_0) (2\lambda+ \bar\lambda+
\delta _0) (\lambda+ \bar\lambda+\delta_0)
m_2(x)
\\
&\quad =(\lambda+ 2\bar\lambda+\delta_0) (2\lambda+ \bar\lambda+
\delta _0)\beta_1(x) +(3\lambda+ 3\bar\lambda+2
\delta_0)d \beta'_1(x)
\\
&\qquad+d^2 \beta''_1(x) -2(
\lambda+ 2\bar\lambda+\delta_0) \beta_2(x) -2d
\beta'_2(x)
\\
&\qquad+2\bar\lambda^2 \bar\theta(\bar\lambda-\lambda) \int
_0^{\infty} m_2(x+y) h_{\bar Y}(y)
\, \mathrm{d}y, \quad x\in[b,\infty), \end{split} %
\end{equation}
where
\begin{align*}
%
%\begin{split}
\beta_1(x) &=\lambda\int
_{x-b}^x m_1(x-y) \bigl(
f_{Y}(y) +\theta h_{Y}(y) \bigr)\, \mathrm{d}y
\\[-2pt]
&\quad+\lambda\int_0^{x-b} m_2(x-y)
\bigl( f_{Y}(y) +\theta h_{Y}(y) \bigr)\, \mathrm{d}y
\\[-2pt]
&\quad+\lambda\int_x^{\infty} w(x,y-x) \bigl(
f_{Y}(y) +\theta h_{Y}(y) \bigr)\, \mathrm{d}y
\\[-2pt]
&\quad+\bar\lambda\int_0^{\infty}
m_2(x+y) \bigl( f_{\bar Y}(y) +\bar \theta
h_{\bar Y}(y) \bigr)\, \mathrm{d}y, \quad x\in[b,\infty),
%\end{split}
%
\end{align*}
and
\begin{equation*}
\begin{split} \beta_2(x) &=\lambda^2 \theta
\Biggl( \int_{x-b}^x m_1(x-y)
h_{Y}(y)\, \mathrm{d}y +\int_0^{x-b}
m_2(x-y) h_{Y}(y)\, \mathrm{d}y
\\[-2pt]
&\quad+\!\!\int_x^{\infty}\! w(x,y-x)
h_{Y}(y)\, \mathrm{d}y \Biggr) \!+\!\bar\lambda^2 \bar
\theta\int_0^{\infty}\! m_2(x\,{+}\,y)
h_{\bar
Y}(y)\, \mathrm{d}y, \quad x\in[b,\infty). \end{split} %
\end{equation*}
\end{thm}

\begin{proof}
It is easily seen that the time of the first jump of $
(X_t^b(x) )_{t\ge0}$, $T_1 \wedge\bar T_1$, is exponentially
distributed with mean $1/(\lambda+\bar\lambda)$. Furthermore, $\mathbb
P[T_1 \wedge\bar T_1 =T_1] =\lambda/(\lambda+\bar\lambda)$
and $\mathbb P[T_1 \wedge\bar T_1 =\bar T_1] =\bar\lambda/(\lambda
+\bar\lambda)$.

We first deal with the case $x\in[0,b]$. Considering the time and the
size of the first jump of $ (X_t^b(x) )_{t\ge0}$
and applying the law of total probability we obtain
\begin{align}
%
%\begin{split}
m(x)&=\int_0^{\infty}
e^{-(\lambda+\bar\lambda)t} \Biggl( \lambda\int_0^x
e^{-\delta_0 t} m(x-y) f_{Y|T}(y \,|\, t)\, \mathrm{d}y
\notag\\[-2pt]
&\quad+\lambda\int_x^{\infty} e^{-\delta_0 t}
w(x,y-x) f_{Y|T}(y \, |\, t)\, \mathrm{d}y
\label{eq:11}\\
&\quad+\bar\lambda\int_0^{\infty} e^{-\delta_0 t}
m(x+y) f_{\bar Y |
\bar T}(y \,|\, t)\, \mathrm{d}y \Biggr) \mathrm{d}t, \quad x
\in[0,b].\notag
%\end{split} %
\end{align}

Substituting~\eqref{eq:4} and~\eqref{eq:5} into~\eqref{eq:11} and
taking into account~\eqref{eq:7} give
\begin{equation}
\label{eq:12} %
\begin{split} m_1(x)&=\int
_0^{\infty} e^{-(\lambda+\bar\lambda+\delta_0)t} \Biggl( \lambda\int
_0^x m_1(x-y) \bigl(
f_{Y}(y) +\theta h_{Y}(y) \bigl(2e^{-\lambda t}-1\bigr)
\bigr)\, \mathrm {d}y
\\[-2pt]
&\quad+\lambda\int_x^{\infty} w(x,y-x) \bigl(
f_{Y}(y) +\theta h_{Y}(y) \bigl(2e^{-\lambda t}-1\bigr)
\bigr)\, \mathrm{d}y
\\[-2pt]
&\quad+\bar\lambda\int_0^{b-x}
m_1(x+y) \bigl( f_{\bar Y}(y) +\bar \theta
h_{\bar Y}(y) \bigl(2e^{-\bar\lambda t}-1\bigr) \bigr)\, \mathrm{d}y
\\[-2pt]
&\quad+\bar\lambda\int_{b-x}^{\infty}
m_2(x+y) \bigl( f_{\bar Y}(y) +\bar\theta h_{\bar Y}(y)
\bigl(2e^{-\bar\lambda t}-1\bigr) \bigr)\, \mathrm{d}y \Biggr) \mathrm{d}t, \quad x
\in[0,b]. \end{split} %
\end{equation}

Separating the integrals on the right-hand side of~\eqref{eq:12} into
integrals w.r.t. either $t$ or $y$ yields
\begin{equation}
\label{eq:13} %
\begin{split} m_1(x)&=\lambda\int
_0^{\infty} e^{-(\lambda+\bar\lambda+\delta
_0)t}\, \mathrm{d}t\, \Biggl(
\int_0^x m_1(x-y)
f_{Y}(y)\, \mathrm{d}y
\\[-2pt]
&\quad\qquad\qquad+\int_x^{\infty} w(x,y-x)
f_{Y}(y)\, \mathrm{d}y \Biggr)
\\[-2pt]
&\quad+\lambda\theta\int_0^{\infty}
e^{-(\lambda+\bar\lambda+\delta
_0)t}\, \bigl(2e^{-\lambda t}-1\bigr)\, \mathrm{d}t\, \Biggl( \int
_0^x m_1(x-y)
h_{Y}(y)\, \mathrm{d}y
\\[-2pt]
&\quad\qquad\qquad+\int_x^{\infty} w(x,y-x)
h_{Y}(y)\, \mathrm{d}y \Biggr)
\\[-2pt]
&\quad+\bar\lambda\int_0^{\infty} e^{-(\lambda+\bar\lambda+\delta
_0)t}
\, \mathrm{d}t\, \Biggl( \int_0^{b-x}
m_1(x+y) f_{\bar Y}(y)\, \mathrm{d}y
\\[-2pt]
&\quad\qquad\qquad+\int_{b-x}^{\infty}
m_2(x+y) f_{\bar Y}(y)\, \mathrm {d}y \Biggr)
\\[-2pt]
&\quad+\bar\lambda\bar\theta\int_0^{\infty}
e^{-(\lambda+\bar
\lambda+\delta_0)t}\, \bigl(2e^{-\bar\lambda t}-1\bigr)\, \mathrm{d}t\, \Biggl( \int
_0^{b-x}\! m_1(x+y)
h_{\bar Y}(y)\, \mathrm{d}y
\\[-2pt]
&\quad\qquad\qquad+\int_{b-x}^{\infty}
m_2(x+y) h_{\bar Y}(y)\, \mathrm {d}y \Biggr), \quad x\in[0,b].
\end{split} %
\end{equation}

Taking the integrals w.r.t. $t$ on the right-hand side of~\eqref{eq:13}
we get
\begin{align*}
 m_1(x)&=\frac{\lambda}{\lambda+\bar\lambda+\delta_0}\, \Biggl( \int
_0^x m_1(x-y)
f_{Y}(y)\, \mathrm{d}y +\int_x^{\infty}
w(x,y-x) f_{Y}(y)\, \mathrm{d}y \Biggr)
\\
&\quad+\frac{\lambda\theta(\bar\lambda+\delta_0)}{(2\lambda+\bar
\lambda+\delta_0) (\lambda+\bar\lambda+\delta_0)}\, \Biggl( \int_0^x
m_1(x-y) h_{Y}(y)\, \mathrm{d}y
\\
&\quad\qquad\qquad+\int_x^{\infty} w(x,y-x)
h_{Y}(y)\, \mathrm{d}y \Biggr)
\\
&\quad+\frac{\bar\lambda}{\lambda+\bar\lambda+\delta_0}\, \Biggl( \int_0^{b-x}
m_1(x+y) f_{\bar Y}(y)\, \mathrm{d}y +\int
_{b-x}^{\infty} m_2(x+y) f_{\bar Y}(y)
\, \mathrm{d}y \Biggr)
\\
&\quad+\frac{\bar\lambda\bar\theta(\lambda+\delta_0)}{(\lambda
+2\bar\lambda+\delta_0) (\lambda+\bar\lambda+\delta_0)}\, \Biggl( \int_0^{b-x}
\! m_1(x+y) h_{\bar Y}(y)\, \mathrm{d}y
\\
&\quad\qquad\qquad+\int_{b-x}^{\infty}
m_2(x+y) h_{\bar Y}(y)\, \mathrm {d}y \Biggr), \quad x\in[0,b],
%\end{split} %
\end{align*}
which yields~\eqref{eq:9}.

Let now $x\in[b,\infty)$. Considering the time and the size of the
first jump of $ (X_t^b(x) )_{t\ge0}$
and applying the law of total probability we have
\begin{equation}
\label{eq:14} %
\begin{split} m(x)&=\int_0^{(x-b)/d}
e^{-(\lambda+\bar\lambda)t} \Biggl( \lambda\int_0^{x-dt}
e^{-\delta_0 t} m(x-dt-y) f_{Y|T}(y \,|\, t)\, \mathrm{d}y
\\
&\quad\qquad+\lambda\int_{x-dt}^{\infty} e^{-\delta_0 t}
w(x-dt,y-x-dt) f_{Y|T}(y \,|\, t)\, \mathrm{d}y
\\
&\quad\qquad+\bar\lambda\int_0^{\infty}
e^{-\delta_0 t} m(x-dt+y) f_{\bar Y | \bar T}(y \,|\, t)\, \mathrm{d}y \Biggr)
\mathrm{d}t
\\
&\quad+\int_{(x-b)/d}^{\infty} e^{-(\lambda+\bar\lambda)t} \Biggl(
\lambda\int_0^b e^{-\delta_0 t} m(b-y)
f_{Y|T}(y \,|\, t)\, \mathrm {d}y
\\
&\quad\qquad+\lambda\int_b^{\infty} e^{-\delta_0 t}
w(b,y-b) f_{Y|T}(y \,|\, t)\, \mathrm{d}y
\\
&\quad\qquad+\bar\lambda\int_0^{\infty}
e^{-\delta_0 t} m(b+y) f_{\bar Y | \bar T}(y \,|\, t)\, \mathrm{d}y \Biggr)
\mathrm{d}t, \quad x\in[b,\infty). \end{split} %
\end{equation}

Substituting~\eqref{eq:4} and~\eqref{eq:5} into~\eqref{eq:14} and
taking into account~\eqref{eq:7} give
\begin{equation}
\label{eq:15} m_2(x) =I_{1,2,3}(x) +I_{4,5,6}(x),
\quad x\in[b,\infty),
\end{equation}
where
\begin{align*}
%\begin{split}
I_{1,2,3}(x)&=\int_0^{(x-b)/d}
e^{-(\lambda+\bar\lambda+\delta_0)t}
\\
&\quad\times \Biggl( \lambda\int_0^{x-dt-b}
m_2(x-dt-y) \bigl( f_{Y}(y) +\theta h_{Y}(y)
\bigl(2e^{-\lambda t}-1\bigr) \bigr)\, \mathrm{d}y
\\
&\quad\;\; +\lambda\int_{x-dt-b}^{x-dt}
m_1(x-dt-y) \bigl( f_{Y}(y) +\theta
h_{Y}(y) \bigl(2e^{-\lambda t}-1\bigr) \bigr)\, \mathrm{d}y
\\
&\quad\;\; +\lambda\int_{x-dt}^{\infty} w(x-dt,y-x+dt)
\bigl( f_{Y}(y) +\theta h_{Y}(y) \bigl(2e^{-\lambda t}-1
\bigr) \bigr)\, \mathrm{d}y
\\
&\quad\;\; +\bar\lambda\int_{0}^{\infty}
m_2(x-dt+y) \bigl( f_{\bar
Y}(y) +\bar\theta
h_{\bar Y}(y) \bigl(2e^{-\bar\lambda t}-1\bigr) \bigr)\, \mathrm{d}y \Biggr)
\mathrm{d}t %\end{split} %
\end{align*}
and
\begin{align*}
%\begin{split}
I_{4,5,6}(x)&=\int_{(x-b)/d}^{\infty}
e^{-(\lambda+\bar\lambda+\delta
_0)t}
\\
&\quad\times \Biggl( \lambda\int_0^{b}
m_1(b-y) \bigl( f_{Y}(y) +\theta h_{Y}(y)
\bigl(2e^{-\lambda t}-1\bigr) \bigr)\, \mathrm{d}y
\\
&\quad\;\; +\lambda\int_{b}^{\infty} w(b,y-b) \bigl(
f_{Y}(y) +\theta h_{Y}(y) \bigl(2e^{-\lambda t}-1\bigr)
\bigr)\, \mathrm{d}y
\\
&\quad\;\; +\bar\lambda\int_{0}^{\infty}
m_2(b+y) \bigl( f_{\bar Y}(y) +\bar\theta h_{\bar Y}(y)
\bigl(2e^{-\bar\lambda t}-1\bigr) \bigr)\, \mathrm{d}y \Biggr) \mathrm{d}t.
%\end{split}
\end{align*}

Changing the variable $x-dt=s$ in the outer integral of the expression
for $I_{1,2,3}(x)$ yields
\begin{equation}
\label{eq:16} %
\begin{split} I_{1,2,3}(x)&=\frac{1}{d}
\int_b^x e^{-(\lambda+\bar\lambda+\delta
_0)(x-s)/d}
\\
&\quad\times \Biggl( \lambda\int_0^{s-b}
m_2(s-y) \bigl( f_{Y}(y) +\theta h_{Y}(y)
\bigl(2e^{-\lambda(x-s)/d}-1\bigr) \bigr)\, \mathrm{d}y
\\
&\quad +\lambda\int_{s-b}^{s}
m_1(s-y) \bigl( f_{Y}(y) +\theta h_{Y}(y)
\bigl(2e^{-\lambda(x-s)/d}-1\bigr) \bigr)\, \mathrm{d}y
\\
&\quad +\lambda\int_{s}^{\infty} w(s,y-s) \bigl(
f_{Y}(y) +\theta h_{Y}(y) \bigl(2e^{-\lambda(x-s)/d}-1\bigr)
\bigr)\, \mathrm{d}y
\\
&\quad +\bar\lambda\int_{0}^{\infty}
m_2(s\,{+}\,y) \bigl( f_{\bar Y}(y) \,{+}\,\bar\theta h_{\bar Y}(y)
\bigl(2e^{-\bar\lambda(x-s)/d}\,{-}\,1\bigr) \bigr)\, \mathrm{d}y \Biggr) \mathrm{d}s
\\
&=\frac{1}{d}\, e^{-(\lambda+\bar\lambda+\delta_0)x/d}\, I_1(x) +
\frac
{2}{d}\, e^{-(2\lambda+\bar\lambda+\delta_0)x/d}\, I_2(x)
\\
&\quad +\frac{2}{d}\, e^{-(\lambda+2\bar\lambda+\delta_0)x/d}\, I_3(x), \quad x
\in[b,\infty), \end{split} %
\end{equation}
where
\begin{align*}
%\begin{split}
I_1(x)&=\int_b^x
e^{(\lambda+\bar\lambda+\delta_0)s/d} \Biggl( \lambda\int_0^{s-b}
m_2(s-y) \bigl( f_{Y}(y) -\theta h_{Y}(y)
\bigr)\, \mathrm{d}y
\\
&\quad+\lambda\int_{s-b}^{s} m_1(s-y)
\bigl( f_{Y}(y) -\theta h_{Y}(y) \bigr)\, \mathrm{d}y
\\
&\quad+\lambda\int_{s}^{\infty} w(s,y-s) \bigl(
f_{Y}(y) -\theta h_{Y}(y) \bigr)\, \mathrm{d}y
\\
&\quad+\bar\lambda\int_{0}^{\infty}
m_2(s+y) \bigl( f_{\bar Y}(y) -\bar\theta h_{\bar Y}(y)
\bigr)\, \mathrm{d}y \Biggr) \mathrm{d}s,\\ %\end{split} %\end{equation*}
%\begin{equation*}
%\begin{split}
I_2(x)&=\lambda\theta\int
_b^x e^{(2\lambda+\bar\lambda+\delta_0)s/d} \Biggl( \int
_0^{s-b} m_2(s-y) h_{Y}(y)
\, \mathrm{d}y
\\
&\quad+\int_{s-b}^{s} m_1(s-y)
h_{Y}(y)\, \mathrm{d}y +\int_{s}^{\infty}
w(s,y-s) h_{Y}(y)\, \mathrm{d}y \Biggr) \mathrm{d}s
%\end{split}
%
\end{align*}
and
\begin{align*}
%
%\begin{split}
I_3(x)&=\bar\lambda\bar\theta\int
_b^x e^{(\lambda+2\bar\lambda
+\delta_0)s/d} \Biggl( \int
_0^{\infty} m_2(s+y) h_{\bar Y}(y)
\, \mathrm{d}y \Biggr) \mathrm{d}s.
%\end{split} %
\end{align*}

Separating the integrals in the expression for $I_{4,5,6}(x)$ into
integrals w.r.t. either $t$ or $y$ we get
\begin{equation}
\label{eq:17} %
\begin{split} I_{4,5,6}(x)&=\lambda\int
_{(x-b)/d}^{\infty} e^{-(\lambda+\bar\lambda
+\delta_0)t}\, \mathrm{d}t\, \Biggl(
\int_0^b m_1(b-y)
f_{Y}(y)\, \mathrm{d}y
\\
&\quad\qquad\qquad+\int_b^{\infty} w(b,y-b)
f_{Y}(y)\, \mathrm{d}y \Biggr)
\\
&\quad+\lambda\theta\int_{(x-b)/d}^{\infty} e^{-(\lambda+\bar\lambda
+\delta_0)t}
\, \bigl(2e^{-\lambda t}-1\bigr)\, \mathrm{d}t\, \Biggl( \int
_0^b m_1(b-y)
h_{Y}(y)\, \mathrm{d}y
\\
&\quad\qquad\qquad+\int_b^{\infty} w(b,y-b)
h_{Y}(y)\, \mathrm{d}y \Biggr)
\\
&\quad+\bar\lambda\int_{(x-b)/d}^{\infty} e^{-(\lambda+\bar\lambda
+\delta_0)t}\,
\mathrm{d}t\; \int_0^{\infty} m_2(x+y)
f_{\bar Y}(y)\, \mathrm{d}y
\\
&\quad+\bar\lambda\bar\theta\int_{(x-b)/d}^{\infty}
e^{-(\lambda
+\bar\lambda+\delta_0)t}\, \bigl(2e^{-\bar\lambda t}-1\bigr)\, \mathrm{d}t\; \int
_0^{\infty} m_2(b+y) h_{\bar Y}(y)
\, \mathrm{d}y. \end{split} %
\end{equation}

Taking the integrals w.r.t. $t$ on the right-hand side of~\eqref{eq:17}
we obtain
\begin{equation}
\label{eq:18} I_{4,5,6}(x) =I_4(x) +I_5(x)
+I_6(x), \quad x\in[b,\infty),
\end{equation}
where
\begin{align*}
%\begin{split}
I_4(x)&=\frac{e^{-(\lambda+\bar\lambda+\delta_0)(x-b)/d}}{\lambda
+\bar\lambda+\delta_0}
\\
&\quad\times \Biggl( \lambda\int_0^b
m_1(b-y) f_{Y}(y)\, \mathrm{d}y +\lambda\int
_b^{\infty} w(b,y-b) f_{Y}(y)\, \mathrm{d}y
\\
&\quad -\lambda\theta\int_0^b
m_1(b-y) h_{Y}(y)\, \mathrm{d}y -\lambda\theta\int
_b^{\infty} w(b,y-b) h_{Y}(y)\, \mathrm{d}y
\\
&\quad +\bar\lambda\int_0^{\infty}
m_2(b+y) f_{\bar Y}(y)\, \mathrm{d}y -\bar\lambda\bar\theta
\int_0^{\infty} m_2(b+y)
h_{\bar Y}(y)\, \mathrm {d}y \Biggr), \\%\end{split}\end{equation*}
%\begin{equation*}
%\begin{split}
I_5(x)&=\frac{2\lambda\theta e^{-(2\lambda+\bar\lambda+\delta
_0)(x-b)/d}}{2\lambda+\bar\lambda+\delta_0}
\\
&\quad\times \Biggl( \int_0^b
m_1(b-y) h_{Y}(y)\, \mathrm{d}y +\int
_b^{\infty} w(b,y-b) h_{Y}(y)\, \mathrm{d}y
\Biggr) %\end{split} %
\end{align*}
and
\[
I_6(x)=\frac{2 \bar\lambda\bar\theta e^{-(\lambda+2\bar\lambda+\delta
_0)(x-b)/d}}{\lambda+2\bar\lambda+\delta_0}\, \int_0^{\infty}
m_2(b+y) h_{\bar Y}(y)\, \mathrm{d}y.
\]

Thus, substituting~\eqref{eq:16} and~\eqref{eq:18} into~\eqref{eq:15}
we have
\begin{equation}
\label{eq:19} %
\begin{split} m_2(x)&=\frac{1}{d}\,
e^{-(\lambda+\bar\lambda+\delta_0)x/d}\, I_1(x) +\frac{2}{d}\, e^{-(2\lambda+\bar\lambda+\delta_0)x/d}\,
I_2(x)
\\
&\quad+\frac{2}{d}\, e^{-(\lambda+2\bar\lambda+\delta_0)x/d}\, I_3(x)
+I_4(x) +I_5(x) +I_6(x), \quad x\in[b,
\infty). \end{split} %
\end{equation}

It is easily seen from~\eqref{eq:9} that $m_1(x)$ is continuous on
$[0,b]$, and from~\eqref{eq:19} we conclude that $m_2(x)$
is continuous on $[b,\infty)$. Indeed, the right-hand sides of~\eqref
{eq:9} and~\eqref{eq:19} are continuous on $[0,b]$ and
$[b,\infty)$, respectively, and so are the left-hand sides.
Therefore, from~\eqref{eq:19} we deduce that $m_2(x)$ is differentiable
on $[b,\infty)$.
Differentiating~\eqref{eq:19} gives
\begin{equation}
\label{eq:20} %
\begin{split} m'_2(x)&=-
\frac{\lambda+\bar\lambda+\delta_0}{d^2}\, e^{-(\lambda
+\bar\lambda+\delta_0)x/d}\, I_1(x)
\\
&\quad-\frac{2(2\lambda+\bar\lambda+\delta_0)}{d^2}\, e^{-(2\lambda
+\bar\lambda+\delta_0)x/d}\, I_2(x)
\\
&\quad-\frac{2(\lambda+2\bar\lambda+\delta_0)}{d^2}\, e^{-(\lambda
+2\bar\lambda+\delta_0)x/d}\, I_3(x)
\\
&\quad-\frac{\lambda+\bar\lambda+\delta_0}{d}\, I_4(x) -\frac
{2\lambda+\bar\lambda+\delta_0}{d}\,
I_5(x)
\\
&\quad-\frac{\lambda+2\bar\lambda+\delta_0}{d}\, I_6(x) +\frac
{1}{d}\,
\beta_1(x), \quad x\in[b,\infty), \end{split} %
\end{equation}
where the function $\beta_1(x)$ is defined in the assertion of the theorem.

Multiplying~\eqref{eq:19} by $(\lambda+\bar\lambda+\delta_0)/d$ and
adding~\eqref{eq:20} we get
\begin{equation}
\label{eq:21} %
\begin{split} &dm'_2(x) +(
\lambda+\bar\lambda+\delta_0) m_2(x) =-
\frac{2\lambda
}{d}\, e^{-(2\lambda+\bar\lambda+\delta_0)x/d}\, I_2(x)
\\
&\quad-\frac{2 \bar\lambda}{d}\, e^{-(\lambda+2\bar\lambda+\delta
_0)x/d}\, I_3(x) -\lambda
I_5(x) -\bar\lambda I_6(x) +\beta_1(x),
\quad x\in[b,\infty). \end{split} %
\end{equation}

Since $m_1(x)$ and $m_2(x)$ are continuous and bounded on $[0,b]$ and
$[b,\infty)$, respectively, and $w(u_1,u_2)$ is
continuous and bounded on $\mathbb R_{+}^2$ as a function of two
variables, from~\eqref{eq:21} we conclude that so is
$m'_2(x)$ on $[b,\infty)$.
Taking into account that $f'_{Y}(y)$ and $f'_{\bar Y}(y)$ are
continuous and bounded on $\mathbb R_{+}$, and
$w'_{u_1}(u_1,u_2)$ and $w'_{u_2}(u_1,u_2)$ are continuous and bounded
on $\mathbb R_{+}^2$, from~\eqref{eq:9}
we conclude that so is $m'_1(x)$ on $[0,b]$.
Hence, $\beta_1(x)$ is differentiable on $[b,\infty)$. From this
and~\eqref{eq:21} it follows that $m_2(x)$ is
twice differentiable on $[b,\infty)$. Differentiating~\eqref{eq:21} gives
\begin{equation}
\label{eq:22} %
\begin{split} &dm''_2(x)
+(\lambda+\bar\lambda+\delta_0) m'_2(x) =
\frac{2\lambda
(2\lambda+\bar\lambda+\delta_0)}{d^2}\, e^{-(2\lambda+\bar\lambda+\delta_0)x/d}\, I_2(x)
\\
&\quad+\frac{2 \bar\lambda(\lambda+2\bar\lambda+\delta_0)}{d^2}\, e^{-(\lambda+2\bar\lambda+\delta_0)x/d}\, I_3(x) +
\frac{\lambda(2\lambda+\bar\lambda+\delta_0)}{d}\, I_5(x)
\\
&\quad+\frac{\bar\lambda(\lambda+2\bar\lambda+\delta_0)}{d}\, I_6(x) +\beta'_1(x)
-\frac{2}{d}\, \beta_2(x), \quad x\in[b,\infty), \end{split}
\end{equation}
where
\begin{align*}
%\begin{split}
\beta'_1(x) &=\lambda\int
_{x-b}^x m'_1(x-y) \bigl(
f_{Y}(y) +\theta h_{Y}(y) \bigr)\, \mathrm{d}y
\\
&\quad+\lambda\int_0^{x-b} m'_2(x-y)
\bigl( f_{Y}(y) +\theta h_{Y}(y) \bigr)\, \mathrm{d}y
\\
&\quad+\lambda\int_x^{\infty} \bigl(
w'_{u_1}(x,y-x) -w'_{u_2}(x,y-x)
\bigr) \bigl( f_{Y}(y) +\theta h_{Y}(y) \bigr)\,
\mathrm{d}y
\\
&\quad+\bar\lambda\int_0^{\infty}
m'_2(x+y) \bigl( f_{\bar Y}(y) +\bar \theta
h_{\bar Y}(y) \bigr)\, \mathrm{d}y
\\
&\quad+\lambda \bigl( m_1(0) -w(x,0) \bigr) \bigl(
f_{Y}(x) +\theta h_{Y}(x) \bigr), \quad x\in[b,\infty),
%\end{split}
\end{align*}
which is continuous and bounded on $[b,\infty)$, and the function $\beta
_2(x)$ is defined in the assertion of the theorem.
Here $w'_{u_1}(\cdot,\cdot)$ and $w'_{u_2}(\cdot,\cdot)$ stand for the
partial derivatives of $w(u_1,u_2)$ w.r.t.
the first and the second variables, respectively.

Multiplying~\eqref{eq:21} by $(2\lambda+\bar\lambda+\delta_0)/d$ and
adding~\eqref{eq:22} we obtain
\begin{equation}
\label{eq:23} %
\begin{split} &d^2 m''_2(x)
+(3\lambda+2\bar\lambda+2\delta_0)d m'_2(x)
+(2\lambda +\bar\lambda+\delta_0) (\lambda+\bar\lambda+
\delta_0) m_2(x)
\\
&\quad=\frac{2 \bar\lambda(\bar\lambda-\lambda)}{d}\, e^{-(\lambda
+2\bar\lambda+\delta_0)x/d}\, I_3(x) +\bar\lambda(
\bar\lambda-\lambda) I_6(x)
\\
&\qquad+(2\lambda+\bar\lambda+\delta_0) \beta_1(x)
+d\beta'_1(x) -2\beta_2(x), \quad x\in[b,
\infty). \end{split} %
\end{equation}

It is easily seen from~\eqref{eq:23} that $m''_2(x)$ is continuous and
bounded on $[b,\infty)$.
Taking into account that $f'_{Y}(y)$ and $f'_{\bar Y}(y)$ are
continuous and bounded on $\mathbb R_{+}$, and
$w''_{u_1 u_1}(u_1,u_2)$, $w''_{u_1 u_2}(u_1,u_2)$ and $w''_{u_2
u_2}(u_1,u_2)$ are continuous and bounded on $\mathbb R_{+}^2$,
from~\eqref{eq:9} we conclude that so is $m''_1(x)$ on $[0,b]$.
Hence, $\beta_1(x)$ is twice differentiable on $[b,\infty)$.
Moreover, applying similar arguments shows that $\beta_2(x)$ is
differentiable on $[b,\infty)$.
From this and~\eqref{eq:23} it follows that $m_2(x)$ has the third
derivative on $[b,\infty)$.
Differentiating~\eqref{eq:23} gives
\begin{equation}
\label{eq:24} %
\begin{split} &d^2 m'''_2(x)
+(3\lambda+2\bar\lambda+2\delta_0)d m''_2(x)
+(2\lambda+\bar\lambda+\delta_0) (\lambda+\bar\lambda+
\delta_0) m'_2(x)
\\
&\quad=-\frac{2 \bar\lambda(\bar\lambda-\lambda) (\lambda+2\bar
\lambda+\delta_0)}{d^2}\, e^{-(\lambda+2\bar\lambda+\delta_0)x/d}\, I_3(x)
\\
&\qquad-\frac{\bar\lambda(\bar\lambda-\lambda) (\lambda+2\bar
\lambda+\delta_0)}{d}\, I_6(x) +(2\lambda+\bar\lambda+
\delta_0) \beta'_1(x) +d
\beta''_1(x)
\\
&\qquad-2\beta'_2(x) +\frac{2 \bar\lambda^2 \bar\theta(\bar\lambda
-\lambda)}{d}\, \int
_0^{\infty} m_2(x+y) h_{\bar Y}(y)
\, \mathrm{d}y, \quad x\in[b,\infty). \end{split} %
\end{equation}

Multiplying~\eqref{eq:23} by $(\lambda+2\bar\lambda+\delta_0)/d$ and
adding~\eqref{eq:24} yield~\eqref{eq:10},
which completes the proof.
\end{proof}

\begin{remark}
\label{rem:1}
To solve equations~\eqref{eq:9} and~\eqref{eq:10}, we need some
boundary conditions. The first one is $m_1(b)=m_2(b)$.
Next, using standard considerations (see, e.g., \cite{MiRa2016,MiRaSt2015,RoScScTe1999}) we can show that
$\lim_{x\to\infty} m_2(x) =0$ provided that the net profit condition holds.
Finally, we can substitute $x=b$ into the intermediate equations (e.g.,
equation~\eqref{eq:21}) to get additional boundary conditions
involving derivatives of $m_2(x)$. Furthermore, equations~\eqref{eq:9}
and~\eqref{eq:10} are not solvable analytically in the general case,
so we can use, for instance, numerical methods. Nevertheless, we can
give explicit expressions for $m(x)$ in some particular cases
(see Section~\ref{sec:5}).
The uniqueness of the required solutions to these equations should be
justified in each case.
\end{remark}

\begin{remark}
\label{rem:2}
The corresponding model without dividend payments is obtained by $b\to
\infty$. In this case, the Gerber--Shiu function $m(x)$ satisfies
the integral equation
\begin{equation}
\label{eq:25} %
\begin{split}
&(\lambda+ \bar\lambda+
\delta_0)m(x)\\
&\quad =\lambda \Biggl( \int_0^x
m(u) f_{Y}(x-u)\, \mathrm{d}u +\int_0^{\infty
}
w(x,u) f_{Y}(x+u)\, \mathrm{d}u \Biggr)
\\
&\qquad+\frac{\lambda\theta(\bar\lambda+\delta_0)}{2\lambda\,{+}\, \bar
\lambda\,{+}\,\delta_0} \Biggl( \int_0^x
m(u) h_{Y}(x\,{-}\,u)\, \mathrm{d}u \,{+}\int_0^{\infty}
w(x,u) h_{Y}(x\,{+}\,u)\, \mathrm{d}u \Biggr)
\\
&\qquad+\bar\lambda\int_x^{\infty} m(u)
f_{\bar Y}(u-x)\, \mathrm{d}u +\frac{\bar\lambda\bar\theta(\lambda+\delta_0)}{\lambda+ 2\bar\lambda
+\delta_0}\, \int
_x^{\infty} m(u) h_{\bar Y}(u-x)\, \mathrm{d}u,
\\
&\qquad\qquad\qquad\qquad\qquad\qquad\qquad\qquad\qquad\qquad\qquad \qquad\qquad
x\in[0,\infty). \end{split} %
\end{equation}
Note that equation~\eqref{eq:25} for the ruin probability coincides
with the equation derived in \cite{Boi2003} (see also \cite{MiRa2016})
if $\theta=0$ and $\bar\theta=0$.
\end{remark}

\begin{remark}
\label{rem:3}
In Theorem~\ref{thm:1}, we assume that $\theta\neq0$ and $\bar\theta
\neq0$. Otherwise, we do not need to differentiate~\eqref{eq:19}
three times and can obtain equations not involving the third derivative
of $m_2(x)$ instead of~\eqref{eq:10}.

Thus, if $\theta=0$ and $\bar\theta=0$, from~\eqref{eq:21} we have
\begin{equation}
\label{eq:26} dm'_2(x) +(\lambda+\bar\lambda+
\delta_0) m_2(x) =\beta_1(x), \quad x
\in[b,\infty).
\end{equation}

Next, if $\theta\neq0$ and $\bar\theta=0$, from~\eqref{eq:23} we have
\begin{equation}
\label{eq:27} %
\begin{split} &d^2 m''_2(x)
+(3\lambda+2\bar\lambda+2\delta_0)d m'_2(x)
+(2\lambda +\bar\lambda+\delta_0) (\lambda+\bar\lambda+
\delta_0) m_2(x)
\\
&\quad=(2\lambda+\bar\lambda+\delta_0) \beta_1(x) +d
\beta'_1(x) -2\beta_2(x), \quad x\in[b,
\infty). \end{split} %
\end{equation}

Finally, if $\theta=0$ and $\bar\theta\neq0$, multiplying~\eqref
{eq:21} by $(\lambda+2\bar\lambda+\delta_0)/d$ and adding~\eqref{eq:22}
we get
\begin{equation}
\label{eq:28} %
\begin{split} &d^2 m''_2(x)
+(2\lambda+3\bar\lambda+2\delta_0)d m'_2(x)
+(\lambda +2\bar\lambda+\delta_0) (\lambda+\bar\lambda+
\delta_0) m_2(x)
\\
&\quad=(\lambda+2\bar\lambda+\delta_0) \beta_1(x) +d
\beta'_1(x) -2\beta_2(x), \quad x\in[b,
\infty). \end{split} %
\end{equation}
Note that to obtain~\eqref{eq:26}--\eqref{eq:28}, it is enough to have
weaker smoothness assumptions on $f_{Y}(y)$, $f_{\bar Y}(y)$
and $w(u_1,u_2)$.

Equation~\eqref{eq:9} holds in all possible cases.
Furthermore, \eqref{eq:9} involves no derivatives and holds under
weaker assumptions than~\eqref{eq:10}.
To be more precise, we do not need the differentiability of $f_{Y}(y)$,
$f_{\bar Y}(y)$ and $w(u_1,u_2)$ to get~\eqref{eq:9}.
\end{remark}

\section{Equations for the expected discounted dividend payments until ruin}
\label{sec:4}

\begin{thm}
\label{thm:2}
Let the surplus process $ (X_t^b(x) )_{t\ge0}$ follow~\eqref
{eq:6} under the above assumptions with $\theta\neq0$ and $\bar\theta
\neq0$.
Moreover, let the p.d.f.'s $f_{Y}(y)$ and $f_{\bar Y}(y)$ have the
derivatives $f'_{Y}(y)$ and $f'_{\bar Y}(y)$ on $\mathbb R_{+}$,
which are continuous and bounded on $\mathbb R_{+}$.
Then the expected discounted dividend payments until ruin $v(x)$
satisfy the equations\vadjust{\goodbreak}
\begin{equation}
\label{eq:29} %
\begin{split} &(\lambda+ \bar\lambda+
\delta)v_1(x)
\\[-2pt]
&\quad=\lambda\int_0^x v_1(x-y)
f_{Y}(y)\, \mathrm{d}y +\frac{\lambda\theta(\bar\lambda+\delta)}{2\lambda+ \bar\lambda+\delta
}\, \int
_0^x v_1(x-y)
h_{Y}(y)\, \mathrm{d}y
\\[-2pt]
&\qquad+\bar\lambda \Biggl( \int_0^{b-x}
v_1(x+y) f_{\bar Y}(y)\, \mathrm{d}y +\int
_{b-x}^{\infty} v_2(x+y) f_{\bar Y}(y)
\, \mathrm{d}y \Biggr)
\\[-2pt]
&\qquad+\frac{\bar\lambda\bar\theta(\lambda+\delta)}{\lambda\,{+}\, 2\bar
\lambda\,{+}\,\delta} \Biggl( \int_0^{b-x}
v_1(x\,{+}\,y) h_{\bar Y}(y)\, \mathrm{d}y \!+\!\int
_{b-x}^{\infty} v_2(x\,{+}\,y) h_{\bar Y}(y)
\, \mathrm{d}y \Biggr),
\\[-2pt]
&\qquad\qquad\qquad\qquad\qquad\qquad\qquad\qquad\qquad\qquad\qquad \qquad\qquad
x\in[0,b], \end{split} %
\end{equation}
and
\begin{equation}
\label{eq:30} %
\begin{split} &d^3 v'''_2(x)
+(4\lambda+ 4\bar\lambda+3\delta) d^2 v''_2(x)
\\[-2pt]
&\quad+ \bigl( (\lambda+ 2\bar\lambda+\delta) (3\lambda+ 2\bar\lambda +2\delta)
+(2\lambda+ \bar\lambda+\delta) (\lambda+ \bar\lambda+\delta) \bigr) d
v'_2(x)
\\[-2pt]
&\quad+(\lambda+ 2\bar\lambda+\delta) (2\lambda+ \bar\lambda+\delta ) (\lambda+
\bar\lambda+\delta) v_2(x)
\\[-2pt]
&\quad =(\lambda\,{+}\, 2\bar\lambda\,{+}\,\delta) (2\lambda\,{+}\, \bar\lambda\,{+}\,\delta)
\beta_3(x) \,{+}\,(3\lambda\,{+}\, 3\bar\lambda\,{+}\,2\delta)d \beta'_3(x)
\,{+}\,d^2 \beta''_3(x)
\\[-2pt]
&\qquad-2(\lambda+ 2\bar\lambda+\delta) \beta_4(x) -2d
\beta'_4(x) +(\lambda+ 2\bar\lambda+\delta) (2\lambda+
\bar\lambda+\delta)d
\\[-2pt]
&\qquad+2\bar\lambda^2 \bar\theta(\bar\lambda-\lambda) \int
_0^{\infty} v_2(x+y) h_{\bar Y}(y)
\, \mathrm{d}y, \quad x\in[b,\infty), \end{split} %
\end{equation}
where
\begin{equation*}
\begin{split} \beta_3(x) &=\lambda\int
_{x-b}^x v_1(x-y) \bigl(
f_{Y}(y) +\theta h_{Y}(y) \bigr)\, \mathrm{d}y
\\[-2pt]
&\quad+\lambda\int_0^{x-b} v_2(x-y)
\bigl( f_{Y}(y) +\theta h_{Y}(y) \bigr)\, \mathrm{d}y
\\[-2pt]
&\quad+\bar\lambda\int_0^{\infty}
v_2(x+y) \bigl( f_{\bar Y}(y) +\bar \theta
h_{\bar Y}(y) \bigr)\, \mathrm{d}y, \quad x\in[b,\infty), \end{split}
\end{equation*}
and
\begin{align*}
%\begin{split}
\beta_4(x) &=\lambda^2 \theta
\Biggl( \int_{x-b}^x v_1(x-y)
h_{Y}(y)\, \mathrm{d}y +\int_0^{x-b}
v_2(x-y) h_{Y}(y)\, \mathrm{d}y \Biggr)
\\[-2pt]
&\quad+\bar\lambda^2 \bar\theta\int_x^{\infty}
v_2(u) h_{\bar
Y}(u-x)\, \mathrm{d}u, \quad x\in[b,\infty).
%\end{split} %
\end{align*}
\end{thm}

\begin{proof}
The proof is similar to the proof of Theorem~\ref{thm:1}, so we omit
detailed considerations.

Let $x\in[0,b]$. Considering the time and the size of the first jump
of $ (X_t^b(x) )_{t\ge0}$
and applying the law of total probability, we obtain
\begin{equation}
\label{eq:31} %
\begin{split} v(x)&=\int_0^{\infty}
e^{-(\lambda+\bar\lambda)t} \Biggl( \lambda\int_0^x
e^{-\delta t} v(x-y) f_{Y|T}(y \,|\, t)\, \mathrm{d}y
\\
&\quad+\bar\lambda\int_0^{\infty} e^{-\delta t}
v(x+y) f_{\bar Y |
\bar T}(y \,|\, t)\, \mathrm{d}y \Biggr) \mathrm{d}t, \quad x
\in[0,b]. \end{split} %
\end{equation}

Comparing~\eqref{eq:31} with~\eqref{eq:11} and applying arguments
similar to those in the proof of Theorem~\ref{thm:1} yield~\eqref{eq:29}.\vadjust{\goodbreak}

Let now $x\in[b,\infty)$. By the law of total probability, we have
\begin{equation}
\label{eq:32} %
\begin{split} v(x)&=\int_0^{(x-b)/d}
e^{-(\lambda+\bar\lambda)t} \Biggl( (\lambda +\bar\lambda) \int_0^t
de^{-\delta s}\, \mathrm{d}s
\\
&\quad\qquad+\lambda\int_0^{x-dt}
e^{-\delta t} v(x-dt-y) f_{Y|T}(y \, |\, t)\, \mathrm{d}y
\\
&\quad\qquad+\bar\lambda\int_0^{\infty}
e^{-\delta t} v(x-dt+y) f_{\bar Y | \bar T}(y \,|\, t)\, \mathrm{d}y \Biggr)
\mathrm{d}t
\\
&\quad+\int_{(x-b)/d}^{\infty} e^{-(\lambda+\bar\lambda)t} \Biggl( (
\lambda+\bar\lambda) \int_0^{(x-b)/d}
de^{-\delta s}\, \mathrm{d}s
\\
&\quad\qquad+\lambda\int_0^b e^{-\delta t}
v(b-y) f_{Y|T}(y \,|\, t)\, \mathrm{d}y
\\
&\quad\qquad+\bar\lambda\int_0^{\infty}
e^{-\delta t} v(b+y) f_{\bar
Y | \bar T}(y \,|\, t)\, \mathrm{d}y \Biggr)
\mathrm{d}t, \quad x\in[b,\infty). \end{split} %
\end{equation}

Taking into account that
\begin{align*}
%\begin{split}
&\int_0^{(x-b)/d} (\lambda+
\bar\lambda) e^{-(\lambda+\bar\lambda)t} \int_0^t
de^{-\delta s}\, \mathrm{d}s\, \mathrm{d}t
\\
&\quad=\frac{(\lambda+\bar\lambda)d}{\delta}\, \int_0^{(x-b)/d}
e^{-(\lambda+\bar\lambda)t} \bigl(1-e^{-\delta t}\bigr)\, \mathrm{d}t
\\
&\quad=\frac{d}{\delta}\, \biggl( \frac{\delta}{\lambda+\bar\lambda
+\delta} -e^{-(\lambda+\bar\lambda)(x-b)/d} +
\frac{\lambda+\bar\lambda}{\lambda+\bar\lambda+\delta}\, e^{-(\lambda+\bar\lambda+\delta)(x-b)/d} \biggr)
%\end{split} %
\end{align*}
and
\begin{align*}
%\begin{split}
&\int_{(x-b)/d}^{\infty} (\lambda+
\bar\lambda) e^{-(\lambda+\bar
\lambda)t} \int_0^{(x-b)/d}
de^{-\delta s}\, \mathrm{d}s\, \mathrm{d}t
\\
&\quad=\frac{(\lambda+\bar\lambda)d}{\delta}\,\bigl(1-e^{-\delta
(x-b)/d}\bigr)\, \int
_{(x-b)/d}^{\infty} e^{-(\lambda+\bar\lambda)t}\, \mathrm{d}t
\\
&\quad=\frac{d}{\delta}\, \bigl( e^{-(\lambda+\bar\lambda)(x-b)/d} -e^{-(\lambda+\bar\lambda+\delta)(x-b)/d} \bigr),
%\end{split}
\end{align*}
substituting~\eqref{eq:4} and~\eqref{eq:5} into~\eqref{eq:32} and
using~\eqref{eq:8} give
\begin{equation}
\label{eq:33} %
\begin{split} v_2(x) &=I_{7,8,9}(x)
+I_{10,11,12}(x)
\\
&\quad+\frac{d}{\lambda+\bar\lambda+\delta}\, \bigl( 1-e^{-(\lambda
+\bar\lambda+\delta)(x-b)/d} \bigr), \quad x\in[b,
\infty), \end{split} %
\end{equation}
where
\begin{align*}
%\begin{split}
I_{7,8,9}(x)&=\int_0^{(x-b)/d}
e^{-(\lambda+\bar\lambda+\delta)t}
\\
&\quad\times \Biggl( \lambda\int_0^{x-dt-b}
v_2(x-dt-y) \bigl( f_{Y}(y) +\theta h_{Y}(y)
\bigl(2e^{-\lambda t}-1\bigr) \bigr)\, \mathrm{d}y
\\
&\quad +\lambda\int_{x-dt-b}^{x-dt}
v_1(x-dt-y) \bigl( f_{Y}(y) +\theta
h_{Y}(y) \bigl(2e^{-\lambda t}-1\bigr) \bigr)\, \mathrm{d}y
\\
&\quad +\bar\lambda\int_{0}^{\infty}
v_2(x-dt+y) \bigl( f_{\bar
Y}(y) +\bar\theta
h_{\bar Y}(y) \bigl(2e^{-\bar\lambda t}-1\bigr) \bigr)\, \mathrm{d}y \Biggr)
\mathrm{d}t %\end{split}
\end{align*}
and
\begin{equation*}
\begin{split} I_{10,11,12}(x)&=\int_{(x-b)/d}^{\infty}
e^{-(\lambda+\bar\lambda
+\delta)t}
\\
&\quad\times \Biggl( \lambda\int_0^{b}
v_1(b-y) \bigl( f_{Y}(y) +\theta h_{Y}(y)
\bigl(2e^{-\lambda t}-1\bigr) \bigr)\, \mathrm{d}y
\\
&\quad +\bar\lambda\int_{0}^{\infty}
v_2(b+y) \bigl( f_{\bar Y}(y) +\bar\theta h_{\bar Y}(y)
\bigl(2e^{-\bar\lambda t}-1\bigr) \bigr)\, \mathrm{d}y \Biggr) \mathrm{d}t.
\end{split} %
\end{equation*}

Changing the variable $x-dt=s$ in the outer integral of the expression
for $I_{7,8,9}(x)$ yields
\begin{equation}
\label{eq:34} %
\begin{split} I_{7,8,9}(x)&=\frac{1}{d}
\, e^{-(\lambda+\bar\lambda+\delta)x/d}\, I_7(x) +\frac{2}{d}\, e^{-(2\lambda+\bar\lambda+\delta)x/d}
\, I_8(x)
\\
&\quad+\frac{2}{d}\, e^{-(\lambda+2\bar\lambda+\delta)x/d}\, I_9(x), \quad x
\in[b,\infty), \end{split} %
\end{equation}
where
\begin{align*}
%\begin{split}
I_7(x)&=\int_b^x
e^{(\lambda+\bar\lambda+\delta)s/d} \Biggl( \lambda\int_0^{s-b}
v_2(s-y) \bigl( f_{Y}(y) -\theta h_{Y}(y)
\bigr)\, \mathrm{d}y
\\
&\quad+\lambda\int_{s-b}^{s} v_1(s-y)
\bigl( f_{Y}(y) -\theta h_{Y}(y) \bigr)\, \mathrm{d}y
\\
&\quad+\bar\lambda\int_{0}^{\infty}
v_2(s+y) \bigl( f_{\bar Y}(y) -\bar\theta h_{\bar Y}(y)
\bigr)\, \mathrm{d}y \Biggr) \mathrm{d}s,\\ %\end{split}\end{equation*}
%\begin{equation*}
%\begin{split}
I_8(x)&=\lambda\theta\int
_b^x e^{(2\lambda+\bar\lambda+\delta)s/d} \Biggl( \int
_0^{s-b} v_2(s-y) h_{Y}(y)
\, \mathrm{d}y
\\
&\quad+\int_{s-b}^{s} v_1(s-y)
h_{Y}(y)\, \mathrm{d}y \Biggr) \mathrm{d}s
%\end{split} %
\end{align*}
and
\begin{equation*}
\begin{split} I_9(x)&=\bar\lambda\bar\theta\int
_b^x e^{(\lambda+2\bar\lambda
+\delta)s/d} \Biggl( \int
_0^{\infty} v_2(s+y) h_{\bar Y}(y)
\, \mathrm{d}y \Biggr) \mathrm{d}s. \end{split} %
\end{equation*}

Separating the integrals in the expression for $I_{10,11,12}(x)$ into
integrals w.r.t. either $t$ or $y$
and then taking the integrals w.r.t. $t$ we obtain
\begin{equation}
\label{eq:35} I_{10,11,12}(x) =I_{10}(x) +I_{11}(x)
+I_{12}(x), \quad x\in[b,\infty),
\end{equation}
where
\begin{align*}
%\begin{split}
I_{10}(x)&=\frac{e^{-(\lambda+\bar\lambda+\delta)(x-b)/d}}{\lambda
+\bar\lambda+\delta}
\\
&\quad\times \Biggl( \lambda\int_0^b
v_1(b-y) f_{Y}(y)\, \mathrm{d}y -\lambda\theta\int
_0^b v_1(b-y)
h_{Y}(y)\, \mathrm{d}y
\\
&\quad +\bar\lambda\int_0^{\infty}
v_2(b+y) f_{\bar Y}(y)\, \mathrm{d}y -\bar\lambda\bar\theta
\int_0^{\infty} v_2(b+y)
h_{\bar Y}(y)\, \mathrm {d}y \Biggr),\\ %\end{split}\end{equation*}
%\[
I_{11}(x)&=\frac{2\lambda\theta e^{-(2\lambda+\bar\lambda+\delta
)(x-b)/d}}{2\lambda+\bar\lambda+\delta}\, \int_0^b
v_1(b-y) h_{Y}(y)\, \mathrm{d}y
%\end{split}
\end{align*}
and
\[
I_{12}(x)=\frac{2 \bar\lambda\bar\theta e^{-(\lambda+2\bar\lambda
+\delta)(x-b)/d}}{\lambda+2\bar\lambda+\delta}\, \int_0^{\infty}
v_2(b+y) h_{\bar Y}(y)\, \mathrm{d}y.
\]

Thus, substituting~\eqref{eq:34} and~\eqref{eq:35} into~\eqref{eq:33}
we have
\begin{align}
%\begin{split}
v_2(x)&=\frac{1}{d}\,
e^{-(\lambda+\bar\lambda+\delta)x/d}\, I_7(x) +\frac{2}{d}\, e^{-(2\lambda+\bar\lambda+\delta)x/d}\,
I_8(x)\notag
\\
&\quad+\frac{2}{d}\, e^{-(\lambda+2\bar\lambda+\delta)x/d}\, I_9(x)
+I_{10}(x) +I_{11}(x) +I_{12}(x)\label{eq:36}
\\
&\quad+\frac{d}{\lambda+\bar\lambda+\delta}\, \bigl( 1-e^{-(\lambda
+\bar\lambda+\delta)(x-b)/d} \bigr), \quad x\in[b,
\infty).\notag %\end{split} %
\end{align}

It is easily seen from~\eqref{eq:29} that $v_1(x)$ is continuous on
$[0,b]$, and from~\eqref{eq:36} we conclude that $v_2(x)$
is continuous on $[b,\infty)$. Hence, from~\eqref{eq:36} we deduce that
$v_2(x)$ is differentiable on $[b,\infty)$.
Differentiating~\eqref{eq:36} gives
\begin{equation}
\label{eq:37} %
\begin{split} v'_2(x)&=-
\frac{\lambda+\bar\lambda+\delta}{d^2}\, e^{-(\lambda+\bar
\lambda+\delta)x/d}\, I_7(x)
\\
&\quad-\frac{2(2\lambda+\bar\lambda+\delta)}{d^2}\, e^{-(2\lambda
+\bar\lambda+\delta)x/d}\, I_8(x)
\\
&\quad-\frac{2(\lambda+2\bar\lambda+\delta)}{d^2}\, e^{-(\lambda
+2\bar\lambda+\delta)x/d}\, I_9(x)
\\
&\quad-\frac{\lambda+\bar\lambda+\delta}{d}\, I_{10}(x) -\frac
{2\lambda+\bar\lambda+\delta}{d}\,
I_{11}(x) -\frac{\lambda+2\bar\lambda+\delta}{d}\, I_{12}(x)
\\
&\quad+\frac{1}{d}\, \beta_3(x) +e^{-(\lambda+\bar\lambda+\delta
)(x-b)/d}, \quad x
\in[b,\infty), \end{split} %
\end{equation}
where the function $\beta_3(x)$ is defined in the assertion of the theorem.

Multiplying~\eqref{eq:36} by $(\lambda+\bar\lambda+\delta)/d$ and
adding~\eqref{eq:37} we get
\begin{equation}
\label{eq:38} %
\begin{split}\!\!\!\!dv'_2(x) +(
\lambda+\bar\lambda+\delta) v_2(x)& =-\frac{2\lambda}{d}\,
e^{-(2\lambda+\bar\lambda+\delta)x/d}\, I_8(x)
\\
&\quad-\frac{2 \bar\lambda}{d}\, e^{-(\lambda+2\bar\lambda+\delta
)x/d}\, I_9(x) -\lambda
I_{11}(x) -\bar\lambda I_{12}(x)
\\
&\quad+\beta_3(x) +d, \quad x\in[b,\infty). \end{split} %
\end{equation}

Since $v_1(x)$ and $v_2(x)$ are continuous and bounded on $[0,b]$ and
$[b,\infty)$, respectively, from~\eqref{eq:38}
we conclude that so is $v'_2(x)$ on $[b,\infty)$.
Taking into account that $f'_{Y}(y)$ and $f'_{\bar Y}(y)$ are
continuous and bounded on $\mathbb R_{+}$, from~\eqref{eq:29}
we conclude that so is $v'_1(x)$ on $[0,b]$.
Hence, $\beta_3(x)$ is differentiable on $[b,\infty)$. From this
and~\eqref{eq:38} it follows that $v_2(x)$ is
twice differentiable on $[b,\infty)$. Differentiating~\eqref{eq:38} gives
\begin{equation}
\label{eq:39} %
\begin{split} &dv''_2(x)
+(\lambda+\bar\lambda+\delta) v'_2(x) =
\frac{2\lambda
(2\lambda+\bar\lambda+\delta)}{d^2}\, e^{-(2\lambda+\bar\lambda+\delta)x/d}\, I_8(x)
\\
&\quad+\frac{2 \bar\lambda(\lambda+2\bar\lambda+\delta)}{d^2}\, e^{-(\lambda+2\bar\lambda+\delta)x/d}\, I_9(x) +
\frac{\lambda(2\lambda+\bar\lambda+\delta)}{d}\, I_{11}(x)
\\
&\quad+\frac{\bar\lambda(\lambda+2\bar\lambda+\delta)}{d}\, I_{12}(x) +\beta'_3(x)
-\frac{2}{d}\, \beta_4(x), \quad x\in[b,\infty), \end{split}
\end{equation}
where
\begin{align*}
%\begin{split}
\beta'_3(x) &=\lambda\int
_{x-b}^x v_1'(x-y) \bigl(
f_{Y}(y) +\theta h_{Y}(y) \bigr)\, \mathrm{d}y
\\
&\quad+\lambda\int_0^{x-b} v_2'(x-y)
\bigl( f_{Y}(y) +\theta h_{Y}(y) \bigr)\, \mathrm{d}y
\\
&\quad+\bar\lambda\int_0^{\infty}
v_2'(x+y) \bigl( f_{\bar Y}(y) +\bar \theta
h_{\bar Y}(y) \bigr)\, \mathrm{d}y
\\
&\quad+\lambda v_1(0) \bigl( f_{Y}(x) +\theta
h_{Y}(x) \bigr), \quad x\in[b,\infty), %\end{split} %
\end{align*}
which is continuous and bounded on $[b,\infty)$, and the function $\beta
_4(x)$ is defined in the assertion of the theorem.

Multiplying~\eqref{eq:38} by $(2\lambda+\bar\lambda+\delta)/d$ and
adding~\eqref{eq:39} we obtain
\begin{equation}
\label{eq:40} %
\begin{split} &d^2 v''_2(x)
+(3\lambda+2\bar\lambda+2\delta)d v'_2(x) +(2\lambda +
\bar\lambda+\delta) (\lambda+\bar\lambda+\delta) v_2(x)
\\
&\quad=\frac{2 \bar\lambda(\bar\lambda-\lambda)}{d}\, e^{-(\lambda
+2\bar\lambda+\delta)x/d}\, I_9(x) +\bar
\lambda(\bar\lambda-\lambda) I_{12}(x)
\\
&\qquad+(2\lambda+\bar\lambda+\delta) \beta_3(x) +d
\beta'_3(x) -2\beta_4(x)
\\
&\qquad+(2\lambda+\bar\lambda+\delta)d, \quad x\in[b,\infty). \end{split}
\end{equation}

It is easily seen from~\eqref{eq:40} that $v''_2(x)$ is continuous and
bounded on $[b,\infty)$.
Taking into account that $f'_{Y}(y)$ and $f'_{\bar Y}(y)$ are
continuous and bounded on $\mathbb R_{+}$,
from~\eqref{eq:29} we conclude that so is $v''_1(x)$ on $[0,b]$.
Hence, $\beta_3(x)$ is twice differentiable on $[b,\infty)$.
Moreover, applying similar arguments shows that $\beta_4(x)$ is
differentiable on $[b,\infty)$.
From this and~\eqref{eq:40} it follows that $v_2(x)$ has the third
derivative on $[b,\infty)$.
Differentiating~\eqref{eq:40} gives
\begin{equation}
\label{eq:41} %
\begin{split} &d^2 v'''_2(x)
+(3\lambda+2\bar\lambda+2\delta)d v''_2(x)
+(2\lambda +\bar\lambda+\delta) (\lambda+\bar\lambda+\delta) v'_2(x)
\\
&\quad=-\frac{2 \bar\lambda(\bar\lambda-\lambda) (\lambda+2\bar
\lambda+\delta)}{d^2}\, e^{-(\lambda+2\bar\lambda+\delta)x/d}\, I_9(x)
\\
&\qquad-\frac{\bar\lambda(\bar\lambda-\lambda) (\lambda+2\bar
\lambda+\delta)}{d}\, I_{12}(x) +(2\lambda+\bar\lambda+\delta)
\beta'_3(x) +d\beta''_3(x)
\\
&\qquad-2\beta'_4(x) +\frac{2 \bar\lambda^2 \bar\theta(\bar\lambda
-\lambda)}{d}\, \int
_0^{\infty} v_2(x+y) h_{\bar Y}(y)
\, \mathrm{d}y, \quad x\in[b,\infty). \end{split} %
\end{equation}

Multiplying~\eqref{eq:40} by $(\lambda+2\bar\lambda+\delta_0)/d$ and
adding~\eqref{eq:41} yield~\eqref{eq:30},
which completes the proof.
\end{proof}

\begin{remark}
\label{rem:4}
To solve equations~\eqref{eq:29} and~\eqref{eq:30}, we use the
following boundary conditions. First of all, we have $v_1(b)=v_2(b)$.
Next, if the net profit condition holds, applying arguments similar to
those in \cite[p.~70]{Sc2008} we can show that
$\lim_{x\to\infty} v_2(x) =d/{\delta}$.
Moreover, we can substitute $x=b$ into the intermediate equations
(e.g., equation~\eqref{eq:38}) to get additional boundary conditions
involving derivatives of $v_2(x)$. The uniqueness of the required
solutions should also be justified.
If $\theta=0$ and $\bar\theta=0$, we can find explicit solutions to the
equations (see Section~\ref{sec:5}).
\end{remark}

\begin{remark}
\label{rem:5}
If at least one of the parameters $\theta$ and $\bar\theta$ is equal to
0, we do not need to differentiate~\eqref{eq:36}
three times and can obtain equations not involving the third derivative
of $v_2(x)$ instead of~\eqref{eq:30}.

Thus, if $\theta=0$ and $\bar\theta=0$, from~\eqref{eq:38} we have
\begin{equation}
\label{eq:42} dv'_2(x) +(\lambda+\bar\lambda+
\delta_0) v_2(x) =\beta_3(x)+d, \quad x
\in[b,\infty).
\end{equation}

If $\theta\neq0$ and $\bar\theta=0$, from~\eqref{eq:40} we get
\begin{equation}
\label{eq:43} %
\begin{split} &d^2 v''_2(x)
+(3\lambda+2\bar\lambda+2\delta_0)d v'_2(x)
+(2\lambda +\bar\lambda+\delta) (\lambda+\bar\lambda+\delta) v_2(x)
\\
&\quad=(2\lambda\,{+}\,\bar\lambda\,{+}\,\delta) \beta_3(x) \,{+}\,d
\beta'_3(x) \,{-}\,2\beta _4(x) \,{+}\,(2\lambda\,{+}\,\bar
\lambda\,{+}\,\delta)d, \quad x\in[b,\infty). \end{split} %
\end{equation}

If $\theta=0$ and $\bar\theta\neq0$, multiplying~\eqref{eq:38} by
$(\lambda+2\bar\lambda+\delta_0)/d$ and adding~\eqref{eq:39}
we have
\begin{equation}
\label{eq:44} %
\begin{split} &d^2 v''_2(x)
+(2\lambda+3\bar\lambda+2\delta)d v'_2(x) +(\lambda +2
\bar\lambda+\delta) (\lambda+\bar\lambda+\delta) v_2(x)
\\
&\quad=(\lambda\,{+}\,2\bar\lambda\,{+}\,\delta) \beta_3(x) \,{+}\,d
\beta'_3(x) \,{-}\,2\beta _4(x) \,{+}\,(\lambda\,{+}\,2\bar
\lambda\,{+}\,\delta)d, \quad x\in[b,\infty). \end{split} %
\end{equation}
To obtain~\eqref{eq:42}--\eqref{eq:44}, it is enough to have weaker
smoothness assumptions on $f_{Y}(y)$ and $f_{\bar Y}(y)$.

Equation~\eqref{eq:29} is true in all possible cases.
Since~\eqref{eq:29} involves no derivatives, it holds under weaker
assumptions than~\eqref{eq:30}.
To obtain~\eqref{eq:29}, we do not need the differentiability of
$f_{Y}(y)$ and $f_{\bar Y}(y)$.
\end{remark}

\section{Exponentially distributed claim and premium sizes}
\label{sec:5}

In this section, we deal with exponentially distributed claim and
premium sizes, i.e.
\begin{equation}
\label{eq:45} f_{Y}(y) =\frac{1}{\mu}\, e^{-y/\mu},
\qquad h_{Y}(y) =\frac{2}{\mu}\, e^{-2y/\mu} -
\frac{1}{\mu}\, e^{-y/\mu}, \quad y\ge0,
\end{equation}
and
\begin{equation}
\label{eq:46} f_{\bar Y}(y) =\frac{1}{\bar\mu}\, e^{-y/\bar\mu},
\qquad h_{\bar Y}(y) =\frac{2}{\bar\mu}\, e^{-2y/\bar\mu} -
\frac{1}{\bar\mu}\, e^{-y/\bar\mu}, \quad y\ge0.
\end{equation}

\subsection{The ruin probability in the model without dividend payments}
\label{sec:5.1}

If no dividends are paid, equation~\eqref{eq:25} for the ruin
probability $\psi(x)$ takes the form
\begin{equation}
\label{eq:47} %
\begin{split}(\lambda+ \bar\lambda)\psi(x) &=\lambda
\Biggl( \int_0^x \psi(u) f_{Y}(x-u)
\, \mathrm{d}u +\int_0^{\infty} f_{Y}(x+u)
\, \mathrm{d}u \Biggr)
\\
&\quad+\frac{\lambda\bar\lambda\theta}{2\lambda+ \bar\lambda} \Biggl( \int_0^x
\psi(u) h_{Y}(x-u)\, \mathrm{d}u +\int_0^{\infty}
h_{Y}(x+u)\, \mathrm{d}u \Biggr)
\\
&\quad+\bar\lambda\int_x^{\infty} \psi(u)
f_{\bar Y}(u-x)\, \mathrm{d}u \\
&\quad+\frac{\lambda\bar\lambda\bar\theta}{\lambda+ 2\bar\lambda}\, \int
_x^{\infty} \psi(u) h_{\bar Y}(u-x)\,
\mathrm{d}u, \quad x\in[0,\infty). \end{split} %
\end{equation}

Substituting~\eqref{eq:45} and~\eqref{eq:46} into~\eqref{eq:47} gives
\begin{equation}
\label{eq:48} %
\begin{split} (\lambda+ \bar\lambda)\psi(x)&= \biggl(
\lambda- \frac{\lambda\bar
\lambda\theta}{2\lambda+ \bar\lambda} \biggr) I_{13}(x) +\frac{\lambda\bar\lambda\theta}{2\lambda+ \bar\lambda}\,
I_{14}(x)
\\
&\quad+ \biggl(\bar\lambda- \frac{\lambda\bar\lambda\bar\theta
}{\lambda+ 2\bar\lambda} \biggr) I_{15}(x) +
\frac{\lambda\bar\lambda\bar\theta}{\lambda+ 2\bar\lambda}\, I_{16}(x)
\\
&\quad+ \biggl(\lambda- \frac{\lambda\bar\lambda\theta}{2\lambda+
\bar\lambda} \biggr) e^{-x/\mu} +
\frac{\lambda\bar\lambda\theta}{2\lambda+ \bar\lambda}\, e^{-2x/\mu
}, \quad x\in[0,\infty), \end{split}
\end{equation}
where
\begin{align*}
I_{13}(x)&= \frac{1}{\mu}\, e^{-x/\mu} \int
_0^x \psi(u) e^{u/\mu}\, \mathrm{d}u,
\qquad I_{14}(x)= \frac{2}{\mu}\, e^{-2x/\mu} \int
_0^x \psi(u) e^{2u/\mu}\, \mathrm{d}u,\\
I_{15}(x)&= \frac{1}{\bar\mu}\, e^{x/\bar\mu} \int
_x^{\infty} \psi(u) e^{-u/\bar\mu}\, \mathrm{d}u,
\qquad I_{16}(x)= \frac{2}{\bar\mu}\, e^{2x/\bar\mu} \int
_x^{\infty} \psi(u) e^{-2u/\bar\mu}\, \mathrm{d}u.
\end{align*}

We now show that if either $\theta=0$ or $\bar\theta=0$,
integro-differential equation~\eqref{eq:48} can be reduced
to a third-order linear differential equation with constant coefficients.

\begin{lemma}
\label{lem:1}
Let the surplus process $ (X_t(x) )_{t\ge0}$ follow~\eqref
{eq:1} under the above assumptions,
and let claim and premium sizes be exponentially distributed with means
$\mu$ and $\bar\mu$, respectively.

If $\theta\neq0$ and $\bar\theta=0$, then $\psi(x)$ is a solution to
the differential equation
\begin{equation}
\label{eq:49} %
\begin{split} &\mu^2 \bar\mu(\lambda+\bar
\lambda) (2\lambda+\bar\lambda) \psi '''(x)
\\
&\quad+ \bigl( \mu\bar\mu(2\lambda+3\bar\lambda) (2\lambda+\bar \lambda) -
\lambda\mu^2 (2\lambda+\bar\lambda) +\lambda\bar\lambda\mu\bar\mu
\theta \bigr) \psi''(x)
\\
&\quad+ \bigl( 2(\bar\lambda\bar\mu-\lambda\mu) (2\lambda+\bar\lambda ) +\lambda
\bar\lambda\mu\theta \bigr) \psi'(x) =0, \quad x\in[0,\infty).
\end{split} %
\end{equation}

If $\theta=0$ and $\bar\theta\neq0$, then $\psi(x)$ is a solution to
the differential equation
\begin{equation}
\label{eq:50} %
\begin{split} &\mu\bar\mu^2 (\lambda+\bar
\lambda) (\lambda+2\bar\lambda) \psi '''(x)
\\
&\quad+ \bigl( -\mu\bar\mu(3\lambda+2\bar\lambda) (\lambda+2\bar \lambda) +\bar
\lambda\bar\mu^2 (\lambda+2\bar\lambda) +\lambda\bar\lambda\mu\bar\mu
\bar\theta \bigr) \psi''(x)
\\
&\quad+ \bigl( 2(\lambda\mu-\bar\lambda\bar\mu) (\lambda+2\bar\lambda ) +\lambda
\bar\lambda\bar\mu\bar\theta \bigr) \psi'(x) =0, \quad x\in[0,
\infty). \end{split} %
\end{equation}
\end{lemma}

\begin{proof}
First of all, note that
\begin{align*}
I'_{13}(x)&= -\frac{1}{\mu}\, I_{13}(x) +
\frac{1}{\mu}\, \psi(x), \qquad I'_{14}(x)= -
\frac{2}{\mu}\, I_{14}(x) +\frac{2}{\mu}\, \psi(x),\\
I'_{15}(x)&= \frac{1}{\bar\mu}\, I_{15}(x) -
\frac{1}{\bar\mu}\, \psi(x), \qquad I'_{16}(x)=
\frac{2}{\bar\mu}\, I_{16}(x) -\frac{2}{\bar\mu}\, \psi(x).
\end{align*}

From~\eqref{eq:48}, it is easily seen that $\psi(x)$ is differentiable
on $[0,\infty)$.
Therefore, differentiating~\eqref{eq:48} yields
\begin{equation}
\label{eq:51} %
\begin{split}&(\lambda+ \bar\lambda)
\psi'(x)\\
&\quad= -\frac{1}{\mu}\, \biggl(\lambda- \frac
{\lambda\bar\lambda\theta}{2\lambda+ \bar\lambda}
\biggr) I_{13}(x) -\frac{2}{\mu}\, \frac{\lambda\bar\lambda\theta}{2\lambda+ \bar\lambda
}\,
I_{14}(x)
\\
&\qquad+\frac{1}{\bar\mu}\, \biggl(\bar\lambda- \frac{\lambda\bar
\lambda\bar\theta}{\lambda+ 2\bar\lambda} \biggr)
I_{15}(x) +\frac{2}{\bar\mu}\, \frac{\lambda\bar\lambda\bar\theta}{\lambda+
2\bar\lambda}\,
I_{16}(x) + \biggl( \frac{1}{\mu}\, \biggl(\lambda-
\frac{\lambda\bar\lambda
\theta}{2\lambda+ \bar\lambda} \biggr)
\\
&\quad\qquad+\frac{2}{\mu}\, \frac{\lambda\bar\lambda\theta}{2\lambda+
\bar\lambda} -\frac{1}{\bar\mu}\, \biggl(
\bar\lambda- \frac{\lambda\bar\lambda\bar
\theta}{\lambda+ 2\bar\lambda} \biggr) -\frac{2}{\bar\mu}\, \frac{\lambda\bar\lambda\bar\theta}{\lambda+
2\bar\lambda}
\biggr) \psi(x)
\\
&\qquad-\frac{1}{\mu}\, \biggl(\lambda- \frac{\lambda\bar\lambda
\theta}{2\lambda+ \bar\lambda} \biggr)
e^{-x/\mu} -\frac{2}{\mu}\, \frac{\lambda\bar\lambda\theta}{2\lambda+ \bar\lambda
}\, e^{-2x/\mu},
\quad x\in[0,\infty). \end{split} %
\end{equation}

Multiplying~\eqref{eq:51} by $\mu$ and adding~\eqref{eq:48} we obtain
\begin{equation}
\label{eq:52} %
\begin{split} &\mu(\lambda+ \bar\lambda)
\psi'(x) +\bar\lambda\psi(x)\\
&\quad= -\frac{\lambda
\bar\lambda\theta}{2\lambda\,{+}\, \bar\lambda}\,
I_{14}(x)
\,{+}\, \biggl(1\,{+}\,\frac{\mu}{\bar\mu} \biggr) \biggl(\bar\lambda\,{-}\,
\frac
{\lambda\bar\lambda\bar\theta}{\lambda\,{+}\, 2\bar\lambda} \biggr) I_{15}(x) \,{+}\, \biggl(1\,{+}\,\frac{2\mu}{\bar\mu}
\biggr) \frac{\lambda\bar\lambda\bar
\theta}{\lambda\,{+}\, 2\bar\lambda}\, I_{16}(x)
\\
&\qquad+ \biggl( \frac{\lambda\bar\lambda\theta}{2\lambda+ \bar\lambda} -\frac{\mu}{\bar\mu}\, \biggl(\bar\lambda+
\frac{\lambda\bar\lambda
\bar\theta}{\lambda+ 2\bar\lambda} \biggr) \biggr) \psi(x) -\frac{\lambda\bar\lambda\theta}{2\lambda+ \bar\lambda}\, e^{-2x/\mu
},
\quad x\in[0,\infty). \end{split} %
\end{equation}

From~\eqref{eq:52}, it follows that $\psi(x)$ is twice differentiable
on $[0,\infty)$.
Differentiating~\eqref{eq:52} gives
\begin{align}
%\begin{split}
&\mu(\lambda+ \bar\lambda)
\psi''(x) +\bar\lambda\psi'(x)\notag\\
&\quad=
\frac{2}{\mu
}\, \frac{\lambda\bar\lambda\theta}{2\lambda+ \bar\lambda}\, I_{14}(x)
+\frac{1}{\bar\mu}\, \biggl(1+\frac{\mu}{\bar\mu} \biggr) \biggl(\bar
\lambda- \frac{\lambda\bar\lambda\bar\theta}{\lambda+
2\bar\lambda} \biggr) I_{15}(x)\notag\\
&\qquad +\frac{2}{\bar\mu}\,
\biggl(1+\frac{2\mu}{\bar\mu} \biggr) \frac{\lambda
\bar\lambda\bar\theta}{\lambda+ 2\bar\lambda}\, I_{16}(x)
+ \biggl( \frac{\lambda\bar\lambda\theta}{2\lambda+ \bar\lambda} -\frac{\mu}{\bar\mu}\, \biggl(\bar\lambda+
\frac{\lambda\bar\lambda
\bar\theta}{\lambda+ 2\bar\lambda} \biggr) \biggr) \psi'(x)\notag\\
&\qquad - \biggl(
\frac{2}{\mu}\, \frac{\lambda\bar\lambda\theta}{2\lambda+
\bar\lambda}
+\frac{1}{\bar\mu}\, \biggl(1+\frac{\mu}{\bar\mu} \biggr) \biggl(\bar
\lambda- \frac{\lambda\bar\lambda\bar\theta}{\lambda+
2\bar\lambda} \biggr) +\frac{2}{\bar\mu}\, \biggl(1+
\frac{2\mu}{\bar\mu} \biggr) \frac{\lambda
\bar\lambda\bar\theta}{\lambda+ 2\bar\lambda} \biggr) \psi(x)\notag
\\
&\qquad+ \frac{2}{\mu}\, \frac{\lambda\bar\lambda\theta}{2\lambda+
\bar\lambda}\, e^{-2x/\mu}, \quad x
\in[0,\infty).\label{eq:53} %\end{split}
\end{align}

Multiplying~\eqref{eq:53} by $(-\bar\mu)$ and adding~\eqref{eq:52} we get
\begin{equation}
\label{eq:54} %
\begin{split} &-\mu\bar\mu(\lambda+ \bar\lambda)
\psi''(x) +(\lambda\mu-\bar\lambda\bar \mu)
\psi'(x)\\
&\quad =- \biggl(1+\frac{2\bar\mu}{\mu} \biggr) \frac{\lambda\bar\lambda\theta
}{2\lambda+ \bar\lambda}\,
I_{14}(x)
- \biggl(1+\frac{2\mu}{\bar\mu} \biggr) \frac{\lambda\bar\lambda
\bar\theta}{\lambda+ 2\bar\lambda}\,
I_{16}(x) \\
&\qquad+ \biggl( -\frac{\lambda\bar\lambda\bar\mu\theta}{2\lambda+ \bar
\lambda} +\frac{\lambda\bar\lambda\mu\bar\theta}{\lambda+ 2\bar\lambda} \biggr)
\psi'(x)
+ \biggl( \biggl(1+\frac{2\bar\mu}{\mu} \biggr) \frac{\lambda\bar
\lambda\theta}{2\lambda+ \bar\lambda}\\
&\qquad +
\biggl(1+\frac{2\mu}{\bar\mu} \biggr) \frac{\lambda\bar\lambda\bar
\theta}{\lambda+ 2\bar\lambda} \biggr) \psi(x)
- \biggl(1+\frac{2\bar\mu}{\mu} \biggr) \frac{\lambda\bar\lambda
\theta}{2\lambda+ \bar\lambda}\,
e^{-2x/\mu}, \quad x\in[0,\infty). \end{split} %
\end{equation}

Let now $\theta\neq0$ and $\bar\theta=0$. Then~\eqref{eq:54} takes the form
\begin{align}
%\begin{split}
&-\mu\bar\mu(\lambda+ \bar\lambda)
\psi''(x) +(\lambda\mu-\bar\lambda\bar \mu)
\psi'(x)\notag\\
&\quad =- \biggl(1+\frac{2\bar\mu}{\mu} \biggr) \frac{\lambda\bar\lambda\theta
}{2\lambda+ \bar\lambda}\,
I_{14}(x)
-\frac{\lambda\bar\lambda\bar\mu\theta}{2\lambda+ \bar\lambda
}\, \psi'(x) + \biggl(1+\frac{2\bar\mu}{\mu}
\biggr) \frac{\lambda\bar\lambda\theta
}{2\lambda+ \bar\lambda}\, \psi(x)\notag
\\
&\qquad- \biggl(1+\frac{2\bar\mu}{\mu} \biggr) \frac{\lambda\bar\lambda
\theta}{2\lambda+ \bar\lambda}\,
e^{-2x/\mu}, \quad x\in[0,\infty).\label{eq:55} %\end{split}
\end{align}

From~\eqref{eq:55}, it follows that $\psi(x)$ has the third derivative
on $[0,\infty)$.
Differentiating~\eqref{eq:55} yields
\begin{equation}
\label{eq:56} %
\begin{split} &-\mu\bar\mu(\lambda+ \bar\lambda)
\psi'''(x) +(\lambda\mu-\bar\lambda \bar
\mu)\psi''(x) \\
&\quad=\frac{2}{\mu}\, \biggl(1+
\frac{2\bar\mu}{\mu} \biggr) \frac{\lambda\bar
\lambda\theta}{2\lambda+ \bar\lambda}\, I_{14}(x)
-\frac{\lambda\bar\lambda\bar\mu\theta}{2\lambda+ \bar\lambda
}\, \psi''(x) + \biggl(1+
\frac{2\bar\mu}{\mu} \biggr) \frac{\lambda\bar\lambda\theta
}{2\lambda+ \bar\lambda}\, \psi'(x)
\\
&\qquad-\frac{2}{\mu}\, \biggl(1+\frac{2\bar\mu}{\mu} \biggr)
\frac
{\lambda\bar\lambda\theta}{2\lambda+ \bar\lambda}\, \psi(x) +\frac{2}{\mu}\, \biggl(1+\frac{2\bar\mu}{\mu}
\biggr) \frac{\lambda\bar
\lambda\theta}{2\lambda+ \bar\lambda}\, e^{-2x/\mu}, \quad x\in[0,\infty). \end{split}
\end{equation}

Multiplying~\eqref{eq:56} by $\mu$ and adding~\eqref{eq:55} multiplied
by 2 we obtain
\begin{equation*}
\begin{split} &-\mu^2\bar\mu(\lambda+ \bar\lambda)
\psi'''(x) + \bigl( \mu(\lambda\mu -\bar
\lambda\bar\mu) -2\mu\bar\mu(\lambda+ \bar\lambda) \bigr)\psi''(x)
+2(\lambda\mu-\bar\lambda\bar\mu)\psi'(x)
\\
&\quad=\frac{\lambda\bar\lambda\mu\bar\mu\theta}{2\lambda+ \bar
\lambda}\, \psi''(x) +
\frac{\lambda\bar\lambda\mu\theta}{2\lambda+ \bar\lambda}\, \psi '(x), \quad x\in[0,\infty), \end{split}
\end{equation*}
from which~\eqref{eq:49} follows.

If $\theta=0$ and $\bar\theta\neq0$, then~\eqref{eq:54} takes the form
\begin{equation}
\label{eq:57} %
\begin{split} &-\mu\bar\mu(\lambda+ \bar\lambda)
\psi''(x) +(\lambda\mu-\bar\lambda\bar \mu)
\psi'(x) \\
&\quad=- \biggl(1+\frac{2\mu}{\bar\mu} \biggr) \frac{\lambda\bar\lambda\bar
\theta}{\lambda+ 2\bar\lambda}\,
I_{16}(x)
+\frac{\lambda\bar\lambda\mu\bar\theta}{\lambda+ 2\bar\lambda
} \psi'(x) \\
&\qquad+ \biggl(1+\frac{2\mu}{\bar\mu}
\biggr) \frac{\lambda\bar\lambda\bar
\theta}{\lambda+ 2\bar\lambda}\, \psi(x), \quad x\in[0,\infty). \end{split}
\end{equation}

From~\eqref{eq:57}, it follows that $\psi(x)$ has the third derivative
on $[0,\infty)$.
Differentiating~\eqref{eq:57} gives\vadjust{\goodbreak}
\begin{align}
%\begin{split}
&-\mu\bar\mu(\lambda+ \bar\lambda)
\psi'''(x) +(\lambda\mu-\bar\lambda \bar
\mu)\psi''(x)\notag\\
&\quad =-\frac{2}{\bar\mu}\, \biggl(1+
\frac{2\mu}{\bar\mu} \biggr) \frac
{\lambda\bar\lambda\bar\theta}{\lambda+ 2\bar\lambda}\, I_{16}(x)
+\frac{\lambda\bar\lambda\mu\bar\theta}{\lambda+ 2\bar\lambda
} \psi''(x) + \biggl(1+
\frac{2\mu}{\bar\mu} \biggr) \frac{\lambda\bar\lambda\bar
\theta}{\lambda+ 2\bar\lambda}\, \psi'(x)
\notag\\
&\qquad+\frac{2}{\bar\mu}\, \biggl(1+\frac{2\mu}{\bar\mu} \biggr)
\frac
{\lambda\bar\lambda\bar\theta}{\lambda+ 2\bar\lambda}\, \psi(x), \quad x\in[0,\infty).\label{eq:58}
%\end{split}
\end{align}

Multiplying~\eqref{eq:58} by $(-\bar\mu)$ and adding~\eqref{eq:57}
multiplied by 2 we get
\begin{equation*}
\begin{split} &\mu\bar\mu^2(\lambda+ \bar\lambda)
\psi'''(x) - \bigl( \bar\mu(\lambda\mu -
\bar\lambda\bar\mu) +2\mu\bar\mu(\lambda+ \bar\lambda) \bigr)\psi''(x)
+2(\lambda\mu-\bar\lambda\bar\mu)\psi'(x)
\\
&\quad=-\frac{\lambda\bar\lambda\mu\bar\mu\bar\theta}{\lambda+ 2\bar
\lambda}\, \psi''(x) -
\frac{\lambda\bar\lambda\bar\mu\bar\theta}{\lambda+ 2\bar\lambda}\, \psi'(x), \quad x\in[0,\infty), \end{split}
\end{equation*}
from which~\eqref{eq:50} follows.
\end{proof}

To formulate the next theorem, we define the following constants:
\begin{align*}
%\begin{split}
\mathrm{D}_1&= \bigl( \mu\bar\mu(2
\lambda+3\bar\lambda) (2\lambda +\bar\lambda) -\lambda\mu^2 (2\lambda+
\bar\lambda) +\lambda\bar\lambda\mu\bar\mu\theta \bigr)^2
\\
&\quad-4\mu^2 \bar\mu(\lambda+\bar\lambda) (2\lambda+\bar\lambda)
\bigl( 2(\bar\lambda\bar\mu-\lambda\mu) (2\lambda+\bar\lambda) +\lambda\bar\lambda
\mu\theta \bigr), \\%\end{split}\end{equation*}
%\[
z_2&=\frac{- ( \mu\bar\mu(2\lambda+3\bar\lambda) (2\lambda+\bar
\lambda) -\lambda\mu^2 (2\lambda+\bar\lambda)
+\lambda\bar\lambda\mu\bar\mu\theta ) +\sqrt{\mathrm{D}_1}}{
2\mu^2 \bar\mu(\lambda+\bar\lambda) (2\lambda+\bar\lambda)},\\
z_3&=\frac{- ( \mu\bar\mu(2\lambda+3\bar\lambda) (2\lambda+\bar
\lambda) -\lambda\mu^2 (2\lambda+\bar\lambda)
+\lambda\bar\lambda\mu\bar\mu\theta ) -\sqrt{\mathrm{D}_1}}{
2\mu^2 \bar\mu(\lambda+\bar\lambda) (2\lambda+\bar\lambda)}
\end{align*}
and
\begin{equation*}
\begin{split} \Delta_1&= \biggl( \lambda+\bar\lambda-
\frac{\bar\lambda}{1-\bar\mu
z_2} \biggr) \biggl( \mu(\lambda+\bar\lambda)z_3 -
\biggl(\bar\lambda+\frac{\bar
\lambda\mu}{\bar\mu} \biggr) \frac{\bar\mu z_3}{1-\bar\mu z_3} -
\frac{\lambda\bar\lambda\theta}{2\lambda+\bar\lambda} \biggr)
\\
&\quad- \biggl( \lambda+\bar\lambda-\frac{\bar\lambda}{1-\bar\mu z_3} \biggr) \biggl( \mu(
\lambda+\bar\lambda)z_2 - \biggl(\bar\lambda+\frac{\bar
\lambda\mu}{\bar\mu}
\biggr) \frac{\bar\mu z_2}{1-\bar\mu z_2} -\frac{\lambda\bar\lambda\theta}{2\lambda+\bar\lambda} \biggr). \end{split} %
\end{equation*}

\begin{thm}
\label{thm:3}
Let the surplus process $ (X_t(x) )_{t\ge0}$ follow~\eqref
{eq:1} under the above assumptions with $\theta\neq0$ and $\bar\theta=0$.
Moreover, let claim and premium sizes be exponentially distributed with
means $\mu$ and $\bar\mu$, respectively, and let
$\bar\lambda\bar\mu>\lambda\mu$.

If $2(\bar\lambda\bar\mu-\lambda\mu) (2\lambda+\bar\lambda) +\lambda
\bar\lambda\mu\theta\le0$, then
\begin{equation}
\label{eq:59} \psi(x)=\frac{\lambda(1-\bar\mu z_3)}{\lambda(1-\bar\mu z_3) -\lambda
\mu z_3}\, e^{z_3 x}, \quad x\in[0,
\infty).
\end{equation}

If $2(\bar\lambda\bar\mu-\lambda\mu) (2\lambda+\bar\lambda) +\lambda
\bar\lambda\mu\theta>0$, then
\begin{equation}
\label{eq:60} \psi(x)=C_2 e^{z_2 x} +C_3
e^{z_3 x}, \quad x\in[0,\infty),
\end{equation}
where the constants $C_2$ and $C_3$ are determined from the system of
linear equations
\begin{equation}
\label{eq:61} \biggl( \lambda+\bar\lambda-\frac{\bar\lambda}{1-\bar\mu z_2} \biggr)
C_2 + \biggl( \lambda+\bar\lambda-\frac{\bar\lambda}{1-\bar\mu z_3} \biggr)
C_3 =\lambda
\end{equation}
and
\begin{equation}
\label{eq:62} %
\begin{split} & \biggl( \mu(\lambda+\bar
\lambda)z_2 - \biggl(\bar\lambda+\frac{\bar
\lambda\mu}{\bar\mu} \biggr)
\frac{\bar\mu z_2}{1-\bar\mu z_2} -\frac{\lambda\bar\lambda\theta}{2\lambda+\bar\lambda} \biggr) C_2
\\
&\quad+ \biggl( \mu(\lambda+\bar\lambda)z_3 - \biggl(\bar\lambda+
\frac
{\bar\lambda\mu}{\bar\mu} \biggr) \frac{\bar\mu z_3}{1-\bar\mu z_3} -\frac{\lambda\bar\lambda\theta}{2\lambda+\bar\lambda} \biggr)
C_3 =-\frac{\lambda\bar\lambda\theta}{2\lambda+\bar\lambda} \end{split} %
\end{equation}
provided that $\Delta_1 \neq0$.
\end{thm}

\begin{proof}
By Lemma~\ref{lem:1}, $\psi(x)$ is a solution to~\eqref{eq:49}.
We now find the general solution to~\eqref{eq:49}. Its characteristic
equation is
\begin{align}
%\begin{split}
&\mu^2 \bar\mu(\lambda+\bar
\lambda) (2\lambda+\bar\lambda) z^3 + \bigl( \mu\bar\mu(2\lambda+3\bar
\lambda) (2\lambda+\bar\lambda) -\lambda\mu^2 (2\lambda+\bar\lambda) +
\lambda\bar\lambda\mu\bar\mu\theta \bigr) z^2\notag
\\
&\quad+ \bigl( 2(\bar\lambda\bar\mu-\lambda\mu) (2\lambda+\bar\lambda ) +\lambda
\bar\lambda\mu\theta \bigr) z =0.\label{eq:63} %\end{split}
\end{align}

It is evident that $z_1=0$ is a solution to~\eqref{eq:63}.
Next, we prove that the equation
\begin{align}
%\begin{split}
&\mu^2 \bar\mu(\lambda+\bar
\lambda) (2\lambda+\bar\lambda) z^2 + \bigl( \mu\bar\mu(2\lambda+3\bar
\lambda) (2\lambda+\bar\lambda) -\lambda\mu^2 (2\lambda+\bar\lambda) +
\lambda\bar\lambda\mu\bar\mu\theta \bigr) z\notag
\\
&\quad+ \bigl( 2(\bar\lambda\bar\mu-\lambda\mu) (2\lambda+\bar\lambda ) +\lambda
\bar\lambda\mu\theta \bigr)=0 \label{eq:64}%\end{split}
\end{align}
has two real roots.
To this end, we show that its discriminant $\mathrm{D}_1$ defined
before the assertion of the theorem is positive.
We have
\begin{equation*}
\begin{split} \mathrm{D}_1/\mu^2 &= \bigl(
(2\lambda\bar\mu+3\bar\lambda\bar\mu -\lambda\mu) (2\lambda+\bar\lambda) +\lambda
\bar\lambda\bar\mu\theta \bigr)^2
\\
&\quad-8\bar\mu(\lambda+\bar\lambda) (\bar\lambda\bar\mu-\lambda\mu ) (2\lambda+
\bar\lambda)^2 -4\lambda\bar\lambda\mu\bar\mu\theta(\lambda+\bar
\lambda) (2\lambda +\bar\lambda)
\\
&=(2\lambda+\bar\lambda)^2 \bigl( (2\lambda\bar\mu+3\bar\lambda\bar
\mu-\lambda\mu)^2 -8\bar\mu(\lambda+\bar\lambda) (\bar\lambda\bar\mu-
\lambda\mu) \bigr)
\\
&\quad+2\lambda\bar\lambda\bar\mu\theta(2\lambda+\bar\lambda) \bigl( 2\lambda
\bar\mu+3\bar\lambda\bar\mu-\lambda\mu-2\mu(\lambda +\bar\lambda) \bigr) +(\lambda
\bar\lambda\bar\mu\theta)^2
\\
&=(2\lambda+\bar\lambda)^2 \bigl( 2\lambda(\mu+\bar\mu) +(\bar
\lambda \bar\mu-\lambda\mu) \bigr)^2
\\
&\quad+2\lambda\bar\lambda\bar\mu\theta(2\lambda+\bar\lambda) \bigl( 2(\bar\mu-
\mu) (\lambda+\bar\lambda) +(\bar\lambda\bar\mu+\lambda \mu) \bigr) +(\lambda\bar
\lambda\bar\mu\theta)^2
\\
&=(2\lambda+\bar\lambda)^2 (\bar\lambda\bar\mu-\lambda
\mu)^2 +2\lambda\bar\lambda\bar\mu\theta(2\lambda+\bar\lambda) (\bar
\lambda \bar\mu-\lambda\mu) +(\lambda\bar\lambda\bar\mu\theta)^2
\\
&\quad+(2\lambda+\bar\lambda)^2 \bigl( 4\lambda^2 (\mu+
\bar\mu)^2 +4\lambda(\mu+\bar\mu) (\bar\lambda\bar\mu-\lambda\mu)
\bigr)
\\
&\quad+4\lambda\bar\lambda\bar\mu\theta(\bar\mu-\mu) (\lambda+\bar \lambda) (2
\lambda+\bar\lambda)
\\
&=(2\lambda+\bar\lambda)^2 (\bar\lambda\bar\mu-\lambda
\mu)^2 +2\lambda\bar\lambda\bar\mu\theta(2\lambda+\bar\lambda) (\bar
\lambda \bar\mu-\lambda\mu) +(\lambda\bar\lambda\bar\mu\theta)^2
\\
&\quad+4\lambda(2\lambda+\bar\lambda) \bigl( \lambda^2 (\mu+\bar\mu
)^2 +\lambda(\mu+\bar\mu) (\bar\lambda\bar\mu-\lambda\mu)
\\
&\qquad+(\lambda+\bar\lambda) \bigl( \lambda(\mu+\bar\mu)^2 +(\mu +
\bar\mu) (\bar\lambda\bar\mu-\lambda\mu) +\bar\lambda\bar\mu\theta(\bar\mu-\mu)
\bigr) \bigr). \end{split} %
\end{equation*}

Since $\bar\lambda\bar\mu>\lambda\mu$, it suffices to show that
\[
\lambda(\mu+\bar\mu)^2 +(\mu+\bar\mu) (\bar\lambda\bar\mu-\lambda\mu
) +\bar\lambda\bar\mu\theta(\bar\mu-\mu) )>0.
\]

It is obvious that the minimal value of the expression on the left-hand
side of the above inequality is attained
when either $\theta=1$ or $\theta=-1$ and equals either $\lambda\bar\mu
(\mu+\bar\mu) +2\lambda\bar\mu^2$ or
$\lambda\bar\mu(\mu+\bar\mu) +2\bar\lambda\mu\bar\mu$, respectively.
Both these expressions are positive.
Thus, $\mathrm{D}_1>0$ and~\eqref{eq:64} has two real roots $z_2$ and
$z_3$ defined before the assertion of the theorem.

Next, it is easily seen that
\[
\mu\bar\mu(2\lambda+3\bar\lambda) (2\lambda+\bar\lambda) -\lambda
\mu^2 (2\lambda+\bar\lambda) +\lambda\bar\lambda\mu\bar\mu\theta>0.
\]

Indeed, since $\bar\lambda\bar\mu>\lambda\mu$, we have
\begin{equation*}
\begin{split} &\mu\bar\mu(2\lambda+3\bar\lambda) (2\lambda+\bar
\lambda) -\lambda \mu^2 (2\lambda+\bar\lambda) +\lambda\bar\lambda\mu
\bar\mu\theta
\\
&\quad=\mu(2\lambda+\bar\lambda) (2\lambda\bar\mu+3\bar\lambda\bar \mu-\lambda
\mu) +\lambda\bar\lambda\mu\bar\mu\theta
\\
&\quad>2\mu\bar\mu(2\lambda+\bar\lambda) (\lambda+\bar\lambda) -\lambda\bar
\lambda\mu\bar\mu >3\lambda\bar\lambda\mu\bar\mu>0. \end{split} %
\end{equation*}

Therefore, by Vieta's theorem applied to~\eqref{eq:64}, $z_3<0$.
Moreover, $z_2\ge0$ if
$2(\bar\lambda\bar\mu-\lambda\mu) (2\lambda+\bar\lambda) +\lambda
\bar\lambda\mu\theta\le0$,
and $z_2<0$ if $2(\bar\lambda\bar\mu-\lambda\mu) (2\lambda+\bar
\lambda) +\lambda\bar\lambda\mu\theta>0$.

Thus, the general solution to~\eqref{eq:49} has the form
\begin{equation}
\label{eq:65} \psi(x)=C_1 +C_2 e^{z_2 x}
+C_3 e^{z_3 x}, \quad x\in[0,\infty),
\end{equation}
where $C_1$, $C_2$ and $C_3$ are some constants. To determine them, we
use the following boundary conditions.
Firstly, since $\bar\lambda\bar\mu>\lambda\mu$, using standard
considerations (see, e.g., \cite{MiRa2016,MiRaSt2015,RoScScTe1999})
we can easily show that $\lim_{x\to\infty} \psi(x) =0$. Consequently,
$C_1=0$ and $C_2=0$ if
$2(\bar\lambda\bar\mu-\lambda\mu) (2\lambda+\bar\lambda) +\lambda
\bar\lambda\mu\theta\le0$,
and $C_1=0$ if $2(\bar\lambda\bar\mu-\lambda\mu) (2\lambda+\bar
\lambda) +\lambda\bar\lambda\mu\theta>0$.
Secondly, we use intermediate equations to find other constants.

Let now $2(\bar\lambda\bar\mu-\lambda\mu) (2\lambda+\bar\lambda)
+\lambda\bar\lambda\mu\theta\le0$.
The constant $C_3$ is determined by letting $x=0$ in~\eqref{eq:48},
i.e. from the equation
\begin{equation}
\label{eq:66} (\lambda+\bar\lambda)\psi(0)= \frac{\bar\lambda}{\bar\mu}\, \int
_0^{\infty} \psi(u) e^{-u/\bar\mu}\, \mathrm{d}u +
\lambda.
\end{equation}

Substituting $\psi(x)=C_3 e^{z_3 x}$ into~\eqref{eq:66} gives
\[
(\lambda+\bar\lambda)C_3= \frac{\bar\lambda}{1-\bar\mu z_3}\, C_3 +
\lambda,
\]
from which~\eqref{eq:59} follows immediately.

If $2(\bar\lambda\bar\mu-\lambda\mu) (2\lambda+\bar\lambda) +\lambda
\bar\lambda\mu\theta>0$, then the constants $C_2$ and $C_3$
are determined by letting $x=0$ in~\eqref{eq:48} and~\eqref{eq:52},
i.e. from~\eqref{eq:66} and the equation
\begin{equation}
\label{eq:67} %
\begin{split} &\mu(\lambda+\bar\lambda)
\psi'(0) +\bar\lambda\psi(0)
\\
&\quad=\frac{\bar\lambda}{\bar\mu}\, \biggl( 1+\frac{\mu}{\bar\mu} \biggr) \int
_0^{\infty} \psi(u) e^{-u/\bar\mu}\, \mathrm{d}u +
\biggl( \frac{\lambda\bar\lambda\theta}{2\lambda+\bar\lambda} -\frac
{\bar\lambda\mu}{\bar\mu} \biggr) \psi(0) -
\frac{\lambda\bar\lambda\theta}{2\lambda+\bar\lambda}. \end{split} %
\end{equation}

Substituting~\eqref{eq:60} into~\eqref{eq:66} and~\eqref{eq:67} yields
equations~\eqref{eq:61} and~\eqref{eq:62},
respectively. The system of equations~\eqref{eq:61} and~\eqref{eq:62}
has a unique solution provided that $\Delta_1 \neq0$.

Note that letting $x=0$ in~\eqref{eq:54} (and in~\eqref{eq:52} when
$C_2=0$) gives no additional information about unknown constants.
Nevertheless, the equalities must hold for the values of the constants
found from~\eqref{eq:48} (and~\eqref{eq:52} when $C_2 \neq0$).
Consequently, differential equation~\eqref{eq:49} has the unique
solution given by~\eqref{eq:59} or~\eqref{eq:60}.
Since we have derived~\eqref{eq:49} from~\eqref{eq:48} without any
additional assumptions, we conclude that the function $\psi(x)$ given
by~\eqref{eq:59} or~\eqref{eq:60} is a unique solution to~\eqref{eq:48}
satisfying the certain conditions.
This guaranties that the solution we have found is the ruin probability
and completes the proof.
\end{proof}

The case $\theta=0$ and $\bar\theta\neq0$ can be considered in a
similar way by finding the required solution to equation~\eqref{eq:50}.

\subsection{The ruin probability in the model without dependence}
\label{sec:5.2}

Let now $\theta=0$ and $\bar\theta=0$. We set
\begin{equation*}
\psi(x)= %
\begin{cases}
\psi_1(x) &\text{if} \ x\in[0,b], \\
\psi_2(x) &\text{if} \ x\in[b,\infty).
\end{cases} %
\end{equation*}
Then equations~\eqref{eq:9} and~\eqref{eq:26} for the ruin probability
$\psi(x)$
in the case of exponentially distributed claim and premium sizes take
the form
\begin{equation}
\label{eq:68} %
\begin{split} &(\lambda+\bar\lambda)
\psi_1(x)\\
&\quad =\frac{\lambda}{\mu}\, e^{-x/\mu} \int
_0^x \psi_1(u) e^{u/\mu}\,
\mathrm{d}u +\lambda e^{-x/\mu}
+\frac{\bar\lambda}{\bar\mu}\, e^{x/\bar\mu} \int_x^b
\psi_1(u) e^{-u/\bar\mu}\, \mathrm{d}u\\
&\qquad +\frac{\bar\lambda}{\bar\mu}\,
e^{x/\bar\mu} \int_b^{\infty} \psi_2(u)
e^{-u/\bar\mu}\, \mathrm{d}u, \quad x\in[0,b], \end{split} %
\end{equation}
and
\begin{equation}
\label{eq:69} %
\begin{split} &d\psi'_2(x) +(
\lambda+\bar\lambda)\psi_2(x)\\
&\quad =\frac{\lambda}{\mu}\, e^{-x/\mu}
\int_0^b \psi_1(u)
e^{u/\mu}\, \mathrm{d}u
+\frac{\lambda}{\mu}\, e^{-x/\mu} \int_b^x
\psi_2(u) e^{u/\mu}\, \mathrm{d}u +\lambda e^{-x/\mu}
\\
&\qquad+\frac{\bar\lambda}{\bar\mu}\, e^{x/\bar\mu} \int_x^{\infty}
\psi _2(u) e^{-u/\bar\mu}\, \mathrm{d}u, \quad x\in[b,\infty),
\end{split} %
\end{equation}
respectively.

We now show that integro-differential equations~\eqref{eq:68} and~\eqref
{eq:69} can be reduced to linear differential equations
with constant coefficients.

\begin{lemma}
\label{lem:2}
Let the surplus process $ (X_t^b(x) )_{t\ge0}$ follow~\eqref
{eq:6} under the above assumptions with $\theta=0$ and $\bar\theta=0$,
and let claim and premium sizes be exponentially distributed with means
$\mu$ and $\bar\mu$, respectively.
Then $\psi_1(x)$ and $\psi_2(x)$ are solutions to the differential equations
\begin{equation}
\label{eq:70} \mu\bar\mu(\lambda+\bar\lambda)\psi''_1(x)
+(\bar\lambda\bar\mu -\lambda\mu)\psi'_1(x)=0, \quad x
\in[0,b],
\end{equation}
and
\begin{equation}
\label{eq:71} %
\begin{split} &d\mu\bar\mu\psi'''_2(x)
+ \bigl( d\bar\mu-d\mu+\mu\bar\mu(\lambda +\bar\lambda) \bigr)
\psi''_2(x)
\\
&\quad+(\bar\lambda\bar\mu-\lambda\mu-d)\psi'_2(x)=0,
\quad x\in [b,\infty). \end{split} %
\end{equation}
\end{lemma}

\begin{proof}
From~\eqref{eq:68}, it is easily seen that $\psi(x)$ is differentiable
on $[0,b]$.
Differentiating~\eqref{eq:68} yields
\begin{equation}
\label{eq:72} %
\begin{split} &(\lambda+\bar\lambda)
\psi'_1(x)\\
&\quad =-\frac{1}{\mu} \Biggl(
\frac{\lambda
}{\mu}\, e^{-x/\mu} \int_0^x
\psi_1(u) e^{u/\mu}\, \mathrm{d}u +\lambda e^{-x/\mu}
\Biggr)
\\
&\qquad+\frac{1}{\bar\mu} \Biggl( \frac{\bar\lambda}{\bar\mu}\, e^{x/\bar\mu} \int
_x^b \psi_1(u) e^{-u/\bar\mu}\,
\mathrm{d}u +\frac{\bar\lambda}{\bar\mu}\, e^{x/\bar\mu} \int_b^{\infty}
\psi_2(u) e^{-u/\bar\mu}\, \mathrm{d}u \Biggr)
\\
&\qquad+ \biggl( \frac{\lambda}{\mu} -\frac{\bar\lambda}{\bar\mu} \biggr)
\psi_1(x), \quad x\in[0,b]. \end{split} %
\end{equation}

Multiplying~\eqref{eq:72} by $(-\bar\mu)$ and adding~\eqref{eq:68} we obtain
\begin{equation}
\label{eq:73} %
\begin{split} &-\bar\mu(\lambda+\bar\lambda)
\psi'_1(x) +\lambda \biggl(1+\frac{\bar
\mu}{\mu}
\biggr)\psi_1(x)
\\
&\quad= \biggl(1+\frac{\bar\mu}{\mu} \biggr) \Biggl( \frac{\lambda}{\mu
}\,
e^{-x/\mu} \int_0^x
\psi_1(u) e^{u/\mu}\, \mathrm{d}u +\lambda e^{-x/\mu}
\Biggr), \quad x\in[0,b]. \end{split} %
\end{equation}

From~\eqref{eq:73}, it follows that $\psi(x)$ is twice differentiable
on $[0,b]$.
Differentiating~\eqref{eq:73} gives
\begin{equation}
\label{eq:74} %
\begin{split} &-\bar\mu(\lambda+\bar\lambda)
\psi''_1(x) +\lambda \biggl(1+
\frac{\bar
\mu}{\mu} \biggr)\psi'_1(x)
\\
&\quad=-\frac{1}{\mu} \biggl(1+\frac{\bar\mu}{\mu} \biggr) \Biggl(
\frac
{\lambda}{\mu}\, e^{-x/\mu} \int_0^x
\psi_1(u) e^{u/\mu}\, \mathrm{d}u +\lambda e^{-x/\mu}
\Biggr)
\\
&\qquad+\frac{\lambda}{\mu} \biggl(1+\frac{\bar\mu}{\mu} \biggr)\psi
_1(x), \quad x\in[0,b]. \end{split} %
\end{equation}

Multiplying~\eqref{eq:74} by $\mu$ and adding~\eqref{eq:73} we
get~\eqref{eq:70}.

From~\eqref{eq:69}, it is easily seen that $\psi(x)$ is twice
differentiable on $[b,\infty)$.
Differentiating~\eqref{eq:69} yields
\begin{equation}
\label{eq:75} %
\begin{split} &d\psi''_2(x)
+(\lambda+\bar\lambda)\psi'_2(x)
\\
&\quad=-\frac{1}{\mu} \Biggl( \frac{\lambda}{\mu}\, e^{-x/\mu} \int
_0^b \psi_1(u) e^{u/\mu}\,
\mathrm{d}u +\frac{\lambda}{\mu}\, e^{-x/\mu} \int_b^x
\psi_2(u) e^{u/\mu}\, \mathrm {d}u +\lambda e^{-x/\mu}
\Biggr)
\\
&\qquad+\frac{\bar\lambda}{\bar\mu^2}\, e^{x/\bar\mu} \int_x^{\infty}
\psi_2(u) e^{-u/\bar\mu}\, \mathrm{d}u + \biggl(
\frac{\lambda}{\mu} -\frac{\bar\lambda}{\bar\mu} \biggr)\psi _2(x), \quad x
\in[b,\infty). \end{split} %
\end{equation}

Multiplying~\eqref{eq:75} by $\mu$ and adding~\eqref{eq:69} we obtain
\begin{equation}
\label{eq:76} %
\begin{split} &d\mu\psi''_2(x)
+ \bigl(d+\mu(\lambda+\bar\lambda) \bigr)\psi'_2(x) +
\bar\lambda \biggl(1+\frac{\mu}{\bar\mu} \biggr)\psi_2(x)
\\
&\quad=\frac{\bar\lambda}{\bar\mu} \biggl(1+\frac{\mu}{\bar\mu} \biggr) e^{x/\bar\mu}
\int_x^{\infty} \psi_2(u) e^{-u/\bar\mu}
\, \mathrm{d}u, \quad x\in[b,\infty). \end{split} %
\end{equation}

From~\eqref{eq:76}, it follows that $\psi(x)$ has the third derivative
on $[b,\infty)$.
Differentiating~\eqref{eq:76} gives
\begin{equation}
\label{eq:77} %
\begin{split} &d\mu\psi'''_2(x)
+ \bigl(d+\mu(\lambda+\bar\lambda) \bigr)\psi''_2(x)
+\bar\lambda \biggl(1+\frac{\mu}{\bar\mu} \biggr)\psi'_2(x)
\\
&\; =\frac{\bar\lambda}{\bar\mu^2} \biggl(1+\frac{\mu}{\bar\mu} \biggr) e^{x/\bar\mu} \!
\int_x^{\infty} \psi_2(u) e^{-u/\bar\mu}
\, \mathrm{d}u -\frac{\bar\lambda}{\bar\mu} \biggl(1+\frac{\mu}{\bar\mu} \biggr)\psi
_2(x), \; x\in[b,\infty). \end{split} %
\end{equation}

Multiplying~\eqref{eq:77} by $(-\bar\mu)$ and adding~\eqref{eq:76} we
get~\eqref{eq:71}.
\end{proof}

To formulate the next theorem, we define the following constants:
\begin{align*}
\mathrm{D}_2&= \bigl( d\bar\mu-d\mu+\mu\bar\mu(\lambda+\bar\lambda)
\bigr)^2 -4d\mu\bar\mu(\bar\lambda\bar\mu-\lambda\mu-d),\\
z_5&=\frac{\lambda\mu-\bar\lambda\bar\mu}{\mu\bar\mu(\lambda+\bar
\lambda)},\\
z_7&=\frac{- ( d\bar\mu-d\mu+\mu\bar\mu(\lambda+\bar\lambda)
 ) +\sqrt{\mathrm{D}_2}}{2d\mu\bar\mu}
\end{align*}
and
\[
z_8=\frac{- ( d\bar\mu-d\mu+\mu\bar\mu(\lambda+\bar\lambda)
 ) -\sqrt{\mathrm{D}_2}}{2d\mu\bar\mu}.
\]

\begin{thm}
\label{thm:4}
Let the surplus process $ (X_t^b(x) )_{t\ge0}$ follow~\eqref
{eq:6} under the above assumptions with $\theta=0$ and $\bar\theta=0$,
and let claim and premium sizes be exponentially distributed with means
$\mu$ and $\bar\mu$, respectively, and let
$\bar\lambda\bar\mu>\lambda\mu+d$.
Then we have
\begin{equation}
\label{eq:78} \psi_1(x)=C_4 +C_5
e^{z_5 x}, \quad x\in[0,b],
\end{equation}
and
\begin{equation}
\label{eq:79} \psi_2(x)=C_7 e^{z_7 x}
+C_8 e^{z_8 x}, \quad x\in[b,\infty),
\end{equation}
where the constants $C_4$, $C_5$, $C_7$ and $C_8$ are determined from
the system of linear equations~\eqref{eq:80}--\eqref{eq:83}:
\begin{equation}
\label{eq:80} %\begin{split}
 \bigl( \lambda e^{b/\bar\mu} +\bar
\lambda \bigr) C_4 + \frac{\lambda
+\bar\lambda}{\mu+\bar\mu} \bigl( \bar\mu
e^{b/\bar\mu} +\mu e^{z_5 b} \bigr) C_5
+\frac{\bar\lambda e^{z_7 b}}{\bar\mu z_7 -1}\, C_7 +\frac{\bar
\lambda e^{z_8 b}}{\bar\mu z_8 -1}\,
C_8 =\lambda e^{b/\bar\mu}, %\end{split}
\end{equation}
\begin{equation}
\label{eq:81} \lambda \biggl(1+\frac{\bar\mu}{\mu} \biggr) C_4 +
\frac{\bar\mu(\lambda+\bar\lambda)}{\mu}\, C_5 =\lambda \biggl(1+\frac{\bar\mu}{\mu}
\biggr),
\end{equation}
\begin{equation}
\label{eq:82} C_4 +e^{z_5 b} C_5
-e^{z_7 b} C_7 -e^{z_8 b} C_8=0
\end{equation}
and
\begin{equation}
\label{eq:83} %
\begin{split} &\lambda \bigl(e^{-b/\mu}-1
\bigr) C_4 + \frac{\bar\mu(\lambda+\bar
\lambda)}{\mu+\bar\mu} \bigl( e^{-b/\mu}
-e^{z_5 b} \bigr) C_5
\\
&\quad+ \biggl( \lambda+\bar\lambda+dz_7 +\frac{\bar\lambda}{\bar\mu
z_7 -1}
\biggr) e^{z_7 b} C_7
\\
&\quad+ \biggl( \lambda+\bar\lambda+dz_8 +\frac{\bar\lambda}{\bar\mu
z_8 -1}
\biggr) e^{z_8 b} C_8 =\lambda e^{-b/\mu}. \end{split}
\end{equation}
\end{thm}

\begin{proof}
By Lemma~\ref{lem:2}, $\psi_1(x)$ and $\psi_2(x)$ are solutions
to~\eqref{eq:70} and~\eqref{eq:71}.
We now find the general solutions to these equations.

It is easily seen that the characteristic equation corresponding
to~\eqref{eq:70} has two roots: $z_4=0$ and $z_5$
given before the assertion of the theorem. Hence, \eqref{eq:78} is true
with some constants $C_4$ and $C_5$.

The characteristic equation corresponding to~\eqref{eq:71} has the form
\begin{equation}
\label{eq:84} %
\begin{split} &d\mu\bar\mu z^3 + \bigl( d
\bar\mu-d\mu+\mu\bar\mu(\lambda+\bar \lambda) \bigr) z^2 +(\bar\lambda
\bar\mu-\lambda\mu-d) z =0. \end{split} %
\end{equation}

It is obvious that $z_6=0$ is a solution to~\eqref{eq:84}.
We now show that the equation
\begin{equation}
\label{eq:85} %
\begin{split} &d\mu\bar\mu z^2 + \bigl( d
\bar\mu-d\mu+\mu\bar\mu(\lambda+\bar \lambda) \bigr) z +(\bar\lambda\bar\mu-
\lambda\mu-d)=0. \end{split} %
\end{equation}
has two negative roots. We first notice that its discriminant $\mathrm
{D}_2$ defined above is positive.
Indeed, we have
\begin{equation*}
\begin{split} \mathrm{D}_2&=d^2 (\bar\mu-
\mu)^2 \,{+}\,\mu^2 \bar\mu^2 (\lambda\,{+}\,\bar\lambda
)^2 \,{+}\,2d\mu\bar\mu(\lambda\,{+}\,\bar\lambda) (\bar\mu-\mu) \,{+}\,4d\mu\bar\mu(d
\,{+}\,\lambda\mu-\bar\lambda\bar\mu)
\\
&=d^2 (\mu+\bar\mu)^2 +\mu^2 \bar
\mu^2 (\lambda+\bar\lambda)^2 +2d\mu\bar\mu(\lambda-\bar
\lambda) (\mu+\bar\mu)
\\
&= \bigl( d(\mu+\bar\mu) +\mu\bar\mu(\lambda-\bar\lambda)
\bigr)^2 +4\lambda\bar\lambda\mu^2 \bar
\mu^2>0. \end{split} %
\end{equation*}

Therefore, \eqref{eq:85} has two real roots. Next, by the conditions of
the theorem, we have
\[
\bar\lambda\bar\mu-\lambda\mu-d >0
\]
and
\[
d\bar\mu-d\mu+\mu\bar\mu(\lambda+\bar\lambda) =\mu(\bar\lambda\bar \mu-\lambda
\mu-d) +\lambda\mu^2 +\lambda\mu\bar\mu+d\bar\mu>0,
\]
which shows that both roots are negative. Consequently, we get
\[
\psi_2(x)=C_6 +C_7 e^{z_7 x}
+C_8 e^{z_8 x}, \quad x\in[b,\infty),
\]
with some constants $C_6$, $C_7$ and $C_8$. Moreover, since $\bar\lambda
\bar\mu>\lambda\mu+d$, using standard considerations
(see, e.g., \cite{MiRa2016,MiRaSt2015,RoScScTe1999}) we can easily
show that $\lim_{x\to\infty} \psi(x) =0$, which yields $C_6=0$.
Thus, we obtain~\eqref{eq:79}.

The constants $C_4$, $C_5$, $C_7$ and $C_8$ are determined by letting
$x=0$ in~\eqref{eq:68} and~\eqref{eq:73},
taking into account that $\psi_1(b)=\psi_2(b)$ and letting $x=b$
in~\eqref{eq:69}.

Substituting~\eqref{eq:78} and~\eqref{eq:79} into~\eqref{eq:68}
and~\eqref{eq:73} as $x=0$ and into~\eqref{eq:69} as $x=b$
and doing some simplifications yield equations~\eqref{eq:80}, \eqref
{eq:81} and~\eqref{eq:83}, respectively.
Substituting~\eqref{eq:78} and~\eqref{eq:79} into the equality $\psi
_1(b)=\psi_2(b)$ gives~\eqref{eq:82}.

We denote the determinant of the system of equations~\eqref
{eq:80}--\eqref{eq:83} by $\Delta_2$.
A standard computation shows that
\begin{equation*}
\begin{split} \Delta_2&=d(z_7-z_8)
e^{(z_7+z_8)b} \biggl( \lambda e^{z_5 b} \biggl( \frac{\bar\lambda\bar\mu}{\mu} -
\lambda \biggr)
\\
&\quad+\frac{\bar\mu^2 z_7 z_8}{\bar\mu^2 z_7 z_8 -\bar\mu(z_7+z_8)+1} \biggl( (\lambda+\bar\lambda) \frac{\bar\lambda\bar\mu}{\mu} -
\lambda \bar\lambda e^{z_5 b} \biggl(1+\frac{\bar\mu}{\mu} \biggr) \biggr)
\biggr), \end{split} %
\end{equation*}
which is positive. Indeed, $z_7-z_8>0$ by definition, $\bar\lambda\bar
\mu/\mu-\lambda>0$ by the conditions of the theorem and
$\bar\mu^2 z_7 z_8 -\bar\mu(z_7+z_8)+1>0$ since $z_7<0$ and $z_8<0$.
Moreover, since $z_5<0$, we have
\[
(\lambda+\bar\lambda) \frac{\bar\lambda\bar\mu}{\mu} -\lambda\bar \lambda e^{z_5 b}
\biggl(1+\frac{\bar\mu}{\mu} \biggr) >(\lambda+\bar\lambda) \frac{\bar\lambda\bar\mu}{\mu} -
\lambda\bar \lambda \biggl(1+\frac{\bar\mu}{\mu} \biggr) =\bar\lambda \biggl(
\frac{\bar\lambda\bar\mu}{\mu} -\lambda \biggr)>0.
\]

Thus, since $\Delta_2 \neq0$, the system of equations~\eqref
{eq:80}--\eqref{eq:83} has a unique solution.
Furthermore, note that letting $x=b$ in~\eqref{eq:76} gives no
additional information about unknown constants, but the equality
in~\eqref{eq:76} holds for the values of the constants found from the
system of equations~\eqref{eq:80}--\eqref{eq:83}.
Therefore, each of differential equations~\eqref{eq:70} and~\eqref
{eq:71} has the unique solution given by~\eqref{eq:78} or~\eqref{eq:79},
respectively.
Since we have derived these equations from~\eqref{eq:68} and~\eqref
{eq:69} without any additional assumptions, we conclude that
the functions $\psi_1(x)$ and $\psi_2(x)$ given by~\eqref{eq:78}
and~\eqref{eq:79} are unique solutions to~\eqref{eq:68} and~\eqref{eq:69}
satisfying the certain conditions.
This guaranties that the functions $\psi_1(x)$ and $\psi_2(x)$ we have
found coincide with the ruin probability
on the intervals $[0,b]$ and $[b,\infty)$, respectively, which
completes the proof.
\end{proof}

\subsection{The expected discounted dividend payments until ruin in the
model without dependence}
\label{sec:5.3}

We now also assume that $\theta=0$ and $\bar\theta=0$.
Then equations~\eqref{eq:29} and~\eqref{eq:42} for the expected
discounted dividend payments $v(x)$
in the case of exponentially distributed claim and premium sizes take
the form
\begin{equation}
\label{eq:86} %
\begin{split} &(\lambda+\bar\lambda+
\delta)v_1(x) \\
&\quad=\frac{\lambda}{\mu}\, e^{-x/\mu} \int
_0^x v_1(u) e^{u/\mu}\,
\mathrm{d}u
+\frac{\bar\lambda}{\bar\mu}\, e^{x/\bar\mu} \int_x^b
v_1(u) e^{-u/\bar\mu}\, \mathrm{d}u \\
&\qquad+\frac{\bar\lambda}{\bar\mu}\,
e^{x/\bar\mu} \int_b^{\infty} v_2(u)
e^{-u/\bar\mu}\, \mathrm{d}u, \quad x\in[0,b], \end{split} %
\end{equation}
and
\begin{equation}
\label{eq:87} %
\begin{split} &d v'_2(x) +(
\lambda+\bar\lambda+\delta) v_2(x) \\
&\quad=\frac{\lambda}{\mu
}\,
e^{-x/\mu} \int_0^b v_1(u)
e^{u/\mu}\, \mathrm{d}u
+\frac{\lambda}{\mu}\, e^{-x/\mu} \int_b^x
v_2(u) e^{u/\mu}\, \mathrm{d}u\\
&\qquad +\frac{\bar\lambda}{\bar\mu}\,
e^{x/\bar\mu} \int_x^{\infty} v_2(u)
e^{-u/\bar\mu}\, \mathrm{d}u +d, \;\; x\in[b,\infty), \end{split} %
\end{equation}
respectively.

Lemma~\ref{lem:3} below shows that integro-differential equations~\eqref
{eq:86} and~\eqref{eq:87} can be reduced to
linear differential equations with constant coefficients.

\begin{lemma}
\label{lem:3}
Let the surplus process $ (X_t^b(x) )_{t\ge0}$ follow~\eqref
{eq:6} under the above assumptions with $\theta=0$ and $\bar\theta=0$,
and let claim and premium sizes be exponentially distributed with means
$\mu$ and $\bar\mu$, respectively.
Then $v_1(x)$ and $v_2(x)$ are solutions to the differential equations
\begin{equation}
\label{eq:88} \mu\bar\mu(\lambda+\bar\lambda+\delta) v''_1(x)
+ \bigl( \bar\mu(\bar \lambda+\delta) -\mu(\lambda+\delta) \bigr)
v'_1(x) -\delta v_1(x)=0, \quad x
\in[0,b],
\end{equation}
and
\begin{equation}
\label{eq:89} %
\begin{split} &d\mu\bar\mu v'''_2(x)
+ \bigl( d(\bar\mu-\mu) +\mu\bar\mu(\lambda +\bar\lambda+\delta) \bigr)
v''_2(x)
\\
&\quad+ \bigl( \bar\mu(\bar\lambda+\delta) -\mu(\lambda+\delta) -d \bigr)
v'_2(x) -\delta v_2(x)=-d, \quad x\in[b,
\infty). \end{split} %
\end{equation}
\end{lemma}

The proof of the lemma is similar to the proof of Lemma~\ref{lem:2}.

To formulate the next theorem, we define the following constants:
\begin{align*}
\mathrm{D}_3&= \bigl( \bar\mu(\bar\lambda+\delta) -\mu(\lambda+
\delta) \bigr)^2 +4\delta\mu\bar\mu(\lambda+\bar\lambda+\delta),\\
%\begin{equation*}
%
\mathrm{D}_4&=-18\delta d\mu\bar\mu
\bigl( d(\bar\mu-\mu) +\mu\bar\mu (\lambda+\bar\lambda+\delta) \bigr) \bigl( \bar
\mu(\bar\lambda+\delta) -\mu(\lambda+\delta) -d \bigr)
\\
&\quad+4\delta \bigl( d(\bar\mu-\mu) +\mu\bar\mu(\lambda+\bar \lambda+\delta)
\bigr)^3
\\
&\quad+ \bigl( d(\bar\mu-\mu) +\mu\bar\mu(\lambda+\bar\lambda+\delta )
\bigr)^2 \bigl( \bar\mu(\bar\lambda+\delta) -\mu(\lambda+\delta) -d
\bigr)^2
\\
&\quad-4d\mu\bar\mu \bigl( \bar\mu(\bar\lambda+\delta) -\mu(\lambda +\delta) -d
\bigr)^3 -27 (\delta d\mu\bar\mu)^2,\\ %\end{split}
%\end{equation*}
%\[
z_9&=\frac{- ( \bar\mu(\bar\lambda+\delta) -\mu(\lambda+\delta)
 ) +\sqrt{\mathrm{D}_3}}{2\mu\bar\mu(\lambda+\bar\lambda+\delta)}
\end{align*}
and
\[
z_{10}=\frac{- ( \bar\mu(\bar\lambda+\delta) -\mu(\lambda+\delta)
 ) -\sqrt{\mathrm{D}_3}}{2\mu\bar\mu(\lambda+\bar\lambda+\delta)}.
\]

\begin{thm}
\label{thm:5}
Let the surplus process $ (X_t^b(x) )_{t\ge0}$ follow~\eqref
{eq:6} under the above assumptions with $\theta=0$ and $\bar\theta=0$,
and let claim and premium sizes be exponentially distributed with means
$\mu$ and $\bar\mu$, respectively, and let
$\bar\lambda\bar\mu>\lambda\mu+d$ and $\mathrm{D}_4>0$.
Then we have
\begin{equation}
\label{eq:90} v_1(x)=C_9 e^{z_9 x}
+C_{10} e^{z_{10} x}, \quad x\in[0,b],
\end{equation}
and
\begin{equation}
\label{eq:91} v_2(x)=C_{11} e^{z_{11} x}
+C_{12} e^{z_{12} x} +d/\delta, \quad x\in [b,\infty),
\end{equation}
where $z_{11}$ and $z_{12}$ are negative roots of the cubic equation
\begin{equation}
\label{eq:92} d\mu\bar\mu z^3 + \bigl( d(\bar\mu-\mu) +\mu\bar\mu(
\lambda+\bar \lambda+\delta) \bigr) z^2 + \bigl( \bar\mu(\bar\lambda+
\delta) -\mu(\lambda+\delta) -d \bigr) z -\delta=0
\end{equation}
and the constants $C_9$, $C_{10}$, $C_{11}$ and $C_{12}$ are determined
from the system of linear equations~\eqref{eq:93}--\eqref{eq:96}:
\begin{equation}
\label{eq:93} %
\begin{split} & \biggl( (\lambda+\bar\lambda+\delta)
e^{b/\bar\mu} -\frac{\bar\lambda
}{\bar\mu z_9 -1} \bigl(e^{z_9 b} -e^{b/\bar\mu}
\bigr) \biggr) C_9
\\
&\quad+ \biggl( (\lambda+\bar\lambda+\delta) e^{b/\bar\mu} -
\frac{\bar
\lambda}{\bar\mu z_{10} -1} \bigl(e^{z_{10} b} -e^{b/\bar\mu}\bigr) \biggr)
C_{10}
\\
&\quad+\frac{\bar\lambda e^{z_{11} b}}{\bar\mu z_{11} -1}\, C_{11} +\frac{\bar\lambda e^{z_{12} b}}{\bar\mu z_{12} -1}\,
C_{12} =\frac{d\bar\lambda}{\delta}, \end{split} %
\end{equation}
\begin{equation}
\label{eq:94} %
\begin{split} & \biggl( \lambda+\delta+
\frac{\lambda\bar\mu}{\mu} -\bar\mu z_9 (\lambda+\bar\lambda+\delta)
\biggr) C_9
\\
&\quad+ \biggl( \lambda+\delta+\frac{\lambda\bar\mu}{\mu} -\bar\mu z_{10} (
\lambda+\bar\lambda+\delta) \biggr) C_{10}=0, \end{split} %
\end{equation}
\begin{equation}
\label{eq:95} e^{z_9 b} C_9 +e^{z_{10} b}
C_{10} -e^{z_{11} b} C_{11} -e^{z_{12} b}
C_{12} =d/\delta
\end{equation}
and
\begin{equation}
\label{eq:96} %
\begin{split} &\frac{\lambda}{\mu z_9 +1}
\bigl(e^{-b/\mu}-e^{z_9 b} \bigr) C_9 +
\frac
{\lambda}{\mu z_{10} +1} \bigl( e^{-b/\mu} -e^{z_{10} b} \bigr)
C_{10}
\\
&\quad+ \biggl( \lambda+\bar\lambda+\delta+dz_{11} +
\frac{\bar\lambda
}{\bar\mu z_{11} -1} \biggr) e^{z_{11} b} C_{11}
\\
&\quad+ \biggl( \lambda+\bar\lambda+\delta+dz_{12} +
\frac{\bar\lambda
}{\bar\mu z_{12} -1} \biggr) e^{z_{12} b} C_{12} =-d \biggl(1+
\frac{\lambda}{\delta} \biggr) \end{split} %
\end{equation}
provided that its determinant is not equal to 0.
\end{thm}

\begin{proof}
The proof is similar to the proof of Theorem~\ref{thm:4}, so we omit
detailed considerations.
By Lemma~\ref{lem:3}, $v_1(x)$ and $v_2(x)$ are solutions to~\eqref
{eq:88} and~\eqref{eq:89}.

It is easily seen that $\mathrm{D}_3>0$. Hence the characteristic
equation corresponding to~\eqref{eq:88} has two real roots
$z_9$ and $z_{10}$ given before the assertion of the theorem. This
yields~\eqref{eq:90} with some constants $C_9$ and $C_{10}$.

The assumption $\mathrm{D}_4>0$ guarantees that cubic equation~\eqref
{eq:92} has three distinct real roots.
Consequently, the general solution to~\eqref{eq:89} is given by
\[
v_2(x)=C_{11} e^{z_{11} x} +C_{12}
e^{z_{12} x} +C_{13} e^{z_{13} x}, \quad x\in[b,\infty)
\]
with some constants $C_{11}$, $C_{12}$ and $C_{13}$.

By Vieta's theorem, we conclude that~\eqref{eq:92} has either two or no
negative roots.
Since $\bar\lambda\bar\mu>\lambda\mu+d$, applying arguments similar to
those in \cite[p.~70]{Sc2008} shows that
$\lim_{x\to\infty} v_2(x) =d/{\delta}$.
Therefore, if~\eqref{eq:92} had no negative roots, the function
$v_2(x)$ would be constant, which is impossible.
From this we deduce that~\eqref{eq:92} has two negative roots. We
denote them by $z_{11}$ and $z_{12}$.
Since $z_{13}>0$, we get $C_{13}=0$, which yields~\eqref{eq:91}.

To determine the constants $C_9$, $C_{10}$, $C_{11}$ and $C_{12}$, we
apply considerations similar to those in the proof
of Theorem~\ref{thm:4} and obtain the system of linear equations~\eqref
{eq:93}--\eqref{eq:96}, which has a unique solution
provided that its determinant is not equal to 0.
Finally, applying arguments similar to those in the proof of
Theorem~\ref{thm:4} guaranties that the functions $v_1(x)$ and $v_2(x)$
we have found coincide with the expected discounted dividend payments
on the intervals $[0,b]$ and $[b,\infty)$.
\end{proof}

\section{Numerical illustrations}
\label{sec:6}

We now present numerical examples for the results obtained in
Section~\ref{sec:5}.
The claim and premium sizes are also assumed to be exponentially distributed.
Set $\lambda=0.1$, $\bar\lambda=2.3$, $\mu=3$ and $\bar\mu=0.2$.

Let now the conditions of Theorem~\ref{thm:3} hold. Then applying this
theorem we can calculate the ruin probability for $x\in[0,\infty)$
in the model without dividend payments for different values of $\theta$:
\begin{itemize}
\item if $\theta=-0.9$, then $\psi(x) \approx0.929934 e^{-0.022277 x}
-0.006234 e^{-0.744001 x}$;
\item if $\theta=-0.5$, then $\psi(x) \approx0.817753 e^{-0.059151 x}
-0.009736 e^{-0.712238 x}$;
\item if $\theta=-0.1$, then $\psi(x) \approx0.698198 e^{-0.100061 x}
-0.003545 e^{-0.676439 x}$;
\item if $\theta=0.1$, then $\psi(x) \approx0.634275 e^{-0.122565 x}
+0.004545 e^{-0.656490 x}$;
\item if $\theta=0.5$, then $\psi(x) \approx0.492433 e^{-0.173655 x}
+0.036374 e^{-0.610511 x}$;
\item if $\theta=0.9$, then $\psi(x) \approx0.309485 e^{-0.239185 x}
+0.111461 e^{-0.550092 x}$.
\end{itemize}

Table~\ref{table:1} presents the results of calculations for some
values of $x$.

\begin{table}
\caption{The ruin probabilities in the model without dividend payments
for different values of~$\theta$}\label{table:1}
\begin{tabular*}{\textwidth}{@{\extracolsep{\fill}}c
D{.}{.}{1.6}D{.}{.}{1.6}D{.}{.}{1.6}D{.}{.}{1.6}D{.}{.}{1.6}D{.}{.}{1.6}@{}}
\hline
$x\; \backslash\;\theta$ & \multicolumn{1}{c}{-0.9} & \multicolumn
{1}{c}{-0.5} &
\multicolumn{1}{c}{-0.1} & \multicolumn{1}{c}{0.1} & \multicolumn
{1}{c}{0.5} & \multicolumn{1}{c}{0.9} \\
\hline
0 &0.923700 &0.808017 &0.694653 &0.638820 &0.528807 &0.420945\\
1 &0.906484 &0.766009 &0.629915 &0.563467 &0.433686 &0.307949\\
2 &0.888003 &0.724172 &0.570650 &0.497607 &0.358674 &0.228911\\
5 &0.831762 &0.608107 &0.423229 &0.343831 &0.208380 &0.100718\\
7 &0.795628 &0.540438 &0.346536 &0.268996 &0.146531 &0.060380\\
10 &0.744222 &0.452611 &0.256692 &0.186208 &0.086812 &0.028761\\
15 &0.665780 &0.336734 &0.155646 &0.100887 &0.036402 &0.008589\\
20 &0.595603 &0.250520 &0.094376 &0.054662 &0.015276 &0.002591\\
50 &0.305291 &0.042479 &0.004690 &0.001383 &0.000083 &0.000002\\
70 &0.195532 &0.013013 &0.000634 &0.000119 &0.000003 &0.000000\\
\hline
\end{tabular*}
\end{table}

Next, we denote by $\psi_0(x)$ the ruin probability in this model where
$\theta=0$. It is given by
\[
\psi_0(x)=\frac{\lambda(\mu+\bar\mu)}{\bar\mu(\lambda+\bar\lambda)}\, \exp \biggl( -\frac{(\bar\lambda\bar\mu-\lambda\mu)x}{\mu\bar\mu(\lambda
\bar\lambda)}
\biggr), \quad x\in[0,\infty)
\]
(see \cite{Boi2003,MiRa2016}). In our example, $\psi_0(x) \approx
0.666667 e^{-0.111111 x}$. The values of $\psi_0(x)$
for some $x$ are given in Table~\ref{table:2}.

Let now the conditions of Theorems~\ref{thm:4} and~\ref{thm:5} hold.
Set additionally $b=5$, $d=0.1$ and $\delta=0.01$.
Applying Theorems~\ref{thm:4} and~\ref{thm:5} we can calculate the ruin
probability $\psi(x)$ and
the expected discounted dividend payments until ruin $v(x)$:
\begin{align*}
\psi_1(x) &\approx0.389315 +0.407125 e^{-0.111111 x}, \quad x
\in[0,5],\\
\psi_2(x) &\approx0.809486 e^{-0.051863 x} -1.24332
\cdot10^{39} e^{-19.281470 x}, \quad x\in[5,\infty);\\
v_1(x) &\approx4.555889 e^{0.049220 x} -2.296416
e^{-0.140506 x}, \quad x\in[0,5],\\
v_2(x) &\approx10 -9.149114 e^{-0.107684 x} +4.07834
\cdot10^{40} e^{-19.405407 x}, \quad x\in[5,\infty).
\end{align*}

The results of calculations for some values of $x$ are given in
Table~\ref{table:2}.

\begin{table}
\caption{The ruin probabilities with and without dividend payments and
the expected discounted dividend payments
in the model without dependence}\label{table:2}
\begin{tabular*}{6cm}{@{\extracolsep{\fill}}cD{.}{.}{1.6}D{.}{.}{1.6}D{.}{.}{1.6}@{}}
\hline
$x$ & \multicolumn{1}{c}{$\psi_0(x)$} & \multicolumn{1}{c}{$\psi(x)$} &
\multicolumn{1}{c}{$v(x)$} \\
\hline
0 &0.666667 &0.796440 &2.259472\\
1 &0.596560 &0.753626 &2.790339\\
2 &0.533825 &0.715315 &3.293343\\
5 &0.382502 &0.622904 &4.689607\\
7 &0.306284 &0.563044 &5.694612\\
10 &0.219462 &0.481915 &6.883176\\
15 &0.125917 &0.371835 &8.180807\\
20 &0.072245 &0.286900 &8.938193\\
50 &0.002577 &0.060536 &9.958020\\
70 &0.000279 &0.021455 &9.995128\\
\hline
\end{tabular*}
\end{table}

The results presented in Tables~\ref{table:1} and~\ref{table:2} show
that the positive dependence between the claim sizes
and the inter-claim times decreases the ruin probability and the
negative dependence increases it.
This conclusion seems to be natural. Indeed, in the case of negative
dependence, the situation where large claims arrive
in short time intervals is more probable, which obviously leads to ruin
in the near future.
Moreover, it is easily seen from Table~\ref{table:2} that dividend
payments substantially increase the ruin probability,
which is also an expected conclusion.

\section*{Acknowledgments}
The author is deeply grateful to the anonymous referees for careful
reading and valuable comments and suggestions,
which helped to improve the earlier version of the paper.

%\appendix

%\section{Appendix section}

%\bibliography{bib/biblio}

\begin{thebibliography}{47}

%b1 ###
\bibitem{AlBo2004}
%
\begin{barticle}
\bauthor{\bsnm{Albrecher}, \binits{H.}},
\bauthor{\bsnm{Boxma}, \binits{O.J.}}:
\batitle{A ruin model with dependence between claim sizes and claim intervals}.
\bjtitle{Insurance: Mathematics and Economics}
\bvolume{35},
\bfpage{245}--\blpage{254}
(\byear{2004})
\bid{mr={2095888}}
\end{barticle}
%
%
\OrigBibText
%
\begin{barticle}
\bauthor{\bsnm{Albrecher}, \binits{H.}},
\bauthor{\bsnm{Boxma}, \binits{O.J.}}:
\batitle{A ruin model with dependence between claim sizes and claim intervals}.
\bjtitle{Insurance: Mathematics and Economics}
\bvolume{35},
\bfpage{245}--\blpage{254}
(\byear{2004})
\end{barticle}
%
\endOrigBibText
\bptok{structpyb}%
\endbibitem

%b2 ###
\bibitem{AlBo2005}
%
\begin{barticle}
\bauthor{\bsnm{Albrecher}, \binits{H.}},
\bauthor{\bsnm{Boxma}, \binits{O.J.}}:
\batitle{On the discounted penalty function in a {M}arkov-dependent risk
model}.
\bjtitle{Insurance: Mathematics and Economics}
\bvolume{37},
\bfpage{650}--\blpage{672}
(\byear{2005})
\end{barticle}
%
%
\OrigBibText
%
\begin{barticle}
\bauthor{\bsnm{Albrecher}, \binits{H.}},
\bauthor{\bsnm{Boxma}, \binits{O.J.}}:
\batitle{On the discounted penalty function in a {M}arkov-dependent risk
model}.
\bjtitle{Insurance: Mathematics and Economics}
\bvolume{37},
\bfpage{650}--\blpage{672}
(\byear{2005})
\end{barticle}
%
\endOrigBibText
\bptok{structpyb}%
\endbibitem

%b3 ###
\bibitem{AlTe2006}
%
\begin{barticle}
\bauthor{\bsnm{Albrecher}, \binits{H.}},
\bauthor{\bsnm{Teugels}, \binits{J.L.}}:
\batitle{Exponential behavior in the presence of dependence in risk theory}.
\bjtitle{Journal of Applied Probability}
\bvolume{43},
\bfpage{257}--\blpage{273}
(\byear{2006})
\bid{mr={2225065}}
\end{barticle}
%
%
\OrigBibText
%
\begin{barticle}
\bauthor{\bsnm{Albrecher}, \binits{H.}},
\bauthor{\bsnm{Teugels}, \binits{J.L.}}:
\batitle{Exponential behavior in the presence of dependence in risk theory}.
\bjtitle{Journal of Applied Probability}
\bvolume{43},
\bfpage{257}--\blpage{273}
(\byear{2006})
\end{barticle}
%
\endOrigBibText
\bptok{structpyb}%
\endbibitem

%b4 ###
\bibitem{AlCoLo2011}
%
\begin{barticle}
\bauthor{\bsnm{Albrecher}, \binits{H.}},
\bauthor{\bsnm{Constantinescu}, \binits{C.}},
\bauthor{\bsnm{Loisel}, \binits{S.}}:
\batitle{Explicit ruin formulas for models with dependence among risks}.
\bjtitle{Insurance: Mathematics and Economics}
\bvolume{48},
\bfpage{265}--\blpage{270}
(\byear{2011})
\end{barticle}
%
%
\OrigBibText
%
\begin{barticle}
\bauthor{\bsnm{Albrecher}, \binits{H.}},
\bauthor{\bsnm{Constantinescu}, \binits{C.}},
\bauthor{\bsnm{Loisel}, \binits{S.}}:
\batitle{Explicit ruin formulas for models with dependence among risks}.
\bjtitle{Insurance: Mathematics and Economics}
\bvolume{48},
\bfpage{265}--\blpage{270}
(\byear{2011})
\end{barticle}
%
\endOrigBibText
\bptok{structpyb}%
\endbibitem

%b5 ###
\bibitem{AnBeKiSi2015}
%
\begin{barticle}
\bauthor{\bsnm{Andrulyt\.{e}}, \binits{I.M.}},
\bauthor{\bsnm{Bernackait\.{e}}, \binits{E.}},
\bauthor{\bsnm{Kievinait\.{e}}, \binits{D.}},
\bauthor{\bsnm{\v{S}iaulys}, \binits{J.}}:
\batitle{A {L}undberg-type inequality for an inhomogeneous renewal risk model}.
\bjtitle{Modern Stochastics: Theory and Applications}
\bvolume{2},
\bfpage{173}--\blpage{184}
(\byear{2015})
\bid{mr={3389589}}
\end{barticle}
%
%
\OrigBibText
%
\begin{barticle}
\bauthor{\bsnm{Andrulyt\.{e}}, \binits{I.M.}},
\bauthor{\bsnm{Bernackait\.{e}}, \binits{E.}},
\bauthor{\bsnm{Kievinait\.{e}}, \binits{D.}},
\bauthor{\bsnm{\v{S}iaulys}, \binits{J.}}:
\batitle{A {L}undberg-type inequality for an inhomogeneous renewal risk model}.
\bjtitle{Modern Stochastics: Theory and Applications}
\bvolume{2},
\bfpage{173}--\blpage{184}
(\byear{2015})
\end{barticle}
%
\endOrigBibText
\bptok{structpyb}%
\endbibitem

%b6 ###
\bibitem{AsAl2010}
%
\begin{bbook}
\bauthor{\bsnm{Asmussen}, \binits{S.}},
\bauthor{\bsnm{Albrecher}, \binits{H.}}:
\bbtitle{Ruin Probabilities}.
\bpublisher{World Scientific},
\blocation{Singapore}
(\byear{2010})
\end{bbook}
%
%
\OrigBibText
%
\begin{bbook}
\bauthor{\bsnm{Asmussen}, \binits{S.}},
\bauthor{\bsnm{Albrecher}, \binits{H.}}:
\bbtitle{Ruin Probabilities}.
\bpublisher{World Scientific},
\blocation{Singapore}
(\byear{2010})
\end{bbook}
%
\endOrigBibText
\bptok{structpyb}%
\endbibitem

%b7 ###
\bibitem{BeJa2017}
%
\begin{barticle}
\bauthor{\bsnm{Bekrizadeh}, \binits{H.}},
\bauthor{\bsnm{Jamshidi}, \binits{B.}}:
\batitle{A new class of bivariate copulas: dependence measures and properties}.
\bjtitle{METRON}
\bvolume{75},
\bfpage{31}--\blpage{50}
(\byear{2017})
\end{barticle}
%
%
\OrigBibText
%
\begin{barticle}
\bauthor{\bsnm{Bekrizadeh}, \binits{H.}},
\bauthor{\bsnm{Jamshidi}, \binits{B.}}:
\batitle{A new class of bivariate copulas: dependence measures and properties}.
\bjtitle{METRON}
\bvolume{75},
\bfpage{31}--\blpage{50}
(\byear{2017})
\end{barticle}
%
\endOrigBibText
\bptok{structpyb}%
\endbibitem

%b8 ###
\bibitem{BePaZa2012}
%
\begin{barticle}
\bauthor{\bsnm{Bekrizadeh}, \binits{H.}},
\bauthor{\bsnm{Parham}, \binits{G.A.}},
\bauthor{\bsnm{Zadkarmi}, \binits{M.R.}}:
\batitle{The new generalization of {F}arlie--{G}umbel--{M}orgenstern copulas}.
\bjtitle{Applied Mathematical Sciences}
\bvolume{6},
\bfpage{3527}--\blpage{3533}
(\byear{2012})
\bid{mr={2929551}}
\end{barticle}
%
%
\OrigBibText
%
\begin{barticle}
\bauthor{\bsnm{Bekrizadeh}, \binits{H.}},
\bauthor{\bsnm{Parham}, \binits{G.A.}},
\bauthor{\bsnm{Zadkarmi}, \binits{M.R.}}:
\batitle{The new generalization of {F}arlie--{G}umbel--{M}orgenstern copulas}.
\bjtitle{Applied Mathematical Sciences}
\bvolume{6},
\bfpage{3527}--\blpage{3533}
(\byear{2012})
\end{barticle}
%
\endOrigBibText
\bptok{structpyb}%
\endbibitem

%b9 ###
\bibitem{Boi2003}
%
\begin{barticle}
\bauthor{\bsnm{Boikov}, \binits{A.V.}}:
\batitle{The {C}ram{\'e}r--{L}undberg model with stochastic premium process}.
\bjtitle{Theory of Probability and Its Applications}
\bvolume{47},
\bfpage{489}--\blpage{493}
(\byear{2003})
\bid{mr={1975908}}
\end{barticle}
%
%
\OrigBibText
%
\begin{barticle}
\bauthor{\bsnm{Boikov}, \binits{A.V.}}:
\batitle{The {C}ram{\'e}r--{L}undberg model with stochastic premium process}.
\bjtitle{Theory of Probability and Its Applications}
\bvolume{47},
\bfpage{489}--\blpage{493}
(\byear{2003})
\end{barticle}
%
\endOrigBibText
\bptok{structpyb}%
\endbibitem

%b10 ###
\bibitem{Bou2003}
%
\begin{botherref}
\oauthor{\bsnm{Boudreault}, \binits{M.}}:
Modeling and pricing earthquake risk.
Scor Canada Actuarial Prize
(2003)
\end{botherref}
%
%
\OrigBibText
%
\begin{botherref}
\oauthor{\bsnm{Boudreault}, \binits{M.}}:
Modeling and pricing earthquake risk.
Scor Canada Actuarial Prize
(2003)
\end{botherref}
%
\endOrigBibText
\bptok{structpyb}%
\endbibitem

%b11 ###
\bibitem{BoCoLaMa2006}
%
\begin{barticle}
\bauthor{\bsnm{Boudreault}, \binits{M.}},
\bauthor{\bsnm{Cossette}, \binits{H.}},
\bauthor{\bsnm{Landriault}, \binits{D.}},
\bauthor{\bsnm{Marceau}, \binits{E.}}:
\batitle{On a risk model with dependence between interclaim arrivals
and claim
sizes}.
\bjtitle{Scandinavian Actuarial Journal}
\bvolume{2006},
\bfpage{265}--\blpage{285}
(\byear{2006})
\end{barticle}
%
%
\OrigBibText
%
\begin{barticle}
\bauthor{\bsnm{Boudreault}, \binits{M.}},
\bauthor{\bsnm{Cossette}, \binits{H.}},
\bauthor{\bsnm{Landriault}, \binits{D.}},
\bauthor{\bsnm{Marceau}, \binits{E.}}:
\batitle{On a risk model with dependence between interclaim arrivals
and claim
sizes}.
\bjtitle{Scandinavian Actuarial Journal}
\bvolume{2006},
\bfpage{265}--\blpage{285}
(\byear{2006})
\end{barticle}
%
\endOrigBibText
\bptok{structpyb}%
\endbibitem

%b12 ###
\bibitem{ChVr2014}
%
\begin{barticle}
\bauthor{\bsnm{Chadjiconstantinidis}, \binits{S.}},
\bauthor{\bsnm{Vrontos}, \binits{S.}}:
\batitle{On a renewal risk process with dependence under a
{F}arlie--{G}umbel--{M}orgenstern copula}.
\bjtitle{Scandinavian Actuarial Journal}
\bvolume{2014},
\bfpage{125}--\blpage{158}
(\byear{2014})
\end{barticle}
%
%
\OrigBibText
%
\begin{barticle}
\bauthor{\bsnm{Chadjiconstantinidis}, \binits{S.}},
\bauthor{\bsnm{Vrontos}, \binits{S.}}:
\batitle{On a renewal risk process with dependence under a
{F}arlie--{G}umbel--{M}orgenstern copula}.
\bjtitle{Scandinavian Actuarial Journal}
\bvolume{2014},
\bfpage{125}--\blpage{158}
(\byear{2014})
\end{barticle}
%
\endOrigBibText
\bptok{structpyb}%
\endbibitem

%b13 ###
\bibitem{ChTa2003}
%
\begin{barticle}
\bauthor{\bsnm{Cheng}, \binits{Y.}},
\bauthor{\bsnm{Tang}, \binits{Q.}}:
\batitle{Moments of the surplus before ruin and the deficit of ruin in the
{E}rlang(2) risk process}.
\bjtitle{North American Actuarial Journal}
\bvolume{7},
\bfpage{1}--\blpage{12}
(\byear{2003})
\end{barticle}
%
%
\OrigBibText
%
\begin{barticle}
\bauthor{\bsnm{Cheng}, \binits{Y.}},
\bauthor{\bsnm{Tang}, \binits{Q.}}:
\batitle{Moments of the surplus before ruin and the deficit of ruin in the
{E}rlang(2) risk process}.
\bjtitle{North American Actuarial Journal}
\bvolume{7},
\bfpage{1}--\blpage{12}
(\byear{2003})
\end{barticle}
%
\endOrigBibText
\bptok{structpyb}%
\endbibitem

%b14 ###
\bibitem{ChLi2011}
%
\begin{barticle}
\bauthor{\bsnm{Chi}, \binits{Y.}},
\bauthor{\bsnm{Lin}, \binits{X.S.}}:
\batitle{On the threshold dividend strategy for a generalized jump-diffusion
risk model}.
\bjtitle{Insurance: Mathematics and Economics}
\bvolume{48},
\bfpage{326}--\blpage{337}
(\byear{2011})
\end{barticle}
%
%
\OrigBibText
%
\begin{barticle}
\bauthor{\bsnm{Chi}, \binits{Y.}},
\bauthor{\bsnm{Lin}, \binits{X.S.}}:
\batitle{On the threshold dividend strategy for a generalized jump-diffusion
risk model}.
\bjtitle{Insurance: Mathematics and Economics}
\bvolume{48},
\bfpage{326}--\blpage{337}
(\byear{2011})
\end{barticle}
%
\endOrigBibText
\bptok{structpyb}%
\endbibitem

%b15 ###
\bibitem{CoMaMa2008}
%
\begin{barticle}
\bauthor{\bsnm{Cossette}, \binits{H.}},
\bauthor{\bsnm{Marceau}, \binits{E.}},
\bauthor{\bsnm{Marri}, \binits{F.}}:
\batitle{On the compound {P}oisson risk model with dependence based on a
generalized {F}arlie--{G}umbel--{M}orgenstern copula}.
\bjtitle{Insurance: Mathematics and Economics}
\bvolume{43},
\bfpage{444}--\blpage{455}
(\byear{2008})
\end{barticle}
%
%
\OrigBibText
%
\begin{barticle}
\bauthor{\bsnm{Cossette}, \binits{H.}},
\bauthor{\bsnm{Marceau}, \binits{E.}},
\bauthor{\bsnm{Marri}, \binits{F.}}:
\batitle{On the compound {P}oisson risk model with dependence based on a
generalized {F}arlie--{G}umbel--{M}orgenstern copula}.
\bjtitle{Insurance: Mathematics and Economics}
\bvolume{43},
\bfpage{444}--\blpage{455}
(\byear{2008})
\end{barticle}
%
\endOrigBibText
\bptok{structpyb}%
\endbibitem

%b16 ###
\bibitem{CoMaMa2010}
%
\begin{barticle}
\bauthor{\bsnm{Cossette}, \binits{H.}},
\bauthor{\bsnm{Marceau}, \binits{E.}},
\bauthor{\bsnm{Marri}, \binits{F.}}:
\batitle{Analysis of ruin measures for the classical compound {P}oisson risk
model with dependence}.
\bjtitle{Scandinavian Actuarial Journal}
\bvolume{2010},
\bfpage{221}--\blpage{245}
(\byear{2010})
\end{barticle}
%
%
\OrigBibText
%
\begin{barticle}
\bauthor{\bsnm{Cossette}, \binits{H.}},
\bauthor{\bsnm{Marceau}, \binits{E.}},
\bauthor{\bsnm{Marri}, \binits{F.}}:
\batitle{Analysis of ruin measures for the classical compound {P}oisson risk
model with dependence}.
\bjtitle{Scandinavian Actuarial Journal}
\bvolume{2010},
\bfpage{221}--\blpage{245}
(\byear{2010})
\end{barticle}
%
\endOrigBibText
\bptok{structpyb}%
\endbibitem

%b17 ###
\bibitem{CoMaMa2011}
%
\begin{barticle}
\bauthor{\bsnm{Cossette}, \binits{H.}},
\bauthor{\bsnm{Marceau}, \binits{E.}},
\bauthor{\bsnm{Marri}, \binits{F.}}:
\batitle{Constant dividend barrier in a risk model with ageneralized
{F}arlie--{G}umbel--{M}orgenstern copula}.
\bjtitle{Methodology and Computing in Applied Probability}
\bvolume{13},
\bfpage{487}--\blpage{510}
(\byear{2011})
\bid{mr={2822392}}
\end{barticle}
%
%
\OrigBibText
%
\begin{barticle}
\bauthor{\bsnm{Cossette}, \binits{H.}},
\bauthor{\bsnm{Marceau}, \binits{E.}},
\bauthor{\bsnm{Marri}, \binits{F.}}:
\batitle{Constant dividend barrier in a risk model with ageneralized
{F}arlie--{G}umbel--{M}orgenstern copula}.
\bjtitle{Methodology and Computing in Applied Probability}
\bvolume{13},
\bfpage{487}--\blpage{510}
(\byear{2011})
\end{barticle}
%
\endOrigBibText
\bptok{structpyb}%
\endbibitem

%b18 ###
\bibitem{CoMaMa2014}
%
\begin{barticle}
\bauthor{\bsnm{Cossette}, \binits{H.}},
\bauthor{\bsnm{Marceau}, \binits{E.}},
\bauthor{\bsnm{Marri}, \binits{F.}}:
\batitle{On a compound {P}oisson risk model with dependence and in the presence
of a constant dividend barrier}.
\bjtitle{Applied Stochastic Models in Business and Industry}
\bvolume{30},
\bfpage{82}--\blpage{98}
(\byear{2014})
\end{barticle}
%
%
\OrigBibText
%
\begin{barticle}
\bauthor{\bsnm{Cossette}, \binits{H.}},
\bauthor{\bsnm{Marceau}, \binits{E.}},
\bauthor{\bsnm{Marri}, \binits{F.}}:
\batitle{On a compound {P}oisson risk model with dependence and in the presence
of a constant dividend barrier}.
\bjtitle{Applied Stochastic Models in Business and Industry}
\bvolume{30},
\bfpage{82}--\blpage{98}
(\byear{2014})
\end{barticle}
%
\endOrigBibText
\bptok{structpyb}%
\endbibitem

%b19 ###
\bibitem{CzKaBrMi2012}
%
\begin{barticle}
\bauthor{\bsnm{Czado}, \binits{C.}},
\bauthor{\bsnm{Kastenmeier}, \binits{R.}},
\bauthor{\bsnm{Brechmann}, \binits{E.C.}},
\bauthor{\bsnm{Min}, \binits{A.}}:
\batitle{A mixed copula model for insurance claims and claim sizes}.
\bjtitle{Scandinavian Actuarial Journal}
\bvolume{2012},
\bfpage{278}--\blpage{305}
(\byear{2012})
\end{barticle}
%
%
\OrigBibText
%
\begin{barticle}
\bauthor{\bsnm{Czado}, \binits{C.}},
\bauthor{\bsnm{Kastenmeier}, \binits{R.}},
\bauthor{\bsnm{Brechmann}, \binits{E.C.}},
\bauthor{\bsnm{Min}, \binits{A.}}:
\batitle{A mixed copula model for insurance claims and claim sizes}.
\bjtitle{Scandinavian Actuarial Journal}
\bvolume{2012},
\bfpage{278}--\blpage{305}
(\byear{2012})
\end{barticle}
%
\endOrigBibText
\bptok{structpyb}%
\endbibitem

%b20 ###
\bibitem{De1957}
%
\begin{barticle}
\bauthor{\bsnm{De~Finetti}, \binits{B.}}:
\batitle{Su un'impostazione alternativa dell teoria colletiva del rischio}.
\bjtitle{Transactions of the XV International Congress of Actuaries}
\bvolume{2},
\bfpage{433}--\blpage{443}
(\byear{1957})
\end{barticle}
%
%
\OrigBibText
%
\begin{barticle}
\bauthor{\bsnm{De~Finetti}, \binits{B.}}:
\batitle{Su un'impostazione alternativa dell teoria colletiva del rischio}.
\bjtitle{Transactions of the XV International Congress of Actuaries}
\bvolume{2},
\bfpage{433}--\blpage{443}
(\byear{1957})
\end{barticle}
%
\endOrigBibText
\bptok{structpyb}%
\endbibitem

%b21 ###
\bibitem{DeDhGoKa2005}
%
\begin{bbook}
\bauthor{\bsnm{Denuit}, \binits{M.}},
\bauthor{\bsnm{Dhaene}, \binits{J.}},
\bauthor{\bsnm{Goovaerts}, \binits{M.}},
\bauthor{\bsnm{Kaas}, \binits{R.}}:
\bbtitle{Actuarial Theory for Dependent Risks: Measures, Orders and Models}.
\bpublisher{John Wiley \& Sons},
\blocation{Chichester}
(\byear{2005})
\end{bbook}
%
%
\OrigBibText
%
\begin{bbook}
\bauthor{\bsnm{Denuit}, \binits{M.}},
\bauthor{\bsnm{Dhaene}, \binits{J.}},
\bauthor{\bsnm{Goovaerts}, \binits{M.}},
\bauthor{\bsnm{Kaas}, \binits{R.}}:
\bbtitle{Actuarial Theory for Dependent Risks: Measures, Orders and Models}.
\bpublisher{John Wiley \& Sons},
\blocation{Chichester}
(\byear{2005})
\end{bbook}
%
\endOrigBibText
\bptok{structpyb}%
\endbibitem

%b22 ###
\bibitem{GeSh1998}
%
\begin{barticle}
\bauthor{\bsnm{Gerber}, \binits{H.U.}},
\bauthor{\bsnm{Shiu}, \binits{E.S.W.}}:
\batitle{On the time value of ruin}.
\bjtitle{North American Actuarial Journal}
\bvolume{2},
\bfpage{48}--\blpage{72}
(\byear{1998})
\end{barticle}
%
%
\OrigBibText
%
\begin{barticle}
\bauthor{\bsnm{Gerber}, \binits{H.U.}},
\bauthor{\bsnm{Shiu}, \binits{E.S.W.}}:
\batitle{On the time value of ruin}.
\bjtitle{North American Actuarial Journal}
\bvolume{2},
\bfpage{48}--\blpage{72}
(\byear{1998})
\end{barticle}
%
\endOrigBibText
\bptok{structpyb}%
\endbibitem

%b23 ###
\bibitem{GeSh2005}
%
\begin{barticle}
\bauthor{\bsnm{Gerber}, \binits{H.U.}},
\bauthor{\bsnm{Shiu}, \binits{E.S.W.}}:
\batitle{The time value of ruin in a {S}parre {A}ndersen model}.
\bjtitle{North American Actuarial Journal}
\bvolume{9},
\bfpage{49}--\blpage{69}
(\byear{2005})
\end{barticle}
%
%
\OrigBibText
%
\begin{barticle}
\bauthor{\bsnm{Gerber}, \binits{H.U.}},
\bauthor{\bsnm{Shiu}, \binits{E.S.W.}}:
\batitle{The time value of ruin in a {S}parre {A}ndersen model}.
\bjtitle{North American Actuarial Journal}
\bvolume{9},
\bfpage{49}--\blpage{69}
(\byear{2005})
\end{barticle}
%
\endOrigBibText
\bptok{structpyb}%
\endbibitem

%b24 ###
\bibitem{Gu2013}
%
\begin{barticle}
\bauthor{\bsnm{Gu}, \binits{C.}}:
\batitle{The ruin problem of dependent risk model based on copula function}.
\bjtitle{Journal of Chemical and Pharmaceutical Research}
\bvolume{5},
\bfpage{234}--\blpage{240}
(\byear{2013})
\end{barticle}
%
%
\OrigBibText
%
\begin{barticle}
\bauthor{\bsnm{Gu}, \binits{C.}}:
\batitle{The ruin problem of dependent risk model based on copula function}.
\bjtitle{Journal of Chemical and Pharmaceutical Research}
\bvolume{5},
\bfpage{234}--\blpage{240}
(\byear{2013})
\end{barticle}
%
\endOrigBibText
\bptok{structpyb}%
\endbibitem

%b25 ###
\bibitem{He2014}
%
\begin{barticle}
\bauthor{\bsnm{Heilpern}, \binits{S.}}:
\batitle{Ruin measures for a compound {P}oisson risk model with dependence
based on the {S}pearman copula and the exponential claim sizes}.
\bjtitle{Insurance: Mathematics and Economics}
\bvolume{59},
\bfpage{251}--\blpage{257}
(\byear{2014})
\end{barticle}
%
%
\OrigBibText
%
\begin{barticle}
\bauthor{\bsnm{Heilpern}, \binits{S.}}:
\batitle{Ruin measures for a compound {P}oisson risk model with dependence
based on the {S}pearman copula and the exponential claim sizes}.
\bjtitle{Insurance: Mathematics and Economics}
\bvolume{59},
\bfpage{251}--\blpage{257}
(\byear{2014})
\end{barticle}
%
\endOrigBibText
\bptok{structpyb}%
\endbibitem

%b26 ###
\bibitem{La2008}
%
\begin{barticle}
\bauthor{\bsnm{Landriault}, \binits{D.}}:
\batitle{Constant dividend barrier in a risk model with interclaim-dependent
claim sizes}.
\bjtitle{Insurance: Mathematics and Economics}
\bvolume{42},
\bfpage{31}--\blpage{38}
(\byear{2008})
\end{barticle}
%
%
\OrigBibText
%
\begin{barticle}
\bauthor{\bsnm{Landriault}, \binits{D.}}:
\batitle{Constant dividend barrier in a risk model with interclaim-dependent
claim sizes}.
\bjtitle{Insurance: Mathematics and Economics}
\bvolume{42},
\bfpage{31}--\blpage{38}
(\byear{2008})
\end{barticle}
%
\endOrigBibText
\bptok{structpyb}%
\endbibitem

%b27 ###
\bibitem{LiWuSo2009}
%
\begin{barticle}
\bauthor{\bsnm{Li}, \binits{B.}},
\bauthor{\bsnm{Wu}, \binits{R.}},
\bauthor{\bsnm{Song}, \binits{M.}}:
\batitle{A renewal jump-diffusion process with threshold dividend strategy}.
\bjtitle{Journal of Computational and Applied Mathematics}
\bvolume{228},
\bfpage{41}--\blpage{55}
(\byear{2009})
\end{barticle}
%
%
\OrigBibText
%
\begin{barticle}
\bauthor{\bsnm{Li}, \binits{B.}},
\bauthor{\bsnm{Wu}, \binits{R.}},
\bauthor{\bsnm{Song}, \binits{M.}}:
\batitle{A renewal jump-diffusion process with threshold dividend strategy}.
\bjtitle{Journal of Computational and Applied Mathematics}
\bvolume{228},
\bfpage{41}--\blpage{55}
(\byear{2009})
\end{barticle}
%
\endOrigBibText
\bptok{structpyb}%
\endbibitem

%b28 ###
\bibitem{LiGa2004_1}
%
\begin{barticle}
\bauthor{\bsnm{Li}, \binits{S.}},
\bauthor{\bsnm{Garrido}, \binits{J.}}:
\batitle{On a class of renewal risk models with a constant dividend barrier}.
\bjtitle{Insurance: Mathematics and Economics}
\bvolume{35},
\bfpage{691}--\blpage{701}
(\byear{2004})
\end{barticle}
%
%
\OrigBibText
%
\begin{barticle}
\bauthor{\bsnm{Li}, \binits{S.}},
\bauthor{\bsnm{Garrido}, \binits{J.}}:
\batitle{On a class of renewal risk models with a constant dividend barrier}.
\bjtitle{Insurance: Mathematics and Economics}
\bvolume{35},
\bfpage{691}--\blpage{701}
(\byear{2004})
\end{barticle}
%
\endOrigBibText
\bptok{structpyb}%
\endbibitem

%b29 ###
\bibitem{LiGa2004_2}
%
\begin{barticle}
\bauthor{\bsnm{Li}, \binits{S.}},
\bauthor{\bsnm{Garrido}, \binits{J.}}:
\batitle{On ruin for the {E}rlang(n) risk process}.
\bjtitle{Insurance: Mathematics and Economics}
\bvolume{34},
\bfpage{391}--\blpage{408}
(\byear{2004})
\end{barticle}
%
%
\OrigBibText
%
\begin{barticle}
\bauthor{\bsnm{Li}, \binits{S.}},
\bauthor{\bsnm{Garrido}, \binits{J.}}:
\batitle{On ruin for the {E}rlang(n) risk process}.
\bjtitle{Insurance: Mathematics and Economics}
\bvolume{34},
\bfpage{391}--\blpage{408}
(\byear{2004})
\end{barticle}
%
\endOrigBibText
\bptok{structpyb}%
\endbibitem

%b30 ###
\bibitem{LiPa2006}
%
\begin{barticle}
\bauthor{\bsnm{Lin}, \binits{X.S.}},
\bauthor{\bsnm{Pavlova}, \binits{K.P.}}:
\batitle{The compound {P}oisson risk model with a threshold dividend strategy}.
\bjtitle{Insurance: Mathematics and Economics}
\bvolume{38},
\bfpage{57}--\blpage{80}
(\byear{2006})
\end{barticle}
%
%
\OrigBibText
%
\begin{barticle}
\bauthor{\bsnm{Lin}, \binits{X.S.}},
\bauthor{\bsnm{Pavlova}, \binits{K.P.}}:
\batitle{The compound {P}oisson risk model with a threshold dividend strategy}.
\bjtitle{Insurance: Mathematics and Economics}
\bvolume{38},
\bfpage{57}--\blpage{80}
(\byear{2006})
\end{barticle}
%
\endOrigBibText
\bptok{structpyb}%
\endbibitem

%b31 ###
\bibitem{LiWiDr2003}
%
\begin{barticle}
\bauthor{\bsnm{Lin}, \binits{X.S.}},
\bauthor{\bsnm{Willmot}, \binits{G.E.}},
\bauthor{\bsnm{Drekic}, \binits{S.}}:
\batitle{The classical risk model with a constant dividend barrier:
analysis of
the {G}erber--{S}hiu discounted penalty function}.
\bjtitle{Insurance: Mathematics and Economics}
\bvolume{33},
\bfpage{551}--\blpage{566}
(\byear{2003})
\end{barticle}
%
%
\OrigBibText
%
\begin{barticle}
\bauthor{\bsnm{Lin}, \binits{X.S.}},
\bauthor{\bsnm{Willmot}, \binits{G.E.}},
\bauthor{\bsnm{Drekic}, \binits{S.}}:
\batitle{The classical risk model with a constant dividend barrier:
analysis of
the {G}erber--{S}hiu discounted penalty function}.
\bjtitle{Insurance: Mathematics and Economics}
\bvolume{33},
\bfpage{551}--\blpage{566}
(\byear{2003})
\end{barticle}
%
\endOrigBibText
\bptok{structpyb}%
\endbibitem

%b32 ###
\bibitem{LiLiPe2014}
%
\begin{barticle}
\bauthor{\bsnm{Liu}, \binits{D.}},
\bauthor{\bsnm{Liu}, \binits{Z.}},
\bauthor{\bsnm{Peng}, \binits{D.}}:
\batitle{The {G}erber-{S}hiu expected penalty function for the risk
model with
dependence and a constand dividend barrier}.
\bjtitle{Abstract and Applied Analysis}
\bvolume{2014},
\bfpage{730174}--\blpage{7}
(\byear{2014})
\bid{mr={3246356}}
\end{barticle}
%
%
\OrigBibText
%
\begin{barticle}
\bauthor{\bsnm{Liu}, \binits{D.}},
\bauthor{\bsnm{Liu}, \binits{Z.}},
\bauthor{\bsnm{Peng}, \binits{D.}}:
\batitle{The {G}erber-{S}hiu expected penalty function for the risk
model with
dependence and a constand dividend barrier}.
\bjtitle{Abstract and Applied Analysis}
\bvolume{2014},
\bfpage{730174}--\blpage{7}
(\byear{2014})
\end{barticle}
%
\endOrigBibText
\bptok{structpyb}%
\endbibitem

%b33 ###
\bibitem{MeZhGu2008}
%
\begin{barticle}
\bauthor{\bsnm{Meng}, \binits{Q.}},
\bauthor{\bsnm{Zhang}, \binits{X.}},
\bauthor{\bsnm{Guo}, \binits{J.}}:
\batitle{On a risk model with dependence between claim sizes and claim
intervals}.
\bjtitle{Statistics and Probability Letters}
\bvolume{78},
\bfpage{1727}--\blpage{1734}
(\byear{2008})
\bid{mr={2528548}}
\end{barticle}
%
%
\OrigBibText
%
\begin{barticle}
\bauthor{\bsnm{Meng}, \binits{Q.}},
\bauthor{\bsnm{Zhang}, \binits{X.}},
\bauthor{\bsnm{Guo}, \binits{J.}}:
\batitle{On a risk model with dependence between claim sizes and claim
intervals}.
\bjtitle{Statistics and Probability Letters}
\bvolume{78},
\bfpage{1727}--\blpage{1734}
(\byear{2008})
\end{barticle}
%
\endOrigBibText
\bptok{structpyb}%
\endbibitem

%b34 ###
\bibitem{MiRa2016}
%
\begin{bbook}
\bauthor{\bsnm{Mishura}, \binits{Y.}},
\bauthor{\bsnm{Ragulina}, \binits{O.}}:
\bbtitle{Ruin Probabilities: Smoothness, Bounds, Supermartingale Approach}.
\bpublisher{ISTE Press -- Elsevier},
\blocation{London}
(\byear{2016})
\end{bbook}
%
%
\OrigBibText
%
\begin{bbook}
\bauthor{\bsnm{Mishura}, \binits{Y.}},
\bauthor{\bsnm{Ragulina}, \binits{O.}}:
\bbtitle{Ruin Probabilities: Smoothness, Bounds, Supermartingale Approach}.
\bpublisher{ISTE Press -- Elsevier},
\blocation{London}
(\byear{2016})
\end{bbook}
%
\endOrigBibText
\bptok{structpyb}%
\endbibitem

%b35 ###
\bibitem{MiRaSt2014}
%
\begin{barticle}
\bauthor{\bsnm{Mishura}, \binits{Y.}},
\bauthor{\bsnm{Ragulina}, \binits{O.}},
\bauthor{\bsnm{Stroev}, \binits{O.}}:
\batitle{Practical approaches to the estimation of the ruin probability
in a
risk model with additional funds}.
\bjtitle{Modern Stochastics: Theory and Applications}
\bvolume{1},
\bfpage{167}--\blpage{180}
(\byear{2014})
\end{barticle}
%
%
\OrigBibText
%
\begin{barticle}
\bauthor{\bsnm{Mishura}, \binits{Y.}},
\bauthor{\bsnm{Ragulina}, \binits{O.}},
\bauthor{\bsnm{Stroev}, \binits{O.}}:
\batitle{Practical approaches to the estimation of the ruin probability
in a
risk model with additional funds}.
\bjtitle{Modern Stochastics: Theory and Applications}
\bvolume{1},
\bfpage{167}--\blpage{180}
(\byear{2014})
\end{barticle}
%
\endOrigBibText
\bptok{structpyb}%
\endbibitem

%b36 ###
\bibitem{MiRaSt2015}
%
\begin{barticle}
\bauthor{\bsnm{Mishura}, \binits{Y.S.}},
\bauthor{\bsnm{Ragulina}, \binits{O.Y.}},
\bauthor{\bsnm{Stroev}, \binits{O.M.}}:
\batitle{Analytic property of infinite-horizon survival probability in
a risk
model with additional funds}.
\bjtitle{Theory of Probability and Mathematical Statistics}
\bvolume{91},
\bfpage{131}--\blpage{143}
(\byear{2015})
\end{barticle}
%
%
\OrigBibText
%
\begin{barticle}
\bauthor{\bsnm{Mishura}, \binits{Y.S.}},
\bauthor{\bsnm{Ragulina}, \binits{O.Y.}},
\bauthor{\bsnm{Stroev}, \binits{O.M.}}:
\batitle{Analytic property of infinite-horizon survival probability in
a risk
model with additional funds}.
\bjtitle{Theory of Probability and Mathematical Statistics}
\bvolume{91},
\bfpage{131}--\blpage{143}
(\byear{2015})
\end{barticle}
%
\endOrigBibText
\bptok{structpyb}%
\endbibitem

%b37 ###
\bibitem{Ne2006}
%
\begin{bbook}
\bauthor{\bsnm{Nelsen}, \binits{R.B.}}:
\bbtitle{An Introduction to Copulas}.
\bpublisher{Springer},
\blocation{New York}
(\byear{2006})
\end{bbook}
%
%
\OrigBibText
%
\begin{bbook}
\bauthor{\bsnm{Nelsen}, \binits{R.B.}}:
\bbtitle{An Introduction to Copulas}.
\bpublisher{Springer},
\blocation{New York}
(\byear{2006})
\end{bbook}
%
\endOrigBibText
\bptok{structpyb}%
\endbibitem

%b38 ###
\bibitem{NiKa2008}
%
\begin{barticle}
\bauthor{\bsnm{Nikoloulopoulos}, \binits{A.K.}},
\bauthor{\bsnm{Karlis}, \binits{D.}}:
\batitle{Fitting copulas to bivariate earthquake data: the seismic gap
hypothesis revisited}.
\bjtitle{Environmetrics}
\bvolume{19},
\bfpage{251}--\blpage{269}
(\byear{2008})
\end{barticle}
%
%
\OrigBibText
%
\begin{barticle}
\bauthor{\bsnm{Nikoloulopoulos}, \binits{A.K.}},
\bauthor{\bsnm{Karlis}, \binits{D.}}:
\batitle{Fitting copulas to bivariate earthquake data: the seismic gap
hypothesis revisited}.
\bjtitle{Environmetrics}
\bvolume{19},
\bfpage{251}--\blpage{269}
(\byear{2008})
\end{barticle}
%
\endOrigBibText
\bptok{structpyb}%
\endbibitem

%b39 ###
\bibitem{RoScScTe1999}
%
\begin{bbook}
\bauthor{\bsnm{Rolski}, \binits{T.}},
\bauthor{\bsnm{Schmidli}, \binits{H.}},
\bauthor{\bsnm{Schmidt}, \binits{V.}},
\bauthor{\bsnm{Teugels}, \binits{J.}}:
\bbtitle{Stochastic Processes for Insurance and Finance}.
\bpublisher{John Wiley \& Sons},
\blocation{Chichester}
(\byear{1999})
\end{bbook}
%
%
\OrigBibText
%
\begin{bbook}
\bauthor{\bsnm{Rolski}, \binits{T.}},
\bauthor{\bsnm{Schmidli}, \binits{H.}},
\bauthor{\bsnm{Schmidt}, \binits{V.}},
\bauthor{\bsnm{Teugels}, \binits{J.}}:
\bbtitle{Stochastic Processes for Insurance and Finance}.
\bpublisher{John Wiley \& Sons},
\blocation{Chichester}
(\byear{1999})
\end{bbook}
%
\endOrigBibText
\bptok{structpyb}%
\endbibitem

%b40 ###
\bibitem{Sc2008}
%
\begin{bbook}
\bauthor{\bsnm{Schmidli}, \binits{H.}}:
\bbtitle{Stochastic Control in Insurance}.
\bpublisher{Springer},
\blocation{London}
(\byear{2008})
\end{bbook}
%
%
\OrigBibText
%
\begin{bbook}
\bauthor{\bsnm{Schmidli}, \binits{H.}}:
\bbtitle{Stochastic Control in Insurance}.
\bpublisher{Springer},
\blocation{London}
(\byear{2008})
\end{bbook}
%
\endOrigBibText
\bptok{structpyb}%
\endbibitem

%b41 ###
\bibitem{ShLiZh2013}
%
\begin{barticle}
\bauthor{\bsnm{Shi}, \binits{Y.}},
\bauthor{\bsnm{Liu}, \binits{P.}},
\bauthor{\bsnm{Zhang}, \binits{C.}}:
\batitle{On the compound poisson risk model with dependence and a threshold
dividend strategy}.
\bjtitle{Statistics and Probability Letters}
\bvolume{83},
\bfpage{1998}--\blpage{2006}
(\byear{2013})
\bid{mr={3079035}}
\end{barticle}
%
%
\OrigBibText
%
\begin{barticle}
\bauthor{\bsnm{Shi}, \binits{Y.}},
\bauthor{\bsnm{Liu}, \binits{P.}},
\bauthor{\bsnm{Zhang}, \binits{C.}}:
\batitle{On the compound poisson risk model with dependence and a threshold
dividend strategy}.
\bjtitle{Statistics and Probability Letters}
\bvolume{83},
\bfpage{1998}--\blpage{2006}
(\byear{2013})
\end{barticle}
%
\endOrigBibText
\bptok{structpyb}%
\endbibitem

%b42 ###
\bibitem{Su2005}
%
\begin{barticle}
\bauthor{\bsnm{Sun}, \binits{L.-J.}}:
\batitle{The expected discounted penalty at ruin in the {E}rlang(2) risk
process}.
\bjtitle{Statistics and Probability Letters}
\bvolume{72},
\bfpage{205}--\blpage{217}
(\byear{2005})
\end{barticle}
%
%
\OrigBibText
%
\begin{barticle}
\bauthor{\bsnm{Sun}, \binits{L.-J.}}:
\batitle{The expected discounted penalty at ruin in the {E}rlang(2) risk
process}.
\bjtitle{Statistics and Probability Letters}
\bvolume{72},
\bfpage{205}--\blpage{217}
(\byear{2005})
\end{barticle}
%
\endOrigBibText
\bptok{structpyb}%
\endbibitem

%b43 ###
\bibitem{Wa2015}
%
\begin{barticle}
\bauthor{\bsnm{Wang}, \binits{W.}}:
\batitle{The perturbed {S}parre {A}ndersen model with interest and a threshold
dividend strategy}.
\bjtitle{Methodology and Computing in Applied Probability}
\bvolume{17},
\bfpage{251}--\blpage{283}
(\byear{2015})
\bid{mr={3343407}}
\end{barticle}
%
%
\OrigBibText
%
\begin{barticle}
\bauthor{\bsnm{Wang}, \binits{W.}}:
\batitle{The perturbed {S}parre {A}ndersen model with interest and a threshold
dividend strategy}.
\bjtitle{Methodology and Computing in Applied Probability}
\bvolume{17},
\bfpage{251}--\blpage{283}
(\byear{2015})
\end{barticle}
%
\endOrigBibText
\bptok{structpyb}%
\endbibitem

%b44 ###
\bibitem{XiWu2006}
%
\begin{barticle}
\bauthor{\bsnm{Xing}, \binits{Y.}},
\bauthor{\bsnm{Wu}, \binits{R.}}:
\batitle{Moments of the time of ruin, surplus before ruin and the
deficit at
ruin in the {E}rlang({N}) risk process}.
\bjtitle{Acta Mathematicae Applicatae Sinica, English Series}
\bvolume{22},
\bfpage{599}--\blpage{606}
(\byear{2006})
\end{barticle}
%
%
\OrigBibText
%
\begin{barticle}
\bauthor{\bsnm{Xing}, \binits{Y.}},
\bauthor{\bsnm{Wu}, \binits{R.}}:
\batitle{Moments of the time of ruin, surplus before ruin and the
deficit at
ruin in the {E}rlang({N}) risk process}.
\bjtitle{Acta Mathematicae Applicatae Sinica, English Series}
\bvolume{22},
\bfpage{599}--\blpage{606}
(\byear{2006})
\end{barticle}
%
\endOrigBibText
\bptok{structpyb}%
\endbibitem

%b45 ###
\bibitem{YaZh2008}
%
\begin{barticle}
\bauthor{\bsnm{Yang}, \binits{H.}},
\bauthor{\bsnm{Zhang}, \binits{Z.}}:
\batitle{{G}erber-{S}hiu discounted penalty function in a {S}parre {A}ndersen
model with multi-layer dividend strategy}.
\bjtitle{Insurance: Mathematics and Economics}
\bvolume{42},
\bfpage{984}--\blpage{991}
(\byear{2008})
\end{barticle}
%
%
\OrigBibText
%
\begin{barticle}
\bauthor{\bsnm{Yang}, \binits{H.}},
\bauthor{\bsnm{Zhang}, \binits{Z.}}:
\batitle{{G}erber-{S}hiu discounted penalty function in a {S}parre {A}ndersen
model with multi-layer dividend strategy}.
\bjtitle{Insurance: Mathematics and Economics}
\bvolume{42},
\bfpage{984}--\blpage{991}
(\byear{2008})
\end{barticle}
%
\endOrigBibText
\bptok{structpyb}%
\endbibitem

%b46 ###
\bibitem{YoXi2012}
%
\begin{barticle}
\bauthor{\bsnm{Yong}, \binits{W.}},
\bauthor{\bsnm{Xiang}, \binits{H.}}:
\batitle{Differential equations for ruin probability in a special risk model
with {FGM} copula for the claim size and the inter-claim time}.
\bjtitle{Journal of Inequalities and Applications}
\bvolume{2012},
\bfpage{156}--\blpage{13}
(\byear{2012})
\end{barticle}
%
%
\OrigBibText
%
\begin{barticle}
\bauthor{\bsnm{Yong}, \binits{W.}},
\bauthor{\bsnm{Xiang}, \binits{H.}}:
\batitle{Differential equations for ruin probability in a special risk model
with {FGM} copula for the claim size and the inter-claim time}.
\bjtitle{Journal of Inequalities and Applications}
\bvolume{2012},
\bfpage{156}--\blpage{13}
(\byear{2012})
\end{barticle}
%
\endOrigBibText
\bptok{structpyb}%
\endbibitem

%b47 ###
\bibitem{ZhYa2011}
%
\begin{barticle}
\bauthor{\bsnm{Zhang}, \binits{Z.}},
\bauthor{\bsnm{Yang}, \binits{H.}}:
\batitle{{G}erber--{S}hiu analysis in a perturbed risk model with dependence
between claim sizes and interclaim times}.
\bjtitle{Journal of Computational and Applied Mathematics}
\bvolume{235},
\bfpage{1189}--\blpage{1204}
(\byear{2011})
\end{barticle}
%
%
\OrigBibText
%
\begin{barticle}
\bauthor{\bsnm{Zhang}, \binits{Z.}},
\bauthor{\bsnm{Yang}, \binits{H.}}:
\batitle{{G}erber--{S}hiu analysis in a perturbed risk model with dependence
between claim sizes and interclaim times}.
\bjtitle{Journal of Computational and Applied Mathematics}
\bvolume{235},
\bfpage{1189}--\blpage{1204}
(\byear{2011})
\end{barticle}
%
\endOrigBibText
\bptok{structpyb}%
\endbibitem

\end{thebibliography}

\end{document}